\documentclass[10pt]{article}

\usepackage{amsmath,amsthm,amssymb}
\usepackage{caption,graphicx,subfig,color,paralist}
\usepackage[text={6.5in,9.25in},centering]{geometry}
\usepackage[sort,numbers]{natbib}

\setcaptionmargin{0.25in}

\makeatletter\@addtoreset{equation}{section}\makeatother

\newtheorem{thm}{Theorem}[section] 
\newtheorem{lem}[thm]{Lemma}  
\newtheorem{cor}[thm]{Corollary} 
\newtheorem{prop}[thm]{Proposition}  
\newtheorem{hyp}{Hypothesis}
\newtheorem{defn}[thm]{Definition}

\newtheoremstyle{named}{}{}{\itshape}{}{\bfseries}{.}{.5em}{\thmnote{#3's }#1}
\theoremstyle{named}

\theoremstyle{definition}
\newtheorem{rmk}{Remark}

\begin{document}

\title{Spatially localized structures in lattice dynamical systems}

\author{
Jason J. Bramburger and Bj\"orn Sandstede\\
Division of Applied Mathematics\\
Brown University\\
Providence, RI 02912, USA
}

\date{}
\maketitle

\begin{abstract}
We investigate stationary, spatially localized patterns in lattice dynamical systems that exhibit bistability. The profiles associated with these patterns have a long plateau where the pattern resembles one of the bistable states, while the profile is close to the second bistable state outside this plateau. We show that the existence branches of such patterns generically form either an infinite stack of closed loops (isolas) or intertwined s-shaped curves (snaking). We then use bifurcation theory near the anti-continuum limit, where the coupling between edges in the lattice vanishes, to prove existence of isolas and snaking in a bistable discrete real Ginzburg--Landau equation. We also provide numerical evidence for the existence of snaking diagrams for planar localized patches on square and hexagonal lattices and outline a strategy to analyse them rigorously.
\end{abstract}


\section{Introduction}\label{s1}

We are interested in patterns that form in spatially extended systems due to bistability. Imagine a system that supports two stable stationary states, say a homogeneous rest state $u_0(x)=0$ and a patterned state $u_p(x)$ that may be spatially periodic. We can then attempt to find stationary states $u(x)$ so that $u(x)\approx u_p(x)$ for $|x|<L$ and $u(x)\approx0$ for $|x|>L$; these states therefore resemble the patterned state over a domain of diameter $L$ and are close to the homogeneous rest state outside of this region; see Figure~\ref{fig:Snaking_1D} for an illustration in the spatially one-dimensional case. Localized patterns of this form have been observed in many different systems, ranging from semiconductors \cite{Semiconductor} and chemical reactions \cite{Chemical} to vegetation patterns \cite{Veg1,Veg2}, crime hot spots \cite{Hotspot,Hotspot2}, and ferrofluids \cite{Ferrofluid}; additional references can be found in the review papers \cite{Dawes,Knobloch}.

As shown in Figure~\ref{fig:Snaking_1D}, localized patterns can arrange themselves in intricate existence diagrams upon varying a systems parameter. In many cases, their spatial $L^2$-norm goes to infinity along the associated bifurcation curve -- this case is often referred to as snaking. Alternatively, it is also possible that these patterns exist along an infinite number of closed bifurcation curves, so-called isolas.

\begin{figure}
\centering
\includegraphics[width=\textwidth]{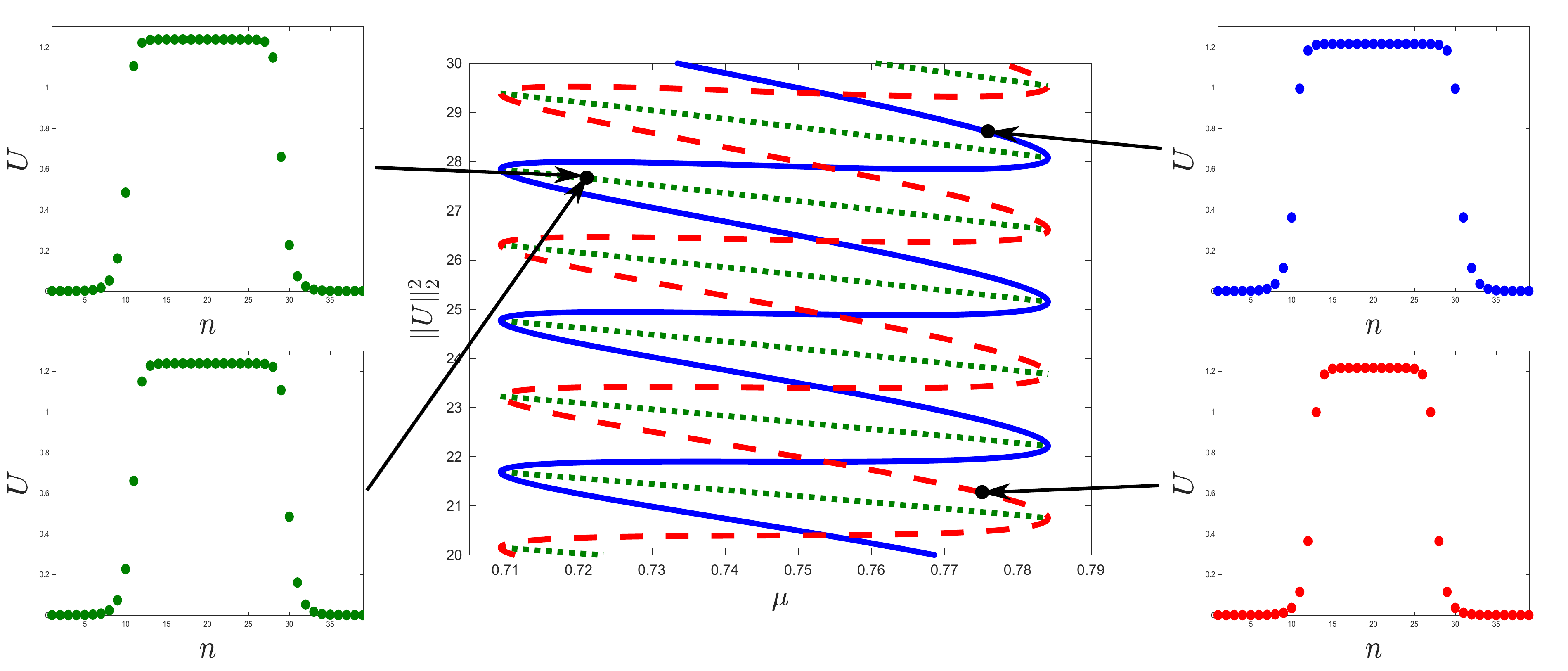}
\caption{Shown are the bifurcation diagrams of localized patterns to (\ref{LDS_Intro}) with $d=0.5$: on-site and off-site solutions exist along the dashed (red) and solid (blue) curves, respectively, while asymmetric states arise along the connecting dotted (green) branches. Representative solution profiles are shown in the insets: on-site solutions have an odd number of points along their plateaus, whereas off-site solutions have an even number.}
\label{fig:Snaking_1D}
\end{figure} 

Understanding these intriguing diagrams has been the focus of much attention over the past decades. Early investigations considered systems that admit a Lyapunov function or energy, which decreases strictly in time along non-stationary solutions, and focused on the apparent paradox that, even in the bistability regime, one of the two patterns involved should have lower energy and therefore invade the second state --- in particular, localized states of the form described above should exist only at the unique parameter value (the so-called Maxwell point) for which the two patterns have equal energy. Heuristic explanations focused on the role of spatial periodicity of the patterned state to explain why localized patterns of arbitrary extent $L$ can exist for an open parameter interval rather than just a single parameter value \cite{Pomeau,Woods,Coullet}. However, these early investigations could not predict the exact shape of bifurcation diagrams. In \cite{kozyreff1,kozyreff2}, Chapman and Kozyreff used formal asymptotics beyond all orders to show that small-amplitude localized patterns emerge near Turing bifurcations and form a snaking bifurcation diagram. The authors of \cite{Beck} pursued a different, complementary approach that relies on spatial dynamics to derive conditions for large-amplitude localized patterns to form snaking diagrams or isolas based on the existence properties of fronts that connect the homogeneous state $u_0$ to the patterned state $u_p$. Other contributions focused on asymmetric states \cite{Burke,Burke2}, stability \cite{Makrides2}, symmetry breaking \cite{Makrides,Xu,Wagenknecht}, and nonexistence of snaking \cite{Aougab}.

Despite this progress, much remains unknown. For instance, we are not aware of examples where the conditions for isolas and snaking for large-amplitude patterns can be checked analytically. Furthermore, little is known about the properties of planar localized patterns: while the bifurcation diagrams of some of these patterns can be explained using spatial dynamics \cite{Avitabile}, hexagon or rhombus patches form very complex bifurcation curves \cite{Lloyd} that are largely unexplained; see Figure~\ref{fig:SquareSnake} for an example.

To understand better what causes localized patterns to organize themselves in snaking diagrams or isolas, we focus in this paper on lattice dynamical systems. Lattice systems consist of an identical differential equation for each point on the lattice that are then coupled by a linear bounded operator that reflects nearest-neighbour interaction. It is well documented that localized patterns in lattice dynamical systems can exhibit snaking 
\cite{Chong,Kusdiantara,Papangelo,Taylor,Woods,Yulin}. In particular, the existence of homoclinic tangles in the discrete maps that capture stationary structures explains the coexistence of many localized structures \cite{Coullet,Woods,Carretero}, though this is not sufficient to explain how localized structures are connected globally to yield snaking diagrams or isolas. A concrete example that we will use in this paper is the real cubic-quintic Ginzburg--Landau equation
\begin{equation}\label{LDS_Intro}
\dot{U}_n = d(U_{n+1} + U_{n-1} - 2U_n) - \mu U_n + 2U_n^3 - U_n^5, \quad n\in\mathbb{Z}
\end{equation}
posed on $\mathbb{Z}$, where $U_n\in\mathbb{R}$ denote the state variables, $d>0$ represents the strength of the coupling between nearest neighbours, and $\mu$ is a bifurcation parameter. Equation (\ref{LDS_Intro}) and its discrete cubic-quintic nonlinear Schr\"{o}dinger version have a long history in nonlinear optics, for instance as a model for rotating waves in optical waveguides, and many papers have been devoted to the existence, stability, and bifurcations of its localized structures, primarily using numerical techniques; we refer to \cite{Carretero,Chong2,Yulin} and the monograph \cite{Pelinovsky} for sample results and further references. As shown in Figure~\ref{fig:Snaking_1D}, the Ginzburg--Landau equation (\ref{LDS_Intro}) exhibits snaking for sufficiently small positive values of the coupling parameter $d$, and our goal is to explain this phenomenon rigorously.

The analysis we shall present in this paper consists of two parts. In part~I, we will use spatial dynamics to understand the existence branches corresponding to localized stationary patterns of a general lattice dynamical system posed on $\mathbb{Z}$, assuming we know the properties of fronts that connect two different patterned states. This analysis will, in particular, predict when patterns arrange themselves in snaking curves or in isolas. Our approach is similar to our previous analysis in \cite{Beck} for the case of partial differential equations. To use (\ref{LDS_Intro}) as an illustration, its stationary solutions satisfy the discrete dynamical system
\begin{equation}\label{e1}
\begin{split}
u_{n+1} &= v_n, \\
v_{n+1} &= 2v_n - u_n +\frac{1}{d}(\mu v_n - 2v_n^3 + v_n^5),
\end{split}
\end{equation}
where $(u_n,v_n)=(U_{n-1},U_n)$, and we can find fronts and localized patterns of (\ref{LDS_Intro}) as heteroclinic and homoclinic orbits of (\ref{e1}).

In part~II, we will focus on the concrete system (\ref{LDS_Intro}) and exploit the fact that its anti-continuum limit, which corresponds to the uncoupled system that arises when setting $d=0$, provides a regime that is readily accessible analytically. We will show that we can verify the conditions of our general theory near this limit for $0<d\ll1$ and demonstrate that both snaking and isolas can occur in (\ref{LDS_Intro}). We note that the anti-continuum limit of (\ref{LDS_Intro}) was studied analytically in \cite{Schneider} via a justification of its variational approximation, though the connection to snaking and isolas was not studied there.

We emphasize that the methods we use to analyse patterns near the anti-continuum limit do not rely on spatial dynamics and can therefore be applied more broadly to other lattices. To illustrate this aspect of our work, select an arbitrary discrete set $\Lambda\subset\mathbb{R}^n$ as the index set, consider, for instance, the map $\mathcal{F}:\ell^\infty(\Lambda)\times \mathbb{R}\to\ell^\infty(\Lambda)$ that acts by $[\mathcal{F}(U,\mu)]_\lambda:=-\mu U_\lambda + 2U_\lambda^3 - U_\lambda^5$ with $\lambda\in\Lambda$, and choose any bounded linear operator $L:\ell^\infty(\Lambda)\to\ell^\infty(\Lambda)$ to reflect coupling between lattice points. The resulting lattice dynamical system is then given by
\begin{equation}\label{GeneralLattice}
\dot{U} = dLU + \mathcal{F}(U,\mu), \quad 
U = (U_\lambda)_{\lambda\in\Lambda} \in \ell^\infty(\Lambda)
\end{equation}
where $d,\mu\in\mathbb{R}$ are as above. The following lemma, which follows directly from the implicit function theorem, shows persistence of stationary patterns near the anti-continuum limit for $0<\mu<1$.

\begin{lem}\label{lem:GeneralLattice}
Choose a compact interval $K\subset(0,1)$ and assume that $U^*(\mu)$ is a smooth function so that $\mathcal{F}(U^*(\mu),\mu)=0$ for all $\mu\in K$. There exist $d_0 > 0$ and a unique function $\tilde{U}:(-d_0,d_0)\times K\to\ell^\infty(\Lambda)$ so that $\tilde{U}(d,\mu)$ is smooth, it is a stationary solution of (\ref{GeneralLattice}) for all $d\in(-d_0,d_0)$ and $\mu\in K$, and it satisfies $\tilde{U}(0,\mu)=U^*(\mu)$ for each $\mu\in K$.
\end{lem}

Lemma~\ref{lem:GeneralLattice} applies to general bounded coupling operators (including graph laplacians on networks for which the degree of nodes is bounded), and it can be generalized easily to more general nonlinearities and to systems of equations to show persistence of localized patterns away from bifurcations near the anti-continuum limit. We refer to \cite{McCullen} for a numerical study of snaking of localized patterns in a predator-prey model on Barab\'{a}si--Albert networks, where the coupling operator $L$ is given by the graph laplacian.

\begin{figure}
\begin{tabular}{cc}
\includegraphics[height=0.4\textwidth]{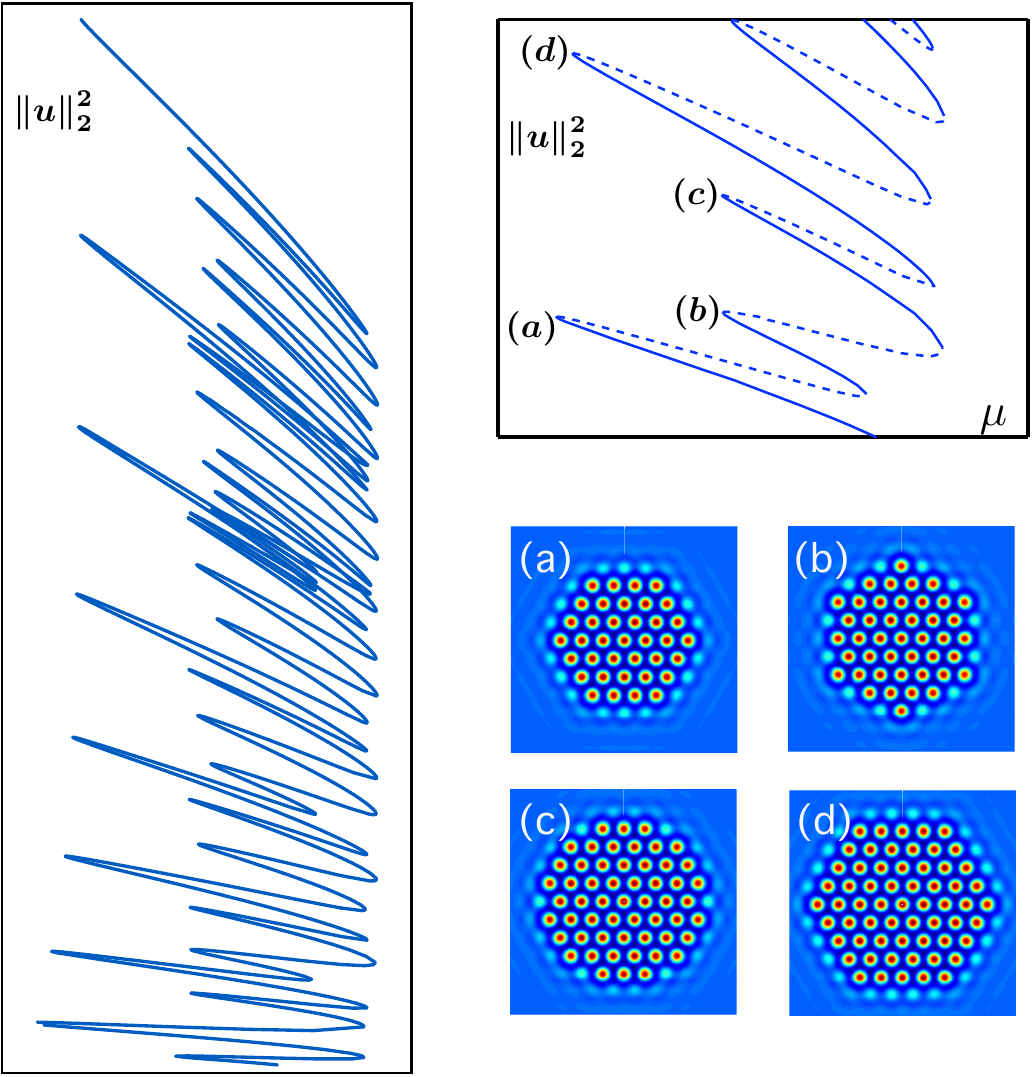} &
\includegraphics[height=0.4\textwidth]{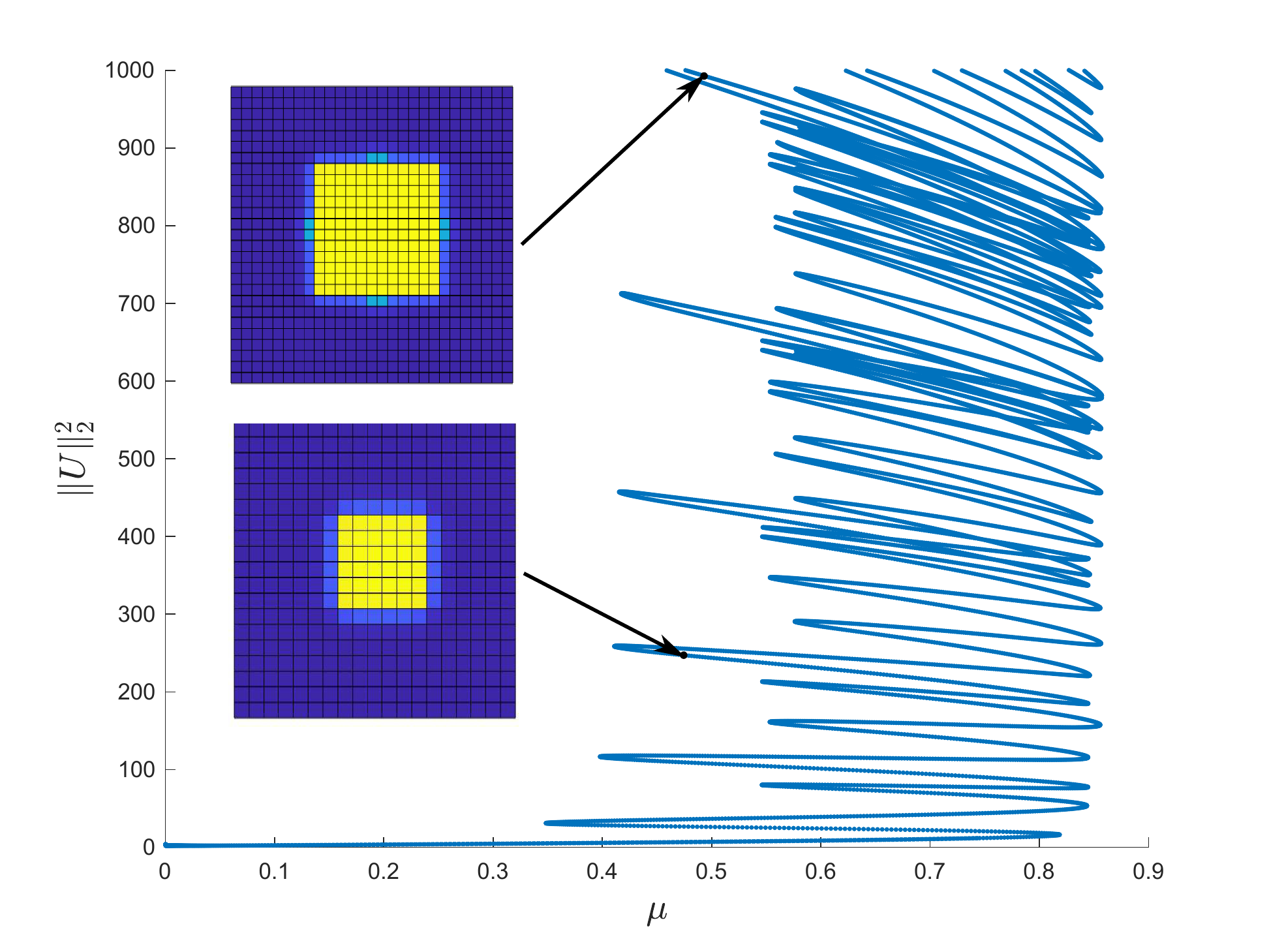} 
\end{tabular}
\caption{Shown are bifurcation curves of localized patterns and representative profiles in the Swift--Hohenberg equation \cite{Lloyd} [left] and the lattice system (\ref{LDS2D_Intro}) posed on a planar square lattice with $d=0.1$ [right].}
\label{fig:SquareSnake}
\end{figure}

The persistence result stated in Lemma~\ref{lem:GeneralLattice} breaks down at $\mu=0$ and $\mu=1$, where bifurcations take place when $d=0$. It is exactly near these values of $\mu$ that solution branches are arranged in isolas or snaking curves, and bifurcation theory can be used to analyse the fate of solutions near these points and therefore help decipher the global bifurcation structure of solutions. For instance, as shown in Figure~\ref{fig:SquareSnake}, the system
\begin{equation} \label{LDS2D_Intro}
\dot{U}_{n,m} = d(U_{n+1,m} + U_{n-1,m} + U_{n,m+1} + U_{n,m-1} - 4U_{n,m}) - \mu U_{n,m} + 2U_{n,m}^3 - U_{n,m}^5, \quad (n,m)\in\mathbb{Z}^2
\end{equation}
posed on the square lattice $\mathbb{Z}^2$ exhibits snaking of localized patterns for sufficiently small positive values of the coupling parameter $d$ that strikingly resembles the snaking curves of hexagon patches found in \cite{Lloyd} for the planar Swift--Hohenberg equation. The general strategy developed in this manuscript for analysing planar localized patterns in the anti-continuum limit via bifurcation theory is applicable to (\ref{LDS2D_Intro}), and we refer to \cite{Bramburger} for details. Such an analysis may provide insight into the bifurcations of planar hexagon patterns arising in the Swift--Hohenberg equation: we believe that the similarity of the bifurcation diagrams in Figure~\ref{fig:SquareSnake} is not an accident but arises because hexagonal patches in the Swift--Hohenberg equation may form through interactions of individual localized spots that respect a hidden hexagonal lattice created by the initial hexagon patch --- this hidden lattice is explicitly enforced in the planar square lattice system.

This manuscript is organized as follows. We first carry out a general analysis of one-dimensional lattice systems: we formulate our hypotheses for discrete maps and state our main results in  \S\ref{sec:Results}, then introduce a local coordinate system to describe trajectories in a neighbourhood of a fixed point in \S\ref{sec:Shilnikov}, and finally apply these technical results in \S\ref{sec:Matching} to construct symmetric and asymmetric homoclinic orbits. In the second part in \S\ref{sec:LDS}, we apply our results to the anti-continuum limit of the concrete system (\ref{LDS_Intro}). Section~\S\ref{sec:Discussion} contains a discussion of our results and our preliminary numerical computations for planar localized solution patches. 


\section{Main Results}\label{sec:Results}

We consider a smooth function $F:\mathbb{R}^2\times\mathbb{R} \to \mathbb{R}^2$ and the iterative scheme
\begin{equation} \label{Map}
	u_{n+1} = F(u_n,\mu),
\end{equation}
where $\mu$ is a bifurcation parameter. We further assume that $F$ is a diffeomorphism for each fixed $\mu \in \mathbb{R}$, and therefore since $F^{-1}(\cdot,\mu)$ exists, for each $\mu$ we may also iterate (\ref{Map}) backwards using the iterative scheme 
\begin{equation} \label{InverseMap}
	u_{n-1} = F^{-1}(u_n,\mu).
\end{equation}
The following hypothesis assumes that the mapping $F$ exhibits a reversible symmetry which relates the functions $F$ and $F^{-1}$ for all $\mu$.

\begin{hyp} \label{hyp:Reverser} 
	There exists a linear map $\mathcal{R}:\mathbb{R}^2 \to \mathbb{R}^2$ with $\mathcal{R}^2 = 1$ and ${\rm dim\ Fix}(\mathcal{R}) = 1$ so that $F^{-1}(u,\mu) = \mathcal{R}F(\mathcal{R}u,\mu)$ for all $u \in \mathbb{R}^2$ and $\mu \in J$.
\end{hyp}

Hypothesis~\ref{hyp:Reverser} is the discrete dynamical systems analogue of the reverser symmetry exploited in the continuous spatial setting. Notice that (\ref{InverseMap}) can be now be written
\[
	u_{n-1} = \mathcal{R}F(\mathcal{R}u_n,\mu).	
\]
Upon setting $v_n = \mathcal{R}u_n$ for all $n \in \mathbb{Z}$, we arrive at the backward iteration scheme governed by $F$ given by
\[
	v_{n-1} = F(v_n,\mu).
\]
Therefore, we see that if $\{u_n\}_{n\in\mathbb{Z}}$ is a solution to (\ref{Map}), so is $\{\mathcal{R}u_{-n}\}_{n\in\mathbb{Z}}$. A solution $\{u_n\}_{n\in\mathbb{Z}}$ of (\ref{Map}) is said to be {\em symmetric} if $\mathcal{R}\{u_n\}_{n\in\mathbb{Z}} = \{u_n\}_{n\in\mathbb{Z}}$. Note that if $u_0 \in {\rm Fix}(\mathcal{R})$ we have $u_n = \mathcal{R}u_{-n}$ for all $n \in \mathbb{Z}$, and if $u_0 = \mathcal{R}u_{-1}$ we then have that $u_{n} = \mathcal{R}u_{-n-1}$ for all $n \in \mathbb{Z}$. Such orbits provide examples of symmetric solutions. This leads to the following lemma which characterizes all symmetric solutions to (\ref{Map}).

\begin{lem}\label{lem:SymSols} 
	Let $u = \{u_n\}_{n\in\mathbb{Z}}$ be a symmetric solution to (\ref{Map}). Then there exists exists an $n \in \mathbb{Z}$ such that $\mathcal{R}u_n = u_n$ or $\mathcal{R}u_{n-1} = u_n$.  
\end{lem}

\begin{proof}
	Assume that $u = \{u_n\}_{n\in\mathbb{Z}}$ is a symmetric solution to (\ref{Map}). That is, $\mathcal{R}u = u$. Then, fixing a positive integer $n$, it follows that there exists $k\in\mathbb{Z}$ such that $\mathcal{R}u_n = u_k$. If $k = n,n-1$ we are done, and therefore we turn now to the case that $k \neq n,n-1$. Hypothesis~\ref{hyp:Reverser} implies that
	\[
		u_{k+1} = F(u_k) = F(\mathcal{R}u_n) = \mathcal{R}F^{-1}(u_n) = \mathcal{R}u_{n-1}, \\
	\] 
	where we have suppressed the dependence on $\mu \in J$ for convenience. This shows that $\mathcal{R}u_{n-1} = u_{k+1}$, and continuing these arguments we can inductively show that for all $j \in \mathbb{Z}$ we have $\mathcal{R}u_{n-j} = u_{k+j}$. In particular, there exists an integer $j$ such that $k = n - 2j$ or $k = n - 2j + 1$ depending on whether $n$ and $k$ have the same parity or not. Therefore, either $n - j = k+j$ or $n -j = k +j -1$, which proves the claim.
\end{proof} 

We now provide the following definition.

\begin{defn}\label{def:SymSols} 
	Let $u = \{u_n\}_{n\in\mathbb{Z}}$ be a symmetric solution to (\ref{Map}). Then, if there exists an $n \in \mathbb{Z}$ such that $\mathcal{R}u_n = u_n$, the solution $u$ is said to be {\bf on-site}. Otherwise, the solution is said to be {\bf off-site}.  
\end{defn}

We next state our assumptions on the fixed points of (\ref{Map}) belonging to ${\rm Fix}(\mathcal{R})$ we are interested in.

\begin{hyp}\label{hyp:FixedPts} 
	We assume that there exists a compact interval $J \subset \mathbb{R}$ with nonempty interior such that for each $\mu \in J$, the points $u = 0$ and $u = u^*$ belonging to ${\rm Fix}(\mathcal{R}) \subset \mathbb{R}^2$ are hyperbolic fixed points of (\ref{Map}). We further assume that the eigenvalues of the matrix $F_u(u^*,\mu)$ are real and positive. 
\end{hyp}

Recall that linearizing about a fixed point of a reversible map belonging to the subspace ${\rm Fix}(\mathcal{R})$ results in a matrix with the property that if $\lambda$ is a nonzero eigenvalue, then so must be $\bar{\lambda}$, $\frac{1}{\lambda}$, $\frac{1}{\bar{\lambda}}$ \cite[Proposition~16.3.4]{Wiggins}. Hence, reversibility of the mapping (\ref{Map}) implies that both $0$ and $u^*$ are saddles, and hence Hypothesis~\ref{hyp:FixedPts} implies the existence of one-dimensional stable and unstable manifolds of the fixed points $0$ and $u^*$. Therefore homoclinic and heteroclinic orbits connecting these fixed points can be obtained by identifying intersections of these stable and unstable manifolds. Our interest here will be in understanding the bifurcation behaviour of heteroclinic orbits of (\ref{Map}) that connect the fixed points $u = 0$ and $u = u^*$, which we will see informs the bifurcation behaviour of on- and off-site homoclinic orbits of the trivial fixed point. To do this, let us consider the Banach space $\ell^\infty := \ell^\infty(\mathbb{Z})$ defined in the introduction and the left shift operator, $S:\ell^\infty \to \ell^\infty$, acting by
\begin{equation}\label{Shift}
	[Su]_n := u_{n+1}, \quad \forall n \in \mathbb{Z},\ u\in\ell^\infty.
\end{equation}
This allows one to identify bounded solutions of (\ref{Map}) with roots of the function 
\[
	\mathcal{G}:\ell^\infty \times J \to \ell^\infty, \quad \mathcal{G}(u,\mu) := Su - \mathcal{F}(u,\mu),
\]
where $[\mathcal{F}(u,\mu)]_n := F(u_n,\mu)$ for all $n \in \mathbb{Z}$. Notice that $\mathcal{G}$ is equivariant with respect to $S$ in that $S\mathcal{G}(u,\mu) = \mathcal{G}(Su,\mu)$ for all $\mu \in J$ and $u \in \ell^\infty$.

It is straightforward to find that $\mathcal{G}$ is a smooth function since $F$ was assumed to be smooth. The partial Fr\'echet derivatives of $\mathcal{G}$ with respect to the first and second component, denoted $\mathcal{G}_u(u,\mu):\ell^\infty \to \ell^\infty$ and $\mathcal{G}_\mu(u,\mu):\mathbb{R} \to \ell^\infty$, are given by
\[
	\begin{split}
		[\mathcal{G}_u(u,\mu)v]_n &= v_{n+1} - F_u(u_n,\mu)v_n, \\
		[\mathcal{G}_\mu(u,\mu)\nu]_n &= -F_\mu(u_n,\mu)\nu,
	\end{split}
\]
respectively, for all $n \in \mathbb{Z}$, $v \in \ell^\infty$, and $\nu \in \mathbb{R}$. Elements of the kernel of $\mathcal{G}_u(u,\mu)$ are exactly the bounded solutions to the linear variational equation 
\[
	v_{n+1} = F_u(u_n,\mu)v_n
\]
associated to (\ref{Map}) about the point $u \in \ell^\infty$. Furthermore, the second partial Fr\'echet derivative of $\mathcal{G}$ with respect to the first component, denoted $\mathcal{G}_{uu}(u,\mu):\ell^\infty \to \ell^\infty$, is defined by the action
\[
	[\mathcal{G}_{uu}(u,\mu)[v,v]]_n = - F_{uu}(u_n,\mu)[v_n,v_n],
\]
for all $n \in \mathbb{Z}$ and $v \in \ell^\infty$. 

Let us now consider the subset $X$ of $\ell^\infty$ given by
\[
	X = \bigg\{u = \{u_n\}_{n \in \mathbb{Z}}\in\ell^\infty :\ \lim_{n\to -\infty} u_n = 0,\ \lim_{n\to\infty} u_n = u^*\bigg\}. 
\]
The set $\mathcal{G}^{-1}(0) \cap (X \times J)$ is the set of all heteroclinic connections of the map (\ref{Map}) from the fixed point $0$ to $u^*$ for $\mu \in J$. Moreover, Hypothesis~\ref{hyp:Reverser} implies that all heteroclinic connections of the map (\ref{Map}) from the fixed point $u^*$ to $0$, for each value of $\mu \in J$, can be completely identified through the set $\mathcal{G}^{-1}(0) \cap (X \times J)$ by simply applying the reverser $\mathcal{R}$ to each element. It follows from Hypothesis~\ref{hyp:FixedPts} and \cite[Lemma 2.3]{Beyn} that for each $(\bar{u},\bar{\mu}) \in \mathcal{G}^{-1}(0)\cap (X \times J)$ the linearization $\mathcal{G}_u(\bar{u},\bar{\mu})$ is a Fredholm operator of index 0. In particular, the total derivative
	\[
		D\mathcal{G}(\bar{u},\bar{\mu}) :=[\mathcal{G}_u(\bar{u},\bar{\mu}), \mathcal{G}_\mu(\bar{u},\bar{\mu})] : \ell^\infty \times \mathbb{R} \to \ell^\infty
	\]  
	is Fredholm with index 1. We now state the following hypothesis. 

\begin{hyp} \label{hyp:Gamma} 
	Assume there exists a connected component $\Gamma \subset \mathcal{G}^{-1}(0)\cap (X \times J)$ such that we have the following:
	\begin{enumerate}
		\item For each $(\bar{u},\bar{\mu}) \in \Gamma$, the total derivative $D\mathcal{G}(\bar{u},\bar{\mu}) : \ell^\infty \times \mathbb{R} \to \ell^\infty$ is surjective.
		\item If there exists a nonzero vector $v \in \ell^\infty$ with $\mathcal{G}_u(\bar{u},\bar{\mu})v = 0$ for some $(\bar{u},\bar{\mu}) \in \Gamma$, then $\mathcal{G}_{uu}(\bar{u},\bar{\mu})[v,v]  \notin {\rm Im}(\mathcal{G}_u(\bar{u},\bar{\mu}))$.
		\item We have $\Gamma \cap (\ell^\infty \times \partial J) = \emptyset$, and there exists $K > 0$ such that $\|\bar{u}\|_\infty \leq K$ for all $(\bar{u},\bar{\mu}) \in \Gamma$.
	\end{enumerate}
\end{hyp}

We note that the first item in Hypothesis~\ref{hyp:Gamma} implies that the set $\Gamma$ is a smooth curve embedded in the space $X \times \mathring{J}$ corresponding to a smooth family $\bar{u}(s)$ of heteroclinic orbits of (\ref{Map}) for $\mu = \mu(s)$ with $s$ in some interval. The following lemma relates $\Gamma$ to the dynamical properties of the heteroclinic orbits of (\ref{Map}).  

\begin{lem}\label{lem:ManInt} 
	Assume Hypotheses~\ref{hyp:Reverser}-\ref{hyp:Gamma} and let $(\bar{u},\bar{\mu}) \in \Gamma$, then
	\begin{enumerate}
		\item If $\mathcal{G}_u(\bar{u},\bar{\mu})$ is invertible, then $W^u(0,\bar{\mu})$ and $W^s(u^*,\bar{\mu})$ intersect transversely along the heteroclinic orbit $\bar{u}$.
		\item If $\mathcal{G}_u(\bar{u},\bar{\mu})$ has nontrivial kernel, then $W^u(0,\bar{\mu})$ and $W^s(u^*,\bar{\mu})$ have a quadratic tangency along the heteroclinic orbit $\bar{u}$.  
	\end{enumerate} 	
\end{lem}

\begin{proof}
	The first statement follows from \cite[Theorem~3.1]{Beyn}. The second statement follows via a Lyapunov-Schmidt reduction in a neighbourhood of the intersection of $W^u(0,\bar{\mu})$ and $W^s(u^*,\bar{\mu})$: the details are similar to the continuum setting proven in \cite[Theorem~2]{Palmer}.	 
\end{proof} 

Next consider the orbit space $\ell^\infty/\langle S\rangle$, which is the set of equivalence classes in $\ell^\infty$ with respect to the shift $S$: for $u,v \in \ell^\infty$ we have $u \sim v$ if, and only if, there exists $k \in \mathbb{Z}$ such that $S^ku = v$ and write $[u] = \{v \in \ell^\infty:\ u\sim v\}$. Let 
\[
	\pi: \ell^\infty \times J \to  \ell^\infty/\langle S\rangle \times J, \quad (u,\mu) \to ([u],\mu)
\] 
be the quotient map onto this orbit space and define 
\[
	\bar{\Gamma} := \pi(\Gamma)
\]
to be the image of $\Gamma$ under the quotient map. Our interest lies in the case that $\bar{\Gamma}$ is a closed loop, leading to the following hypothesis. 

\begin{hyp}\label{hyp:Loop} 
	$\bar{\Gamma}$ is a closed loop, that is, we can parametrize $\bar{\Gamma}$ by a smooth map $\gamma:[0,1]\to\bar{\Gamma}$ by $s \to ([u](s),\mu(s))$ with $\gamma(0) = \gamma(1)$.
\end{hyp}     

\begin{rmk} 
	We note for the reader that Hypothesis~\ref{hyp:Loop} is indeed necessary. For example, one could imagine a scenario where as $\mu$ approaches some value $\mu^* \in J$ the heteroclinic orbit could approach an intersection with another fixed point, say $u^{**}$. Thus, at $\mu = \mu^*$ we have two heteroclinic orbits, one connecting $0$ to $u^{**}$ and one connecting $u^{**}$ to $u^*$, and therefore $\bar{\Gamma}$ would not form a closed loop in this case. Scenarios of this kind were analyzed in the continuous spatial setting in \cite{Fiedler}. 
\end{rmk}

From Hypothesis~\ref{hyp:Loop}, for each $s \in [0,1]$ the elements of the curve $\gamma(s)$ lifts to infinitely many points in $\Gamma$, all of which are merely shifts of each other. Taking one such point $(u(0),\mu(0))\in\Gamma$ so that $\pi(u(0),\mu(0)) = \gamma(0) \in \bar{\Gamma}$, we may produce a smooth connected curve $(u(s),\mu(s)) \in \Gamma$ so that $\pi(u(s),\mu(s)) = \gamma(s)$ for all $s \in [0,1]$. Similar to \cite{Aougab}, if (i) $u(1) = u(0)$ we refer to the curve $(u(s),\mu(s))$ as a 0-loop, otherwise (ii) we must have $S^ku(1) = u(0)$ for some integer $k \neq 0$, and we refer to this case as a 1-loop. If the latter case occurs, without loss of generality we can always consider $k = 1$ by replacing the right-hand side of (\ref{Map}) with $F^k(u_n,\mu)$.   

Our interest lies in constructing homoclinic orbits of the trivial fixed point to (\ref{Map}) that remain close to the fixed point $u^*$ for $n \in \{0,\dots,N\}$ for appropriate large values of $N \gg 1$ and $\mu \in J$. More precisely, we denote the $\delta$-neighbourhood of the fixed point $u^*$ by $U_\delta$ as well as $W^s(0,\mu)$ and $W^u(0,\mu)$ as the parameter-dependent stable and unstable manifolds, respectively, of the trivial fixed point. Then, we seek homoclinic orbits $u = \{u_n\}_{n\in\mathbb{Z}}$ which satisfy
\[
	u \in W^s(0,\mu)\cap W^u(0,\mu),\quad u_n \in U_\delta,\ {\rm for}\ n\in\{0,\dots,N\}, 
\]
for some $N \gg 1$. In \S\ref{sec:Matching} we will show that there exists a smooth one-dimensional manifold $\Gamma_\mathrm{lift} \subset \mathbb{R}\times J$, related to $\Gamma$, which is symmetric under discrete shifts in the first component. Furthermore, $1$-loops in $\bar{\Gamma}$ are lifted to a single connected curve in $\Gamma_\mathrm{lift}$, whereas $0$-loops are lifted to infinitely many distinct closed curves. We now provide the following theorem which is our main result pertaining to symmetric homoclinic orbits.

\begin{thm}\label{thm:MainThm} 
	Assume Hypotheses~\ref{hyp:Reverser}-\ref{hyp:Loop} are met. There exist a number $\eta \in (0,1)$ and submanifolds $\Gamma_{1,2} \subset \mathbb{R} \times J$ such that the following is true:
	\begin{compactenum}
		\item There exists an off-site $($on-site$)$ homoclinic orbit of length $2N$ $(2N+1)$ if, and only if, there exists $s\in[0,2\pi)$ so that $(2\pi N+s,\mu)\in\Gamma_1$ $($resp. $\Gamma_2)$.
		\item For $j = 1,2$, the manifolds $\Gamma_\mathrm{lift}$ and $\Gamma_j$ are, for each fixed $k \geq 2$, $\mathcal{O}(\eta^{2N + (j-1)})$-close to each other in the $C^k$-sense near each point in $(2\pi N+s,\mu)\in\Gamma_j$, where $s \in [0,1)$. 
	\end{compactenum}
\end{thm}

From Theorem~\ref{thm:MainThm} we see that the bifurcation curves of symmetric homoclinic orbits are dictated by the form of $\Gamma_\mathrm{lift}$, which is in turn dictated by the form of $\Gamma$. Hence, we see that symmetric homoclinic orbits snake if $\Gamma$ is a $1$-loop, while the bifurcation diagram consists of isolas if $\Gamma$ contains only $0$-loops. Therefore, our results show that all of the bifurcation structure of symmetric homoclinic orbits of (\ref{Map}) can be inferred through an understanding of the bifurcations of a heteroclinic tangle connecting $0$ and $u^*$. Geometrically, snaking is caused by the intersecting stable and unstable manifolds move through each other as $\mu$ increases, whereas isolas are caused by these manifolds not moving through each other as $\mu$ increases. This is demonstrated in Figure~\ref{fig:Heteroclinics}.   

\begin{figure} 
	\centering
	\includegraphics[width=0.45\textwidth]{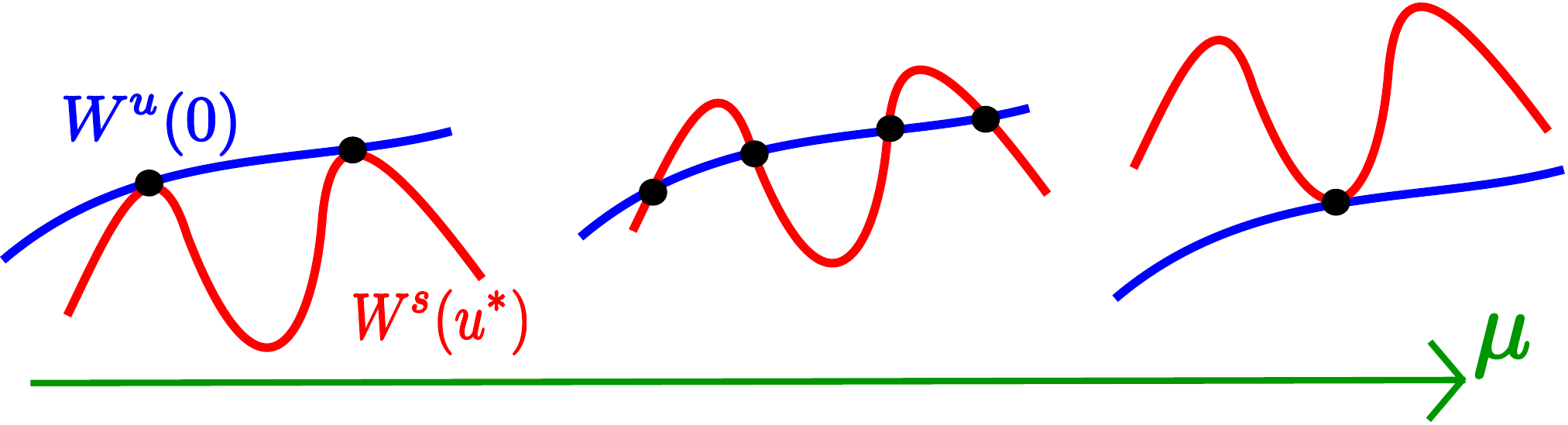} \quad\quad
	\includegraphics[width=0.45\textwidth]{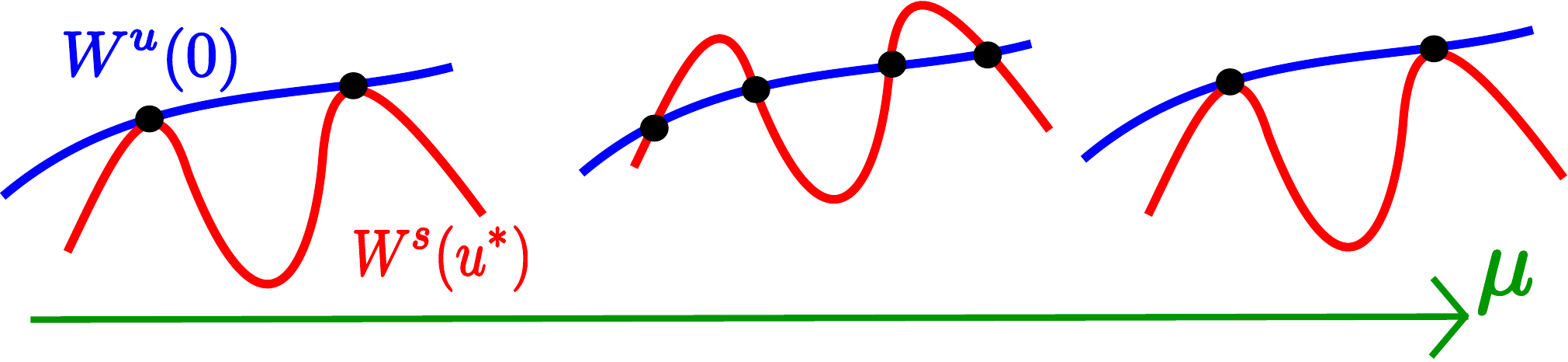} 
	\caption{Different bifurcation scenarios for the heteroclinic orbits between $0$ and $u^*$ that lead to different bifurcation diagrams for the symmetric homoclinic orbits of (\ref{Map}). On the left the intersecting stable and unstable manifolds move through each other as $\mu$ increases, leading to snaking. On the right the intersecting stable and unstable manifolds do not move through each other as $\mu$ increases, leading to isolas.}
	\label{fig:Heteroclinics}
\end{figure} 

We now state the corresponding result for the asymmetric homoclinic orbits. We state this result in full generality, but we refer the reader to \S\ref{sec:AsymHom} for more precise statements and results. We further refer the reader to Figure~\ref{fig:Snaking_1D} for visual confirmation of the results of the following theorem. 

\begin{thm}\label{thm:AsymThm} 
	Assume Hypotheses~\ref{hyp:Reverser}-\ref{hyp:Loop}. The following is true:
	\begin{compactenum}
	\item Assume that at $\mu_0 \in \mathring{J}$ the manifolds $W^s(u^*,\mu_0)$ and $W^u(0,\mu_0)$ intersect along a quadratic tangency. Then there exist $\eta \in (0,1)$ and $N_* \geq 1$ such that for each $N \geq N_*$ precisely two branches of asymmetric homoclinic orbits (mapped into each other by $\mathcal{R}$) emanate in a pitchfork bifurcation from a symmetric homoclinic orbit at a value of $\mu$ that is $\mathcal{O}(\eta^N)$-close to $\mu_0$. 
	\item Generically, these curves of asymmetric homoclinic orbits described in (1) are smooth with boundaries given by these pitchfork bifurcations. These pitchfork bifurcations take place near saddle-node bifurcations of symmetric homoclinic orbits of opposite curvature. 
	\item Generically, all other curves of asymmetric homoclinic orbits must be smooth closed curves (isolas).
	\end{compactenum}  
\end{thm}

The proofs of Theorems~\ref{thm:MainThm} and \ref{thm:AsymThm} are broken down over Sections~\ref{sec:Shilnikov} and \ref{sec:Matching}. In \S\ref{sec:Shilnikov} we use the stable manifold theorem to construct local coordinates in the neighbourhood $U_\delta$ of the fixed point $u^*$. Using these local coordinates we are able to capture the bifurcation behaviour of the stable manifold of the trivial equilibrium in this neighbourhood, which leads to the matching conditions in \S\ref{sec:Matching}. In \S\ref{sec:SymHom} we then prove the existence and detail the full bifurcation structure of symmetric homoclinic orbits of (\ref{Map}), proving Theorem~\ref{thm:MainThm}. Then in \S\ref{sec:AsymHom} we extend these results to prove the existence and full bifurcation structure of asymmetric homoclinic orbits of (\ref{Map}), in turn proving Theorem~\ref{thm:AsymThm}.     


\section{Local Coordinates About $u^*$}\label{sec:Shilnikov} 

Our goal in this section is to characterize and estimate  solutions that pass very close to the equilibrium $u^*$. The estimates we obtain will be used in the following section to construct homoclinic orbits to 0 that transition near $u^*$. To obtain the desired estimates, we introduce new coordinates that bring the dynamical system near fixed point $u^*$ into a form that we can analyze more easily. We begin by noting that Hypotheses~\ref{hyp:Reverser}-\ref{hyp:FixedPts} implies that the eigenvalues of $F_u(u^*,\mu)$ are of the form $0 < \lambda^{-1}(\mu) < 1 < \lambda(\mu)$ for some function $\lambda(\mu)$, which is then necessarily smooth in $\mu$. The next lemma describes the dynamics near the fixed point $u^*$.

\begin{lem}\label{lem:Shilnikov} 
	Assume Hypotheses~\ref{hyp:Reverser} and \ref{hyp:FixedPts} are met. Then, there exist a $\delta > 0$, a smooth change of coordinates mapping $u$ to $v = (v^s,v^u)$ near the fixed point $u = u^*$, and smooth functions $f^s_i,f^u_i:\mathcal{I}\times\mathcal{I}\times J \to \mathbb{R}$, $i = 1,2$, so that (\ref{Map}) is of the form 
	\begin{equation}\label{Shilnikov}
		\begin{split}
			v^s_{n+1} &= [\lambda(\mu)^{-1} + f^s_1(v^s_n,v^u_n,\mu)v^s_n + f^s_2(v^s_n,v^u_n,\mu)v^u_n]v^s_n, \\
			v^u_{n+1} &= [\lambda(\mu) + f^u_1(v^s_n,v^u_n,\mu)v^s_n + f^u_2(v^s_n,v^u_n,\mu)v^u_n]v^u_n, \\
		\end{split}
	\end{equation}
	for all $\mu \in J$, where $v^s_n,v^u_n \in \mathcal{I} := [-\delta,\delta]$, and the reverser $\mathcal{R}$ acts by
	\begin{equation}\label{ShilnikovReverser}
		\mathcal{R}(v^s,v^u) = (v^u,v^s).
	\end{equation}
\end{lem}

\begin{proof}[Proof of Lemma~\ref{lem:Shilnikov}] 
Let $0 \neq \xi(\mu) \in \mathbb{R}^2$ be the eigenvector associated to the eigenvalue $\lambda(\mu)^{-1}$ for each $\mu \in J$. For convenience we will suppress the dependence on $\mu$ throughout. Reversibility of the map $F$ implies that $\mathcal{R}\xi(\mu)$ is the eigenvector associated to the eigenvalue $\lambda(\mu)$ for all $\mu \in J$. The Stable Manifold Theorem for maps and the action of the reverser $\mathcal{R}$ imply the existence of a $\delta > 0$ and a smooth function $w: [-\delta,\delta] \to \mathbb{R}$ with $w(0) = w'(0) = 0$ so that 
	\[
		\begin{split}
			W^s_\mathrm{loc}(u^*) &= \{u^* + v^s\xi  + w(v^s)\mathcal{R}\xi: v^s \in [-\delta,\delta]\}, \\
			\mathcal{R}W^s_\mathrm{loc}(u^*) &= W^u_\mathrm{loc}(u^*) = \{u^* + v^u\mathcal{R}\xi +w(v^u)\xi: v^u \in [-\delta,\delta]\}
		\end{split}
	\]	
	locally describe the stable and unstable manifold of the fixed point $u^*$, respectively.  
	
	Define the map
	\begin{equation}\label{Change1}
		\begin{split}
			&\Phi:[-\delta,\delta] \times [-\delta,\delta] \to U_\delta(u^*) \\
			&(v^s,v^u)\mapsto u = u^* + (v^s\xi + w(v^s)\mathcal{R}\xi) + (v^u\mathcal{R}\xi  + w(v^u)\xi),
		\end{split}
	\end{equation}	
	and note that $\Phi(0) = u^*$ and that $\Phi$ is a local diffeomorphism since $w'(0) = 0$. Applying $\mathcal{R}$ gives
	\[
		\mathcal{R}u = u^* + v^s\mathcal{R}\xi +w(v^u)\xi + v^u\xi + w(v^s)\mathcal{R}\xi,
	\]
	and comparing with (\ref{Change1}) and using that $\Phi$ is a local diffeomorphism gives the desired action of $\mathcal{R}$ given in (\ref{ShilnikovReverser}).
	
	The expansions (\ref{Shilnikov}) for the map $\Phi^{-1}\circ F \circ \Phi$ follow from local invariance of the sets $W^s_\mathrm{loc}(u^*) = \{v^u = 0\}$ and $W^u_\mathrm{loc}(u^*) = \{v^s = 0\}$. This completes the proof.     
\end{proof} 

\begin{figure} 
	\centering
	\includegraphics[width=0.4\textwidth]{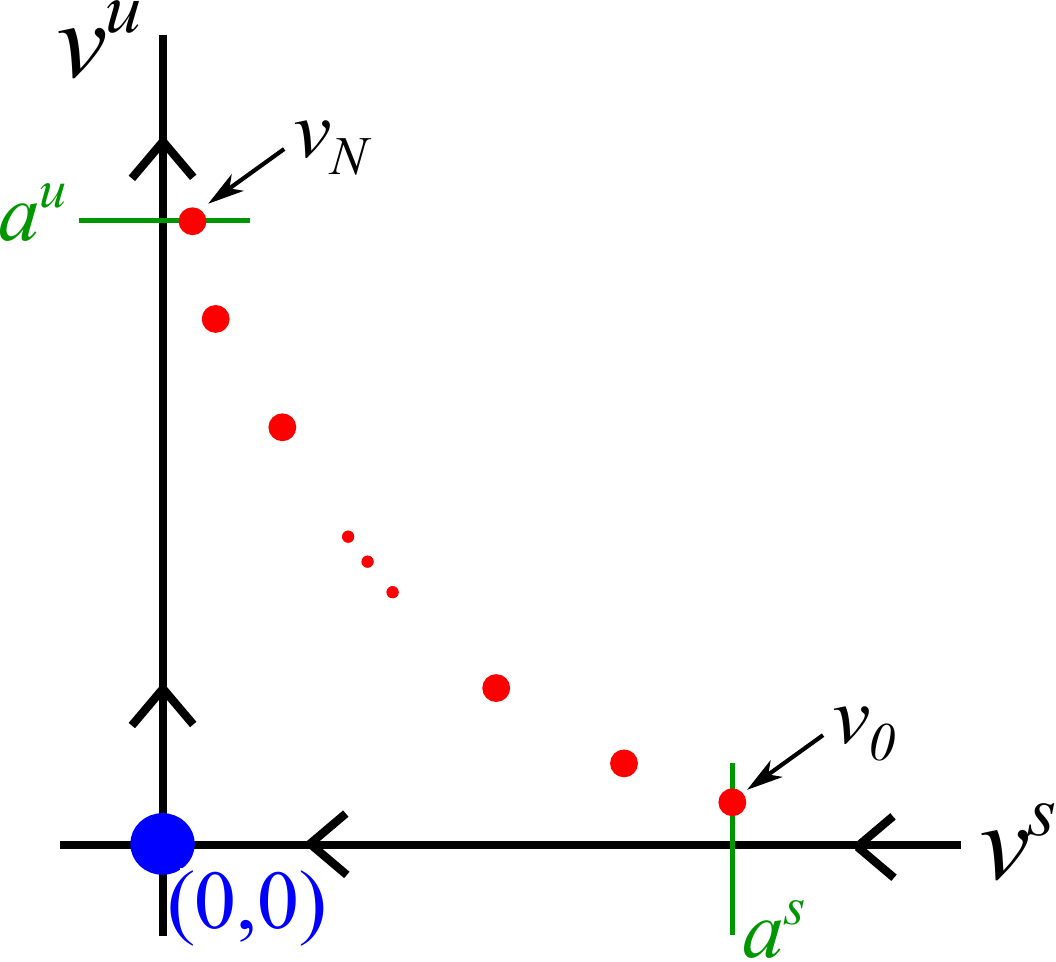}
	\caption{An illustration of the results of Lemma~\ref{lem:ShilSol}. The red points represent the solution, which starts exponentially close to the stable manifold (parametrized by $v^u = 0$) and after $N \geq 1$ iterates end at a point exponentially close to the unstable manifold (parametrized by $v^s = 0$).}
\label{fig:Schecter}
\end{figure}

\begin{lem}\label{lem:ShilSol} 
	There exist constants $\eta \in (0,1)$ and $M>0$ such that the following is true: for each $N > 0$, $a^s,a^u \in \mathcal{I}$, and $\mu \in J$ there exists a unique solution near the origin to (\ref{Shilnikov}), written $v_n = (v^s_n,v^u_n) \in \mathcal{I}\times \mathcal{I}$ with $n \in \{0, \dots, N\}$, such that
	\[
		v^s_{0} = a^s, \quad v^u_{N} = a^u.
	\] 
	Furthermore, this solution satisfies
	\begin{equation}\label{ShilBnds}
		|v^s_{n}| \leq M\eta^{n}, \quad |v^u_{n}| \leq M\eta^{N- n},
	\end{equation}
	for all $n \in \{0, \dots, N\}$, $v_n = v_n(a^s,a^u,\mu)$ depends smoothly on $(a^s,a^u,\mu)$, and the bounds (\ref{ShilBnds}) also hold for the derivatives of $v$ with respect to ($a^s,a^u,\mu$). Moreover, 
	\begin{equation}\label{ShilReverser}
		\mathcal{R}(v^s_n,v^u_n) = (v^u_{N-n},v^s_{N-n}), 
	\end{equation} 
	for all $n \in \{0,\dots, N\}$. In particular, the solution $v$ is symmetric if, and only if, $a^s = a^u$.
\end{lem}

\begin{proof}
	This result is the discrete time analogue of \cite[Theorem 2.2]{Schecter}, and follows via an application of the contraction mapping theorem. The action of $\mathcal{R}$ in (\ref{ShilReverser}) follows in the same way as in Lemma~\ref{lem:SymSols}, and the claims about symmetric solutions follow from (\ref{ShilnikovReverser}), (\ref{ShilReverser}), and uniqueness of solutions. The results of this proof are visualized in Figure~\ref{fig:Schecter}.
\end{proof}	


\section{Matching}\label{sec:Matching} 

\begin{figure} 
	\centering
	\includegraphics[height=0.25\textwidth]{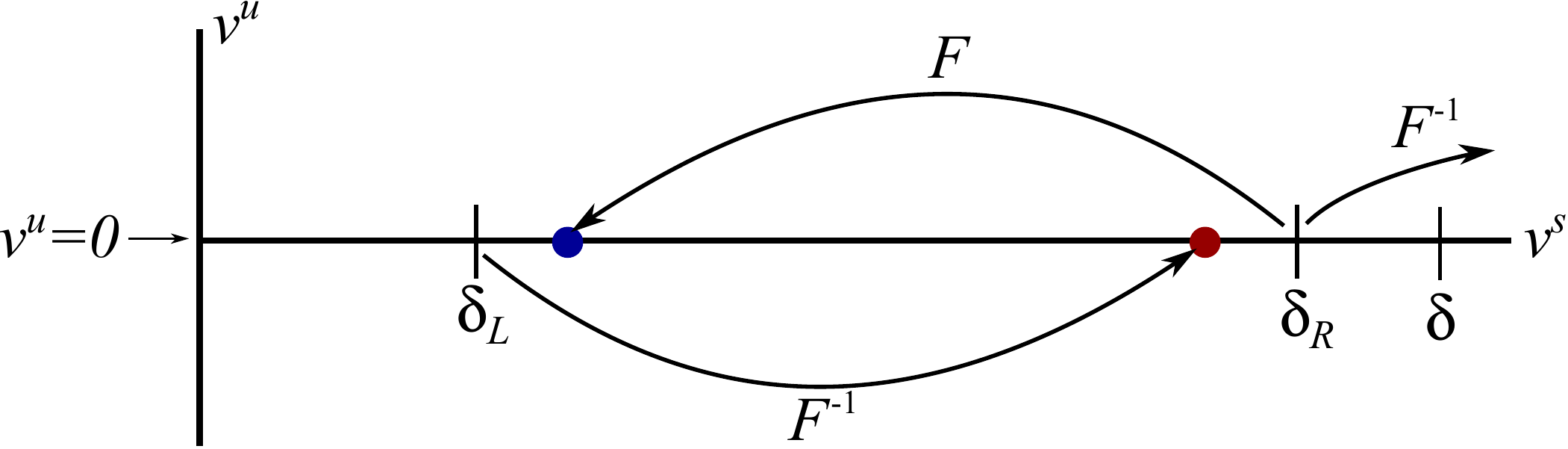}
	\caption{A visual representation of the criteria for forming the interval $\mathcal{I}_0 = [\delta_L,\delta_R]$. The set $\{v^u = 0\}$ is invariant under the mapping $F$, and we take $\mathcal{I}_0$ to be a large enough interval in this invariant set so that its right (left) endpoint is mapped into the interior of the interval by $F$ ($F^{-1}$). Furthermore, the rightmost endpoint must be mapped outside of the region of validity for the results of Lemma~\ref{lem:Shilnikov} by $F^{-1}$.}
\label{fig:Subinterval}
\end{figure}

In this section, we are interested in constructing homoclinic orbits to the fixed point $u = 0$ that spend a long time near the fixed point $u = u^*$. Furthermore, our work in this section will not only give the existence of such homoclinic orbits, but also their bifurcation structure with respect to varying $\mu$. We begin by noting that Hypothesis~\ref{hyp:FixedPts} implies that the stable and unstable manifolds of the fixed point $u^*$ are orientation preserving. Therefore, without loss of generality we can assume that 
\[
	W^u(0,\mu) \cap \{(v^s,0): v^s \in (0,\delta)\} \neq \emptyset \iff W^u(0,\mu) \cap W^s(u^*,\mu) \neq \emptyset
\]
for each $\mu \in J$. Then, consider $\delta_L,\delta_R \in (0,\delta)$ given so that for all $\mu \in J$ the following is true:
\begin{enumerate}
	\item The inverse of mapping (\ref{Map}), $F^{-1}$, maps the point $(v^s,v^u) = (\delta_L,0)$ into the interval $(\delta_L,\delta_R) \times \{v^u = 0\}$.
	\item The mapping (\ref{Map}), $F$, maps the point $(v^s,v^u) = (\delta_R,0)$ into the interval $(\delta_L,\delta_R) \times \{v^u = 0\}$.
	\item The inverse of mapping (\ref{Map}), $F^{-1}$, maps the point $(v^s,v^u) = (\delta_R,0)$ out of the set $\mathcal{I}\times\mathcal{I}$. 
\end{enumerate}
For simplicity, the choices of $\delta_L,\delta_R$ are illustrated for the reader in Figure~\ref{fig:Subinterval}. Note that such an interval can always be found since $\lambda(\mu) > 0$ can be both bounded above and below, and hence the right-hand side of (\ref{Shilnikov}) can be bounded uniformly above and away from zero for all $\mu \in J$ and sufficiently small $\delta > 0$. Let us now denote $\mathcal{I}_0 := [\delta_L,\delta_R]$ so that by definition $\mathcal{I}_0 \subset (0,\delta)$, and consider an open interval $\mathcal{I}_1$ such that $\mathcal{I}_0 \subset \mathcal{I}_1 \Subset (0,\delta)$. This allows for the definition of the segment
\[
	\Sigma_\mathrm{in} := \mathcal{I}_1\times\mathcal{I}.
\]
The action of $\mathcal{R}$ given in (\ref{ShilnikovReverser}) implies that we can further define
\[
	\Sigma_\mathrm{out} := \mathcal{R}\Sigma_\mathrm{in} = \mathcal{I}\times\mathcal{I}_1.
\]
The set $\Sigma_\mathrm{out}$, along with the intervals $\mathcal{I}_0$ and $\mathcal{I}_1$, are represented in Figure~\ref{fig:SigmaOut}. 

\begin{figure} 
	\centering
	\includegraphics[width=0.5\textwidth]{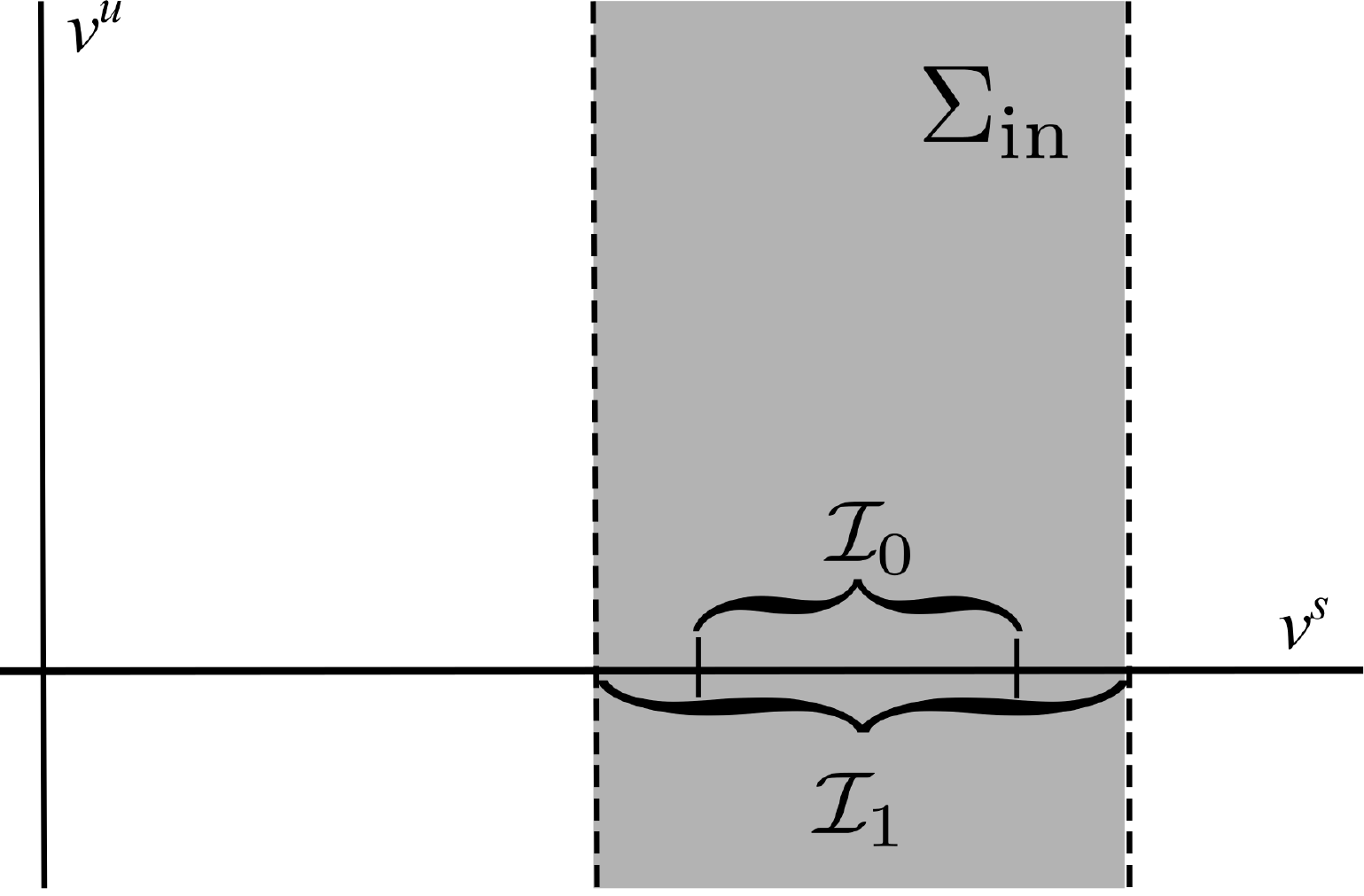}
	\caption{The set $\Sigma_\mathrm{in}$ (shaded) in $\mathcal{I}\times\mathcal{I}$.}
\label{fig:SigmaOut}
\end{figure}

\begin{lem}\label{lem:I0Nonempty} 
	For a fixed $\mu \in J$ we have	
	\[
		W^u(0,\mu) \cap \{(v^s,0): v^s \in (0,\delta))\} \neq \emptyset
	\]
	if, and only if,
	\[
		W^u(0,\mu) \cap \{(v^s,0): v^s \in \mathcal{I}_0\} \neq \emptyset.
	\]
\end{lem}

\begin{proof}
	The `if' direction is trivial since $\{(v^s,0): v^s \in \mathcal{I}_0\} \subset \{(v^s,0): v^s \in (0,\delta))\}$ by definition of $\mathcal{I}_0$. The `only if' direction follows from the definition of $\mathcal{I}_0$ since $\delta_L,\delta_R > 0$ were chosen to be far enough apart that $F$ cannot iterate an element of $\{(v^s,0): v^s \in (\delta_R,\delta))\}$ to an element of $\{(v^s,0):\ v^s \in (0,\delta_L)\}$ for all $\mu \in J$.
\end{proof} 

Using Lemma~\ref{lem:I0Nonempty} we now define the set
\[
	\Gamma_\mathrm{loc} := \bigcup_{\mu \in J} (W^u(0,\mu) \cap \{(v^s,0): v^s \in \mathcal{I}_0\}) \times \{\mu\} \subset \mathcal{I}_0 \times \{v^u = 0\} \times \mathring{J},
\]
which is closed and nonempty since $\Gamma$ is assumed to be nonempty and does not intersect the boundaries of $J$. Here $\Gamma_\mathrm{loc}$ is a local component of the heteroclinic orbits connecting $0$ to $u^*$ belonging to the smooth bifurcation curve $\Gamma$. This leads to the following lemma.

\begin{lem} \label{lem:GammaLoc} 
	For each $(v^s_*,0,\mu_*) \in \Gamma_\mathrm{loc}$ there exists an open neighbourhood $U_*\subset \mathcal{I}_1\times\mathcal{I}\times J$ of $(v^s_*,0,\mu_*)$ and a smooth function $g_*:U_* \to \mathbb{R}$ such that $g_*(v^s,v^u,\mu) = 0$ if, and only if, 
	\[
		(v^s,v^u,\mu) \in (W^u(0,\mu)\times \{\mu\})\cap U_*.
	\]
	Furthermore, we have 
	\[
		\nabla_{(v^s,\mu)} g_*(v^s,0,\mu) \neq 0 
	\]
	for all $(v^s,0,\mu) \in \Gamma_\mathrm{loc} \cap U_*$.
\end{lem} 

\begin{proof}
	From Hypothesis~\ref{hyp:Gamma} we have that $(v^s_*,0,\mu_*) \in \Gamma_\mathrm{loc}$ lies either at a transverse intersection of $W^s(u^*,\mu_*)$ and $W^u(0,\mu_*)$ or a quadratic tangency of these two manifolds. In the former case of a transverse intersection, we can locally parametrize a neighbourhood of $(v^s_*,0,\mu_*)$ in $\mathcal{I}_1\times\mathcal{I}\times J$ by the function $v^s = g(v^u,\mu)$ so that $v^s_* = g(0,\mu_*)$. Then, we can use the function $g$ to define 
	\[
		g_*(v^s,v^u,\mu) = v^s - g(v^u,\mu),
	\]  
	so that $g_*$ satisfies the claims of the lemma. Similarly, in the latter case of $(v^s_*,0,\mu_*)$ lying along a quadratic tangency of $W^s(u^*,\mu_*)$ and $W^u(0,\mu_*)$, we can locally parametrize a neighbourhood of $(v^s_*,0,\mu_*)$ in $\mathcal{I}_1\times\mathcal{I}\times J$ by the function $\mu = \tilde{g}(v^s,v^u)$ so that $\mu_* = \tilde{g}(v^s_*,0)$. In this case we would define 
	\[
		g_*(v^s,v^u,\mu) = \mu - \tilde{g}(v^s,v^u),	
	\]    
	which again satisfies the claims of the lemma. This completes the proof.
\end{proof} 

\begin{rmk} \label{rmk:GFunction} 
	Lemma~\ref{lem:GammaLoc} details that in a neighbourhood of any point in $\Gamma_\mathrm{loc}$ we may obtain a function whose zeros correspond to the unstable manifold of the trivial equilibrium. Since $\Gamma_\mathrm{loc}$ is a compact manifold, it follows that we may cover it with finitely many of these open neighbourhoods, which implies the existence of an $\varepsilon > 0$ such that the open neighbourhood of $\Gamma_\mathrm{loc}$ in $\mathcal{I}_1 \times \mathcal{I} \times J$ given by
	\[
		U_\varepsilon =\{(v^s,v^u,\mu)\in\mathcal{I}_1 \times \mathcal{I} \times J: |v^u| < \varepsilon\}
	\] 
	lies within these finitely many open neighbourhoods overing $\Gamma_\mathrm{loc}$. Uniqueness of the solutions of each $g_*$ allows one to simply consider a global function $G:U_\varepsilon \to \mathbb{R}$ which selects the appropriate local function $g_*$ and therefore satisfies $G(v^s,v^u,\mu) = 0$ if, and only if,
	\[
		(v^s,v^u,\mu) \in (W^u(0,\mu)\times \{\mu\})\cap U_\varepsilon.
	\] 
	and 
	\begin{equation}\label{GDerivative}
		\nabla_{(v^s,\mu)} G(v^s,0,\mu) \neq 0 
	\end{equation}
	for all $(v^s,0,\mu) \in \Gamma_\mathrm{loc}$. 
\end{rmk}


\subsection{On- and Off-Site Homoclinic Orbits}\label{sec:SymHom} 

We will now construct symmetric homoclinic orbits to the fixed point $u = 0$ that spend $N \gg 1$ iterations near the fixed point $u^*$. Here, by definition, a symmetric homoclinic orbit $v = \{v_n\}_{n = -\infty}^\infty$ satisfies
\begin{subequations}\label{SymHomoclinic}
	\begin{equation}\label{SymHomoclinic_Local}
		v_n = (v^s_n,v^u_n) \in \mathcal{I}\times \mathcal{I} \quad \mathrm{for} \quad n \in \{0, \dots, N\}
	\end{equation}
	\begin{equation}\label{SymHomoclinic_SigmaOut}
		v_0 \in \Sigma_\mathrm{in} \cap W^u(0,\mu)
	\end{equation}
	\begin{equation}\label{SymHomoclinic_Reversible}
		\mathcal{R}v = v
	\end{equation}
\end{subequations}   
for sufficiently large $N \gg 1$. Note that reversibility of $F$ implies that $v_N \in \Sigma_\mathrm{out}\cap W^s(0,\mu)$.

\begin{figure} 
	\centering
	\includegraphics[width=\textwidth]{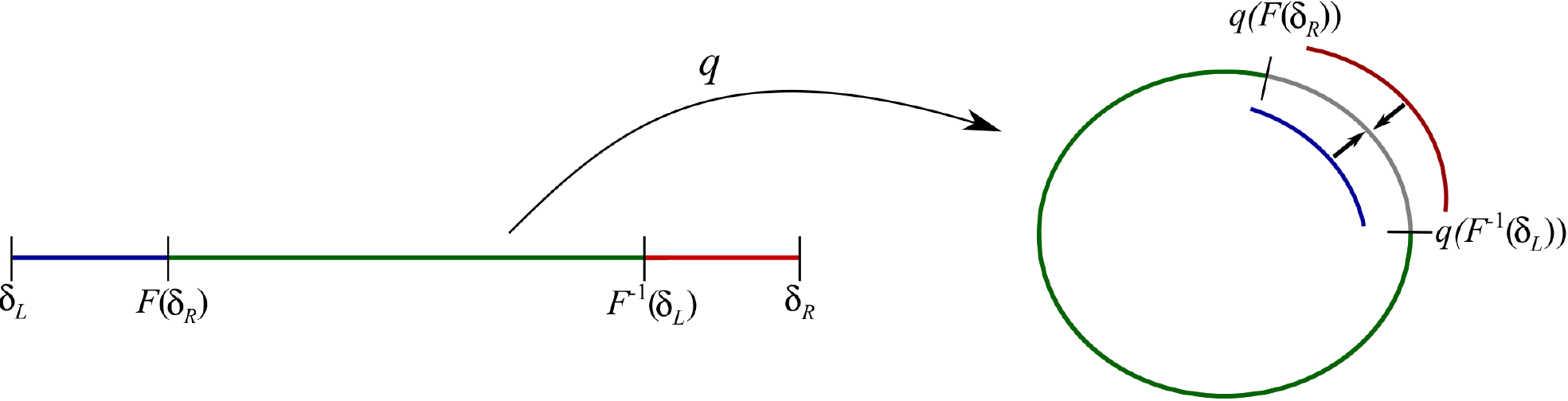}
	\caption{An illustration of the quotient map $q$ mapping $\mathcal{I}_0$ to a circle used to construct $\Gamma_\mathrm{lift}$.}
\label{fig:QuotientMap}
\end{figure}

For each $\mu \in J$, let us identify points on the same trajectory inside $\mathcal{I}_1$. Then, the resulting quotient space can be identified as the circle $S^1$ and we denote $q:\mathcal{I}_1 \times J \to S^1 \times J$ as the associated quotient map which acts as the identity between the $J$ components, illustrated in Figure~\ref{fig:QuotientMap}. Then, $q(\Gamma_\mathrm{loc}) \subset S^1 \times \mathring{J}$ is the image of $\Gamma_\mathrm{loc}$ under the quotient map $q$, which we recall from Hypothesis~\ref{hyp:Loop} is a closed loop in $S^1 \times \mathring{J}$. Take $\Gamma_\mathrm{lift}\subset \mathbb{R} \times \mathring{J}$ to be the preimage of $q(\Gamma_\mathrm{loc})$ under the natural covering projection from $\mathbb{R}\times \mathring{J}$ to $S^1\times \mathring{J}$. For clarity, we have the following correspondences between spaces: 
\[
	\mathcal{I}_1 \times J \xrightarrow{\text{\quad q\quad}} S^1 \times J \xleftarrow[\text{ covering }]{\text{ natural }} \mathbb{R}\times J.
\]
We now restate and prove Theorem~\ref{thm:MainThm}, which provides the existence and bifurcation structure of symmetric homoclinic orbits to (\ref{Map}).

\begin{thm}\label{thm:Pulses} 
	Assume Hypotheses~\ref{hyp:Reverser}-\ref{hyp:Loop} are met. There exist a number $\eta \in (0,1)$ and submanifolds $\Gamma_{1,2} \subset \mathbb{R} \times J$ such that the following is true:
	\begin{enumerate}
		\item There exists an off-site $($on-site$)$ homoclinic orbit of length $2N$ $(2N+1)$ if, and only if, there exists $s\in[0,1)$ so that $(2\pi N+s,\mu)\in\Gamma_1$ $($resp. $\Gamma_2)$.
		\item For $j = 1,2$, the manifolds $\Gamma_\mathrm{lift}$ and $\Gamma_j$ are, for each fixed $k \geq 2$, $\mathcal{O}(\eta^{2N + (j-1)})$-close to each other in the $C^k$-sense near each point in $(2\pi N+s,\mu)\in\Gamma_j$, where $s \in [0,1)$. 
	\end{enumerate}
\end{thm}

\begin{proof}
	We will prove only the existence of off-site homoclinic orbits, as the case of on-site orbits can be treated completely analogously by replacing all instances of $2N$ in this proof with $2N+1$. We now use the definition of reversible off-site homoclinic orbits given in (\ref{SymHomoclinic}) to prove the result. 
	
	Using Lemma~\ref{lem:ShilSol} we see that for an arbitrary $a^s \in \mathcal{I}$, all $\mu \in J$, and every integer $N \geq 1$, we have the existence of a reversible solution to (\ref{Map}) given by $\{v^s_n,v^u_n\}_{n=0}^{2N} \subset \mathcal{I}\times\mathcal{I}$ satisfying
	\[
		\begin{split}
			v^s_n &= v^u_{2N-n}, \\
			v^s_0 &= v^u_{2N} = a^s,	
		\end{split}
	\]
where $v^j_n = v^j_n(a^s,\mu)$ depend smoothly on $a^s \in \mathcal{I}$ and $\mu \in J$ for $j = s,u$. From Lemma~\ref{lem:ShilSol} we have that there exist $M > 0$ and $\eta \in (0,1)$ such that 
\begin{equation}\label{v^u_0Bnd}
	|v^u_0(a^s,\mu)| \leq M\eta^{2N},
\end{equation}
for all $(a^s,\mu) \in \mathcal{I}\times J$, and furthermore, the bound (\ref{v^u_0Bnd}) holds for all partial derivatives of $v^u_0(a^s,\mu)$ with respect to $(a^s,\mu)$. Therefore, the solution of (\ref{Map}) generated by $\{v^s_n,v^u_n\}_{n=0}^{2N}$ satisfies both (\ref{SymHomoclinic_Local}) and (\ref{SymHomoclinic_Reversible}). We now find that the solution $\{v^s_n,v^u_n\}_{n=0}^{2N}$ satisfies (\ref{SymHomoclinic}) if, and only if, 
	\[
		v_0 = (v^s_0,v^u_0) \in \Sigma_\mathrm{in} \cap W^u(0,\mu), 
	\]
	since the only remaining condition is (\ref{SymHomoclinic_SigmaOut}). 
	
	Take $\varepsilon > 0$ to be the constant in Remark~\ref{rmk:GFunction} following Lemma~\ref{lem:GammaLoc} and $N_0 \geq 1$ sufficiently large so that for all $N \geq N_0$ we can use (\ref{v^u_0Bnd}) to guarantee
	\[
		|v^u_0(a^s,\mu)| < \varepsilon,
	\] 
	for all $(a^s,\mu) \in \mathcal{I}\times J$. Therefore, using Lemma~\ref{lem:GammaLoc} and Remark~\ref{rmk:GFunction} we find that satisfying the condition (\ref{SymHomoclinic_Local}) is equivalent to solving   
	\begin{equation}\label{Matching}
		G(v^s_0(a^s,\mu),v^u_0(a^s,\mu),\mu) = G(a^s,\mathcal{O}(\eta^{2N}),\mu) = 0,
	\end{equation}
	since $v^u_0(a^s,\mu) = \mathcal{O}(\eta^{2N})$ from (\ref{v^u_0Bnd}). Now, from our definition of the quotient mapping $q$, it follows from Hypothesis~\ref{hyp:Loop} that $q(\Gamma_\mathrm{loc})$ is a closed loop belonging to $S^1\times\mathring{J}$. Parametrize this loop as
	\[
		q(\Gamma_\mathrm{loc}) = \{(\theta(t),\nu(t)): t\in[0,1]\}	
	\]
	so that $(\theta(1),\nu(1)) = (\theta(0),\nu(0))$. Let $\tilde{\theta}(t)\in\mathcal{I}_0$ be such that for all $t\in[0,1]$ we have
	\[
		q(\tilde{\theta}(t),\nu(s)) = (\theta(t),\nu(t))
	\]
	and $F((\tilde{\theta}(t),0),\nu(t)) \notin\mathcal{I}_0$ for all $t \in [0,1]$. That is, $\tilde{\theta}(t)$ is the rightmost preimage of $\theta(t)$ in $\mathcal{I}_0$ for each $\mu = \nu(t)$. Next, let 
	\[
		n(t) = (n_1(t), n_2(t)) := \frac{1}{|\nabla_{(v^s,\mu)} G(\tilde{\theta}(t),0,\nu(t))|} \nabla_{(v^s,\mu)}G(\tilde{\theta}(t),0,\nu(t)),
	\] 
	which is well-defined from (\ref{GDerivative}). We now set 
	\[
		(a^s,\mu) = (\tilde{\theta}(t) + n_1(t)b, \nu(t) + n_2(t)b),
	\] 
	and let 
	\[
		H(t,b) := G(a^s,\mathcal{O}(\eta^{2N}),\mu) = G(\tilde{\theta}(t) + n_1(t)b,\mathcal{O}(\eta^{2N}),\nu(t) + n_2(t)b). 
	\]
	We therefore have 
	\[
		H(t,b) = \mathcal{O}(|b| + \eta^{2N}), \quad H_b(t,b) = 1 + \mathcal{O}(\eta^{2N}).
	\]
	In particular, we can apply the contraction mapping theorem to solve $H(t,b) = 0$ uniformly in $t \in [0,1]$ for all sufficiently large $N \gg 1$. This gives the existence of a solution to (\ref{Matching}), written $(a^s_*,\mu_*)(t,N)$, which satisfies
	\[
		\begin{split}
			a^s_*(t,N) &= \tilde{\theta}(t) + \mathcal{O}(\eta^{2N}), \\
			\mu_*(t,N) &= \nu(t) + \mathcal{O}(\eta^{2N}).	
		\end{split}	
	\]
	In conclusion, the manifold $\Gamma_1$ is given by 
	\[
		\Gamma_1 := \{2\pi N + q(a^s_*(t,N),\mu_*(t,N)):\ t\in [0,1]\},
	\]
	and the closeness result now follows from the smoothness of the quotient map $q$. This concludes the proof.
\end{proof}


\subsection{Asymmetric Homoclinic Orbits}\label{sec:AsymHom} 

We now focus on asymmetric ($\mathcal{R}v \neq v$) homoclinic orbits, which, by definition satisfy
\begin{subequations}\label{AsymHomoclinic}
	\begin{equation}\label{AsymHomoclinic_Local}
		v_n = (v^s_n,v^u_n) \in \mathcal{I}\times \mathcal{I} \quad \mathrm{for} \quad n \in \{0, \dots, N\}
	\end{equation}
	\begin{equation}\label{Asym_SigmaIn}
		v_0 \in \Sigma_\mathrm{in} \cap W^u(0,\mu)
	\end{equation}
	\begin{equation}\label{Asym_SigmaOut}
		v_N \in \Sigma_\mathrm{out} \cap W^s(0,\mu)
	\end{equation}
\end{subequations}
for sufficiently large $N \gg 1$. First, Lemma~\ref{lem:ShilSol} implies that for $N \geq 1$ sufficiently large, arbitrary $a^s,a^u \in \mathcal{I}$, and every $\mu \in J$ we may construct a solution $v = \{v_n(a^s,a^u,\mu)\}_{n=0}^N\subset \mathcal{I}\times\mathcal{I}$ smoothly depending on $(a^s,a^u,\mu)$ satisfying (\ref{AsymHomoclinic_Local}) such that
	\[
		v^s_{0}(a^s,a^u,\mu) = a^s, \quad v^u_{N}(a^s,a^u,\mu) = a^u,
	\] 
	and there exists some $\eta \in (0,1)$ such that
	\begin{equation}\label{ShilBnds2}
		|v^s_{n}(a^s,a^u,\mu)| \leq M\eta^{n}, \quad |v^u_{n}(a^s,a^u,\mu)| \leq M\eta^{N- n},
	\end{equation} 
for all $n \in \{0,\dots,N\}$. We note that Lemma~\ref{lem:ShilSol} dictates that the bounds (\ref{ShilBnds2}) also hold for all partial derivatives with respect to $(a^s,a^u,\mu)$. 

Taking $N \gg 1$ sufficiently large so that $M\eta^{N} < \varepsilon$, where $\varepsilon > 0$ is the constant required for Remark~\ref{rmk:GFunction}, we conclude as in the proof of Theorem~\ref{thm:Pulses} that $v = \{v_n\}_{n=0}^N$ satisfies (\ref{Asym_SigmaIn}) for some $\mu \in J$ if, and only if, 
\[
	G(a^s,v^u_{0}(a^s,a^u,\mu),\mu) = 0.
\]
Furthermore, applying the reverser $\mathcal{R}$ we have that $v = \{v_n(a^s,a^u,\mu)\}_{n=0}^N$ satisfies (\ref{Asym_SigmaOut}) for some $\mu \in J$ if, and only if, 
\[
	  G(a^u,v^s_{N}(a^s,a^u,\mu),\mu) = 0.
\]
Therefore, we see that satisfying (\ref{AsymHomoclinic}) is now equivalent to solving 
\[
	\mathcal{H}(a^s,a^u,\mu) = \begin{pmatrix}
		G(a^s,v^u_{0}(a^s,a^u,\mu),\mu) \\
		G(a^u,v^s_{N}(a^s,a^u,\mu),\mu)
	\end{pmatrix} = 0.
\]
Denoting by $\kappa$ the map given by $\kappa(a_1,a_2) = (a_2,a_1)$, the action of the reverser $\mathcal{R}$ in (\ref{ShilReverser}) gives  
\[
	\mathcal{H}(\kappa(a^s,a^u),\mu) = \mathcal{H}(a^u,a^s,\mu) = \begin{pmatrix}
		G(a^u,v^u_{0}(a^u,a^s,\mu),\mu) \\
		G(a^s,v^s_{N}(a^u,a^s,\mu),\mu)
	\end{pmatrix}  = \kappa\mathcal{H}(a^s,a^u,\mu),
\]
and hence $\mathcal{H}$ is $\kappa$-equivariant for each $\mu \in J$. In particular, we have that roots of $\mathcal{H}$ with $a^s = a^u$ which are fixed by the action of $\kappa$ are exactly the symmetric homoclinic orbits constructed in the previous subsection, and solutions which have $a^s \neq a^u$ come in pairs and are mapped into each other by $\kappa$. This action of mapping roots of $\mathcal{H}$ into each other using $\kappa$ is equivalent to mapping homoclinic orbits of (\ref{Map}) into each other by $\mathcal{R}$.

Using (\ref{ShilBnds2}) we have that
\[
	\mathcal{H}(a^s,a^u,\mu) = \begin{pmatrix}
		G(a^s,\mathcal{O}(\eta^N),\mu) \\
		G(a^u,\mathcal{O}(\eta^N),\mu)
	\end{pmatrix} = \begin{pmatrix}
		G(a^s,0,\mu) \\
		G(a^u,0,\mu)
	\end{pmatrix} + \mathcal{O}(\eta^N).	
\]
Therefore, upon solving $G(a^s,0,\mu) = G(a^u,0,\mu) = 0$ for some $(a^s,a^u,\mu)$ an application of the contraction mapping theorem as in Theorem~\ref{thm:Pulses} can be used to extend this solution to one which satisfies $\mathcal{H}(a^s,a^u,\mu) = 0$ for all sufficiently large $N \geq 1$. Hence, for the remainder of this section we will introduce the slight abuse of notation by simply considering 
\[
	\mathcal{H}(a^s,a^u,\mu) = \begin{pmatrix}
		G(a^s,0,\mu) \\
		G(a^u,0,\mu)
	\end{pmatrix},	
\]  
and solving $\mathcal{H}(a^s,a^u,\mu) = 0$. Finally, solving $\mathcal{H}(a^s,a^u,\mu) = 0$ can be done equivalently by obtaining roots of 
\begin{equation}\label{AsymMatch}
	\tilde{\mathcal{H}}(a^s,a^u,\mu) = \begin{pmatrix}
		G_1(a^s,a^u,\mu) \\
		G_2(a^s,a^u,\mu)
	\end{pmatrix} := \begin{pmatrix}
		G(a^s,0,\mu) + G(a^u,0,\mu) \\
		G(a^s,0,\mu) - G(a^u,0,\mu)
	\end{pmatrix}, 
\end{equation}
and note that (\ref{AsymMatch}) is $\mathbb{Z}_2$-symmetric under the action $(a^s,a^u) \to (a^u,a^s)$ and $(G_1,G_2) \mapsto (G_1,-G_2)$. We now present the following result which shows that asymmetric homoclinic orbits bifurcate from symmetric homoclinic orbits.

\begin{lem} \label{lem:AsymPitchfork} 
	Assume Hypotheses~\ref{hyp:Reverser}-\ref{hyp:Loop} are met. Assume that at $\mu_0 \in \mathring{J}$ the manifolds $W^s(u^*,\mu_0)$ and $W^u(0,\mu_0)$ intersect along a quadratic tangency. Then for each $N \geq 1$ sufficiently large, precisely two branches of asymmetric homoclinic orbits (mapped into each other by $\mathcal{R}$) bifurcate from the symmetric homoclinic orbit of length $N$ corresponding to $\mu_0$.    
\end{lem}

\begin{proof}
	Let $(a_0,\mu_0) \in \Gamma_\mathrm{loc}$, where $\mu_0 \in J$ is as stated in the lemma. Then, from the constructions in the proof of Lemma~\ref{lem:GammaLoc} we have 
	\[
		G(a_0,0,\mu_0) = \partial_{v^s}G(a_0,0,\mu_0) = 0, \quad \partial^2_{v^s}G(a_0,0,\mu_0)\cdot\partial_\mu G(a_0,0,\mu_0) \neq 0. 
	\]
	Since $\partial_\mu G(a_0,0,\mu) \neq 0$ we have that 
	\[
		\partial_\mu G_1(a_0,a_0,\mu) = 2\partial_\mu G(a_0,0,\mu) \neq 0,
	\] 
	and therefore the implicit function theorem provides that we can solve $G_1(a^s,a^u,\mu) = 0$ near $(a_0,a_0,\mu_0)$ uniquely for $\mu = \mu_*(a^s,a^u)$ as a function of $(a^s,a^u)$. 
	
	Now, introduce the invertible transformation 
	\[
		\begin{pmatrix}
			b_1 \\
			b_2
		\end{pmatrix} = \begin{pmatrix}
			a^s - a^u \\
			a^s + a^u - 2a_0
		\end{pmatrix} 
	\] 
	and note that the $\mathbb{Z}_2$-symmetry guarantees that $G_2$ is odd in $b_1$. Hence, we introduce the smooth function $\tilde{G}_2(b_1,b_2)$ so that 
	\[
		G_2\bigg(a_0 + \frac{1}{2}(b_1 + b_2),a_0 + \frac{1}{2}(b_2 - b_1),\tilde{\mu}_*(b_1,b_2)\bigg) = b_1\tilde{G}_2(b_1,b_2),
	\]
	where for the ease of notation 
	\[
		\tilde{\mu}_*(b_1,b_2) := \mu_*\bigg(a_0 + \frac{1}{2}(b_1 + b_2),a_0 + \frac{1}{2}(b_2 - b_1)\bigg).
	\] 
	Note that $\tilde{G}_2(b_1,b_2)$ is even in $b_1$, and hence expanding $\tilde{G}_2(b_1,b_2)$ near $(b_1,b_2) = (0,0)$ ones finds that
	\[
		\tilde{G}_2(b_1,b_2) = 2\partial^2_{v^s}G(a_0,0,\mu_0)b_2 + \mathcal{O}(b_1^2 + b_2^2), 
	\]   
	where $\partial^2_{v^s}G(a_0,0,\mu_0) \neq 0$. Therefore, the implicit function theorem guarantees that we can solve $\tilde{G}_2(b_1,b_2) = 0$ near $(b_1,b_2) = (0,0)$ uniquely for $b_2$ as a function of $b_1$. This proves the claim.
\end{proof} 

\begin{rmk}\label{rmk:Asym} 
By definition it is possible to have multiple points of the same heteroclinic orbit belonging to $\mathcal{I}_0$ for any $\mu \in J$. Suppose that $(a_1,\mu),(a_2,\mu) \in \Gamma_\mathrm{loc}$ are such that $F$ maps $(a_1,0)$ to $(a_2,0)$ at the parameter value $\mu$. Local uniqueness of solutions to $\mathcal{H} = 0$ gives that the solutions of (\ref{Shilnikov}) from Lemma~\ref{lem:ShilSol} of the form 
	\begin{equation}\label{AsymJump}
		\begin{split}
			v_1 &= \{v_n(a_1,a^u_0,\mu)\}_{n=0}^{N+1}\subset \mathcal{I}\times\mathcal{I}, \\ 
			v_0 &= \{v_n(a_2,a^u_0,\mu)\}_{n=0}^{N}\subset \mathcal{I}\times\mathcal{I},  
		\end{split}   
	\end{equation}   
for some $a^u\in\mathcal{I}$, correspond to exactly the same homoclinic orbit of (\ref{Map}). This redundancy which arises due to the definition of $\mathcal{I}_0$ is eradicated by moving to the quotient space $S^1 \times J$ since we have $q(a_1,\mu) = q(a_2,\mu)$.   
\end{rmk}

Remark~\ref{rmk:Asym} allows one to define the following sets
\[
	\begin{split}
		\Lambda_s &:= \{(a^s,a^u,\mu) \in \mathcal{I}_0\times\mathcal{I}_0\times J: q(a^s,\mu) = q(a^u,\mu),\ G(a^s,0,\mu) = 0\} \\
		\Lambda_s^\mathrm{bif} &:= \{(a^s,a^u,\mu) \in \mathcal{I}_0\times\mathcal{I}_0\times J: q(a^s,\mu) = q(a^u,\mu),\ G(a^s,0,\mu) = 0,\ \partial_{v^s}G(a^s,0,\mu) = 0\} \\
		\Lambda_a &:= \{(a^s,a^u,\mu) \in \mathcal{I}_0\times\mathcal{I}_0\times J: q(a^s,\mu) \neq q(a^u,\mu),\ G(a^s,0,\mu) = G(a^u,0,\mu) = 0\} \\
	\end{split}
\]
to be the sets of symmetric homoclinic orbits, symmetric homoclinic orbits at pitchfork bifurcations, and asymmetric homoclinic orbits. To characterize the set $\Lambda_a$ we require the following non-degeneracy hypothesis.

\begin{hyp}\label{hyp:Asym} 
	If $(a,\mu) \in \Gamma_\mathrm{loc}$ is such that $\partial_{v^s}G(a,0,\mu) = 0$, then $\partial_{v^s}G(\tilde{a},0,\mu) \neq 0$ for all $(\tilde{a},\mu) \in \Gamma_\mathrm{loc}$ with $q(\tilde{a},\mu) \neq q(a,\mu)$.  
\end{hyp}

\begin{lem}\label{lem:Lambda_a} 
	Assume Hypotheses~\ref{hyp:Reverser}-\ref{hyp:Asym} are met. The bifurcation curves of each asymmetric homoclinic orbit in $\ell^\infty\times J$ of (\ref{Map}) is either a smooth isola or a smooth curve with boundaries given by the pitchfork bifurcations described in Lemma~\ref{lem:AsymPitchfork}. 	
\end{lem}

\begin{proof}
	Lemma~\ref{lem:AsymPitchfork} has already shown that precisely two branches of asymmetric homoclinic orbits bifurcate from each point in the set $\Lambda_s^\mathrm{bif}$. Then, taking any $(a^s,a^u,\mu) \in \Lambda_a$ we have $\mathcal{H}(a^s,a^u,\mu) = 0$ and 
	\begin{equation}\label{Jacobian}
		D\mathcal{H}(a^s,a^u,\mu) = \begin{bmatrix}
			\partial_{v^s}G(a^s,0,\mu) & 0 & \partial_\mu G(a^s,0,\mu) \\
			0 & \partial_{v^s}G(a^u,0,\mu) & \partial_\mu G(a^u,0,\mu)
		\end{bmatrix}.
	\end{equation}
	Now, when $\partial_{v^s}G(a^s,0,\mu)\cdot\partial_{v^s}G(a^u,0,\mu) \neq 0$ the matrix $D\mathcal{H}(a^s,a^u,\mu)$ has full rank. Should $\partial_{v^s}G(a^s,0,\mu) = 0$, Hypothesis~\ref{hyp:Gamma} gives that $\partial_\mu G(a^s,0,\mu) \neq 0$. Furthermore, Hypothesis~\ref{hyp:Asym} guarantees that if $\partial_{v^s}G(a^u,0,\mu) = 0$ as well, then we must be at a pitchfork bifurcation. Since we have assumed that $(a^s,a^u,\mu) \in \Lambda_a$, we have that $D\mathcal{H}(a^s,a^u,\mu)$ has full rank. A similar argument shows that when $\partial_{v^s}G(a^u,0,\mu) = 0$ the matrix $D\mathcal{H}(a^s,a^u,\mu)$ again has full rank, and therefore $D\mathcal{H}(a^s,a^u,\mu)$ has full rank for all $(a^s,a^u,\mu) \in \Lambda_a$. This gives that the solutions of $\mathcal{H}(a^s,a^u,\mu) = 0$ is given locally by smooth curves.     
	
	From the fact that $\mathcal{I}_0\times\mathcal{I}_0\times J$ is compact, we have that 
	\[
		\Lambda_a = \mathcal{H}^{-1}(0,0)\setminus\Lambda_s
	\] 
	is composed of finitely many smooth curves. Following along a single curve in $\Lambda_a \cup \Lambda_s^\mathrm{bif}$ we must have that it either i) forms a closed loop in the interior of $\mathcal{I}_0\times\mathcal{I}_0\times J$, or ii) meets the boundary of $\mathcal{I}_0\times\mathcal{I}_0\times J$. If all curves fall into the former case then we are finished. We now focus on the latter case and note that Hypothesis~\ref{hyp:Gamma} implies that curves in $\Lambda_a \cup \Lambda_s^\mathrm{bif}$ may only intersect the boundary in the first two components since $\Gamma_\mathrm{loc} \subset \mathcal{I}_0\times \{0\} \times \mathring{J}$. 
	
	Suppose we have a curve $\Lambda_a \cup \Lambda_s^\mathrm{bif}$ which intersects the boundary of $\mathcal{I}_0\times\mathcal{I}_0\times J$. Let this point of intersection be given by $(a^s_0,a^u_0,\mu_0) \in \mathcal{I}_0\times\mathcal{I}_0\times \mathring{J}$ so that 
	\[
		G(a^s_0,0,\mu_0) = G(a^u_0,0,\mu_0) = 0
	\] 
	and either $a^s_0 \in \{\delta_L,\delta_R\}$ or $a^u_0 \in \{\delta_L,\delta_R\}$. Let us assume that $a^s_0 = \in \{\delta_L,\delta_R\}$ since the other cases follow by exactly the same arguments. By definition of $\mathcal{I}_0 = [\delta_L,\delta_R]$ we have that there exists $a^s_1 \in (\delta_L,\delta_R)$ such that $q(a^s_1,\mu_0) = q(a^s_0,\mu_0)$, and as in Remark~\ref{rmk:Asym} we have that the homoclinic orbits (\ref{AsymJump}) generated by solutions of (\ref{Shilnikov}) from Lemma~\ref{lem:ShilSol} using parameters $(a^s_0,a^u_0,\mu_0)$ and $(a^s_1,a^u_0,\mu_0)$ are exactly the same. Therefore, the curve of asymmetric homoclinic orbits which intersected the boundary $\mathcal{I}_0\times\mathcal{I}_0\times J$ at $(a^s_0,a^u_0,\mu_0)$ can be continued in $\ell^\infty\times J$ by following the curve containing $(a^s_1,a^u_0,\mu) \in \Lambda_a \cup \Lambda_s^\mathrm{bif}$. We may follow this new curve in exactly the same way by tracking where it maps to if/when it intersects with the boundary of $\mathcal{I}_0\times\mathcal{I}_0\times J$ by jumping to another curve in $\Lambda_a \cup \Lambda_s^\mathrm{bif}$. Since $\Lambda_a \cup \Lambda_s^\mathrm{bif}$ contains only finitely many curves, we must have that this process eventually comes back to the point $(a^s_0,a^u_0,\mu_0)$ (see Figure~\ref{fig:AsymIsolas} for a visual depiction). This concludes the proof.   
\end{proof} 

\begin{figure} 
	\centering
	\includegraphics[width=0.4\textwidth]{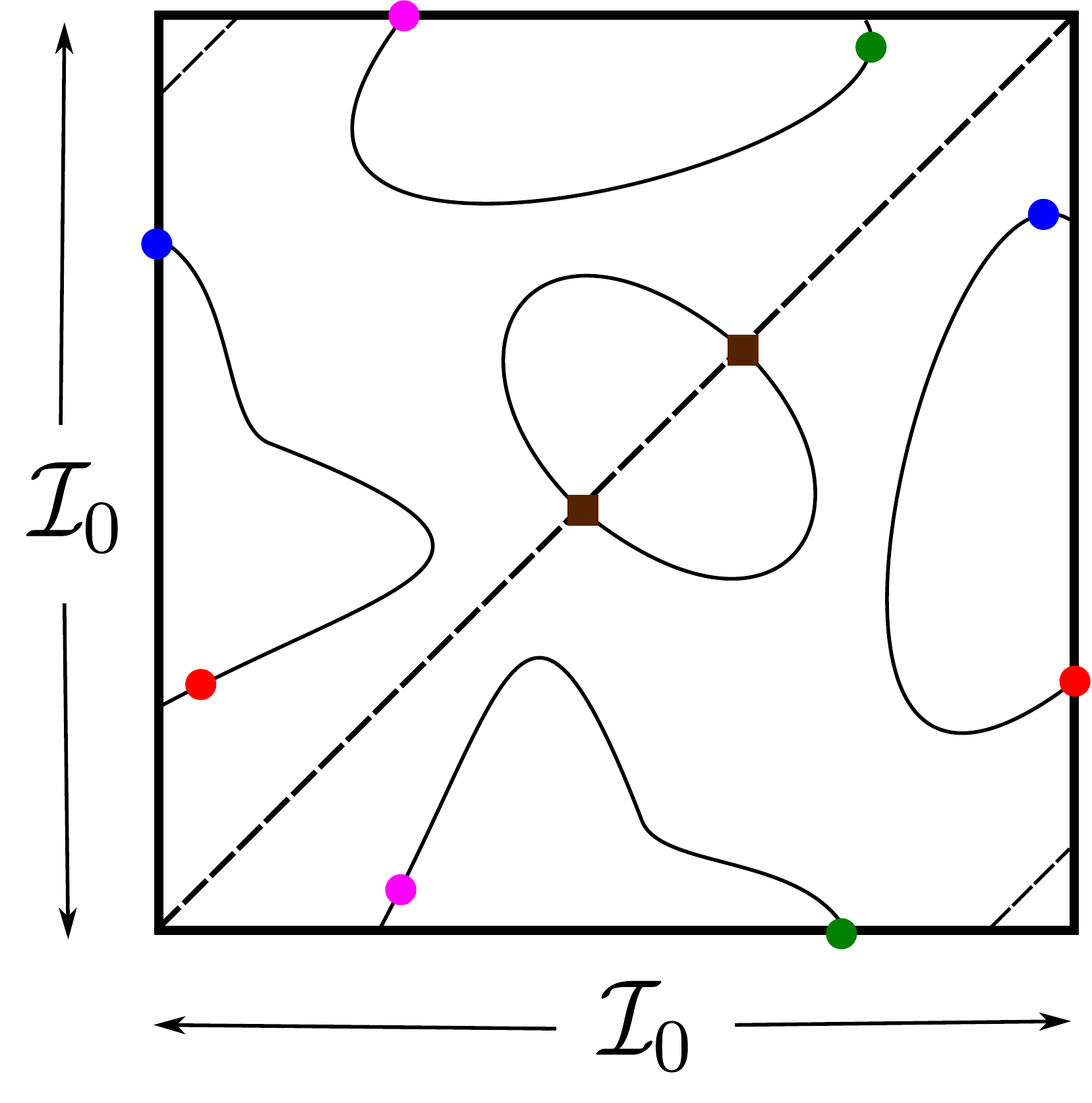}
	\caption{A cartoon of the projection of 0 level set of $\mathcal{H}$ restricted to $\mathcal{I}_0 \times \mathcal{I}_0 \times J$ and projected onto the first two components. The dashed diagonal lines represent $\Lambda_s$, which correspond to symmetric homoclinic orbits. Brown squares represents elements of the discrete set $\Lambda_a^\mathrm{bif}$, and the solid curves are the set $\Lambda_a$. Dots of the same colour represent the same asymmetric homoclinic orbit, as argued in the proof of Lemma~\ref{lem:Lambda_a}. This figure has a $(a^s,a^u) \mapsto (a^u,a^s)$ symmetry coming from the equivariance of $\mathcal{H}$ with respect to this action for each $\mu \in J$.}
\label{fig:AsymIsolas}
\end{figure} 

Our final result in this section shows that branches of asymmetric homoclinic orbits which originate and terminate at pitchfork bifurcations begin and end at points in $\Gamma_\mathrm{loc}$ of opposite curvature in $\mu$. Let us take some point $(a^s_0,0,\mu_0) \in \Gamma_\mathrm{loc}$ such that $G(a^s_0,0,\mu_0) = \ \partial_{v^s}G(a^s_0,0,\mu_0) = 0$. Then, from Hypothesis~\ref{hyp:Gamma} we necessarily have that $\partial_{\mu}G(a^s,0,\mu) \neq 0$, and therefore we can uniquely parametrize $\Gamma_\mathrm{loc}$ locally near $(a^s_0,0,\mu_0)$ as $\mu = \mu_*(a^s)$ so that $G(a^s,0,\mu_*(a^s)) = 0$ for all $a^s$ sufficiently close to $a^s_0$. Taking derivatives, we find that 
\[	
	\mathrm{sign}\mu_*''(a^s_0) = -\mathrm{sign}(\partial_{\mu}G(a^s_0,0,\mu_0)\partial^2_{v^s}G(a^s_0,0,\mu_0)),
\] 
where we recall that Hypothesis~\ref{hyp:Gamma} implies that $\partial^2_{v^s}G(a^s_0,0,\mu_0) \neq 0$. This leads to the following lemma. 

\begin{lem}\label{lem:Asym3} 
	Assume Hypotheses~\ref{hyp:Reverser}-\ref{hyp:Asym} are met. The branches of asymmetric homoclinic orbits described in Lemma~\ref{lem:Lambda_a} begin and end at points in $\Lambda_s^\mathrm{bif}$ at which $\partial_{\mu}G(a^s,0,\mu)\partial^2_{v^s}G(a^s,0,\mu)$ has opposite sign. 	
\end{lem}

\begin{proof}
	Let $(a^s(t),a^u(t),\mu(t))$ be a curve parametrized by $t \in [0,1]$ such that $(a^s,a^u,\mu)(t)\in \Lambda_s^\mathrm{bif}$ for $t = 0,1$, $(a^s,a^u,\mu)(t)\in \Lambda_a$ for all $t \in (0,1)$, $q(a^s(t),\mu(t))$ and $q(a^u(t),\mu(t))$ are continuous curves in $S^1\times J$, and $\mu(t)$ is a smooth in $t$. We note that from the proof of Lemma~\ref{lem:Lambda_a} that $a^s(t)$ and $a^u(t)$ need not be continuous since they could jump at the boundary, but the continuity of $q(a^s(t),\mu(t))$ and $q(a^u(t),\mu(t))$ guarantee that we have parametrized a single curve of asymmetric homoclinic orbits which originates and terminates at pitchfork bifurcations. Moreover, for each fixed $t_0 \in [0,1]$ we can locally parametrize a smooth curve $(b^s,b^u,\mu)(t) \in \Lambda_s^\mathrm{bif}$ such that $q(b^s(t),\mu(t)) = q(a^s(t),\mu(t))$ and $q(b^u(t),\mu(t)) = q(a^u(t),\mu(t))$ for all $t$ sufficiently close to $t_0$. Hence, these local parameterizations allow us to slightly abuse terminology and simply say that $(a^s,a^u,\mu)(t)$ is smooth in $t$, along with the added stipulation that $a^s(0) = a^u(0)$ and $a^s(1) = a^u(1)$ since at $t = 0,1$ we are at pitchfork bifurcations. We wish to show that $\mu'(0)\cdot\mu'(1) < 0$.  
	
	Now, since we have pitchforks at $t=0,1$ we necessarily have $(a^s)'(t)\cdot(a^u)'(t) < 0$ at $t = 0,1$ and $\mu'(0) = \mu'(1) = 0$. Let us now assume that $\mu''(0)\cdot\mu''(1) > 0$, and derive a contradiction. First, we note that there must exist $t_0 \in (0,1)$ such that $\mu'(t_0) = 0$ and $\mu''(0)\cdot\mu''(t_0) < 0$ (or equivalently $\mu''(t_0)\mu''(1) < 0$). Since $\mu'(t_0) = 0$, we must have that $\partial_{v^s}G(a^s(t_0),0\mu(t_0)) = 0$ or $\partial_{v^s}G(a^u(t_0),0\mu(t_0)) = 0$, and Hypothesis~\ref{hyp:Asym} dictates that both cannot be simultaneously true for $t_0 \in (0,1)$. Noting that the vector $(a^s,a^u,\mu)'(t)$ belongs to the null space of the Jacobian (\ref{Jacobian}), the action of $\mathcal{R}$ allows one to assume without loss of generality that $(a^s)'(t_0) \neq 0$, and hence Hypothesis~\ref{hyp:Asym} dictates that $(a^u)'(t_0) = 0$. In the case that $t_0 \in (0,1)$ is the only such value for which $\mu'(t) = 0$ for $t \in(0,1)$, then we have that $(a^s)'(t)$ never changes sign for all $t \in [0,1]$ and that $(a^u)'(0)\cdot(a^u)'(1) < 0$ since $(a^u)'(t)$ changes sign at $t = t_0$. But, this contradicts the assumption that $(a^s)'(t)\cdot(a^u)'(t) < 0$ at $t = 0,1$, showing that if $t_0\in(0,1)$ is the only such value for which $\mu'(t) = 0$ for $t \in (0,1)$ we must have $\mu'(0)\cdot\mu'(1) < 0$. This argument easily extends to the case when $\mu'(t)$ has multiple roots in the interval $(0,1)$ since a simple argument shows that there must be a finite and odd number of such roots, causing one of $(a^s)'(t)$ and $(a^u)'(t)$ to change sign an odd number of times and again leading to a contradiction. This completes the proof.
\end{proof} 

The results of Lemmas~\ref{lem:AsymPitchfork}-\ref{lem:Asym3} therefore prove the claims of Theorem~\ref{thm:AsymThm}. This concludes our theoretical analysis.


\section{Application to Lattice Dynamical Systems}\label{sec:LDS} 

We now return to the lattice dynamical system introduced in \S\ref{s1} given by
\begin{equation} \label{LDS}
	\dot{U}_n = d(U_{n+1} + U_{n-1} - 2U_n) - \mu U_n + 2U_n^3 - U_n^5
\end{equation}
posed on the one-dimensional integer lattice $\mathbb{Z}$, where $d > 0$ represents the strength of coupling between nearest neighbours. We take $\mu$ to be a bifurcation parameter and restrict our attention to $\mu \in [0,1]$, since in this parameter region we have exactly five spatially independent steady-state solutions of (\ref{LDS}) given by $U_n = 0$ and $U_n = \pm U_\pm(\mu)$, where
\[
	U_\pm(\mu) = \sqrt{1\pm \sqrt{1-\mu}}
\]
for all $n \in \mathbb{Z}$. The equilibria $\pm U_-(\mu)$ collide with the trivial equilibrium in a pitchfork bifurcation at $\mu = 0$, and at $\mu =1$ the states $U_+(\mu)$ and $U_-(\mu)$ (and also $-U_+(\mu)$ and $-U_-(\mu)$) meet in a saddle-node bifurcation. We illustrate the bifurcation diagram of these spatially independent steady-states in Figure~\ref{fig:LDS_Steady_States}. Much of this work will focus on the nonnegative equilibria 0 and $U_\pm(\mu)$.  

\begin{figure} 
	\centering
	\includegraphics[width=0.4\textwidth]{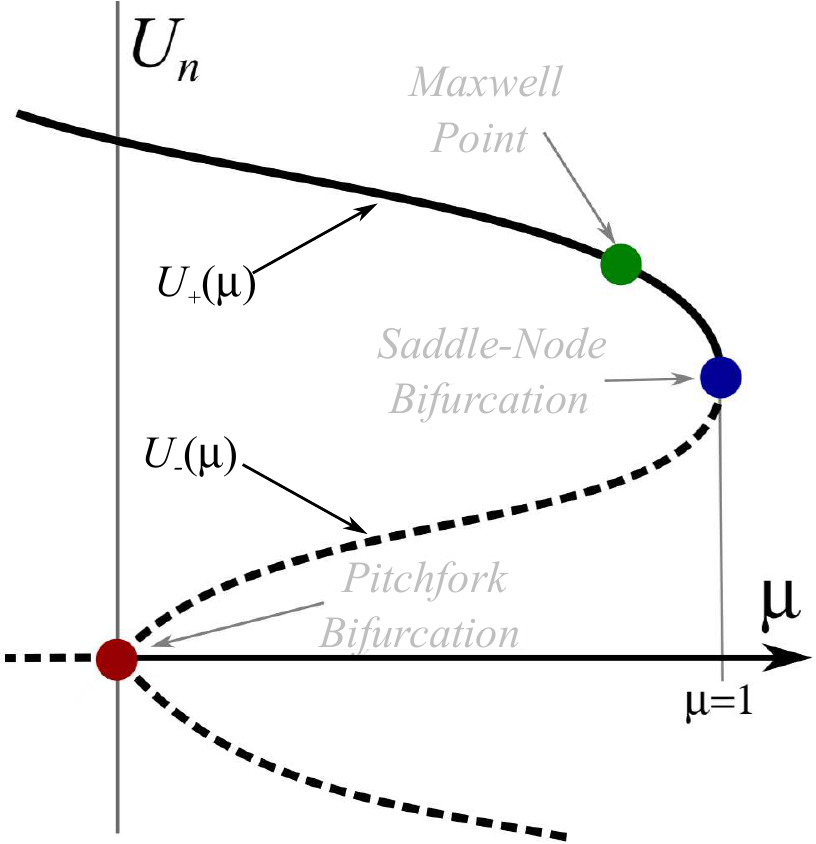}
	\caption{A bifurcation diagram of the spatially independent steady-state solutions of (\ref{LDS}). Stable states are given by solid curves, whereas unstable states are given by dashed curves. The diagram has a $U_n \to -U_n$ symmetry.}
\label{fig:LDS_Steady_States}
\end{figure} 

The lattice equation (\ref{LDS}) was studied in \cite{Taylor}, and we now briefly comment on some of their findings. First, system (\ref{LDS}) is invariant under the `staggering' symmetry given by the transformation
\begin{equation}\label{StaggeringSym}
	U_n \to (-1)^nU_n,\quad d \to -d,\quad \mu\to \mu - 4d.
\end{equation}
Secondly, equation (\ref{LDS}) is a gradient flow on the space $\ell^2$, given by
\[
	\ell^2 = \bigg\{x = \{x_n\}_{n\in\mathbb{Z}}:\ \sum_{n\in\mathbb{Z}} |x_n|^2 < \infty\bigg\}.
\]
That is, we can write $\dot{U}_n = -\partial \mathcal{E}/\partial U_n$, where the potential $\mathcal{E}:\ell^2 \to \mathbb{R}$ is given by
\[
	\mathcal{E}(U) = \sum_{n\in\mathbb{Z}} \bigg(\frac{d}{2}(U_{n+1}-U_n)^2 + \frac{1}{2}\mu U_n^2 - \frac{1}{2}U_n^4 + \frac{1}{6}U_n^6\bigg).
\]
Note that $\dot{\mathcal{E}} \leq 0$, so that every solution of (\ref{LDS}) with initial condition belonging to $\ell^2$ evolves towards an equilibrium solution as $t \to \infty$. Finally, the per-cell potential for the homogeneous equilibria $U_n = U_*$ for each $n \in \mathbb{Z}$ is given by $\mathcal{E}(U_*) = \frac{1}{2}\mu U_*^2 - \frac{1}{2}U_*^4 + \frac{1}{6}U_*^6$. Clearly the zero state, $U_* = 0$, has zero potential, and the potential of the upper state $U = U_+(\mu)$ depends on $\mu$. The point at which the upper state has zero potential is referred to as the {\em Maxwell Point}, and it takes place at $\mu = \mu_\mathrm{mx} = 0.75$.   

Obtaining steady-state solutions of (\ref{LDS}) requires satisfying 
\begin{equation}\label{LDS_Steady}
	0 = d(U_{n+1} + U_{n-1} - 2U_n) - \mu U_n + 2U_n^3 - U_n^5, 
\end{equation}
for all $n \in \mathbb{Z}$. As detailed in the introduction, letting $u_n = U_{n-1}$ and $v_n = U_n$ in (\ref{LDS_Steady}) we obtain the map
\begin{equation}\label{LDSMap}
	\begin{split}
		u_{n+1} &= v_n, \\
		v_{n+1} &= 2v_n - u_n +\frac{1}{d}(\mu v_n - 2v_n^3 + v_n^5),
	\end{split}
\end{equation}
of the form (\ref{Map}) studied in this paper, whose bounded solutions correspond to solutions of (\ref{LDS_Steady}). It is easy to see that the right-hand side of (\ref{LDSMap}) is indeed a diffeomorphism and satisfies Hypothesis~\ref{hyp:Reverser} when we define $\mathcal{R}$ via
\[
	\mathcal{R}[u,v]^T = [v,u]^T.
\]
Moreover, the fixed points $(u,v) = (0,0), (U_\pm(\mu),U_\pm(\mu))$ belong to ${\rm Fix}(\mathcal{R})$, are hyperbolic for all $\mu \in (0,1)$, and therefore satisfy Hypothesis~\ref{hyp:FixedPts} for any closed interval $J \subset (0,1)$. Our interest now lies in understanding heteroclinic connections between the fixed points $(0,0)$ and $(U_+(\mu),U_+(\mu))$ as this will then allow us to apply the results of the previous sections to accurately describe bifurcating localized solutions of (\ref{LDS}).

Over the coming subsections we will see that (\ref{LDS}) is ideal for analytically confirming the hypotheses required to applying the results of our theoretical analysis. In $\S\ref{subsec:Flat}$ we will show that the snaking bifurcation structure shown in Figure~\ref{fig:Snaking_1D} is a consequence of a $1$-loop of heteroclinic orbits connecting the fixed points $(0,0)$ and $(U_+(\mu),U_+(\mu))$ of (\ref{LDSMap}). In $\S\ref{subsec:Oscillatory}$ we exploit the staggering symmetry (\ref{StaggeringSym}) to show that heteroclinic orbits connecting $(0,0)$ to a periodic orbit of (\ref{LDSMap}) leads to a $0$-loop, for which our theory predicts that the bifurcation diagram consists of isolas and therefore does not snake. In $\S\ref{subsec:Numerics}$ we provide numerical investigations which confirm our theoretical work. Finally, in $\S\ref{subsec:Extensions}$ we comment on a number of ways by which the results of this manuscript can be extended to a more diverse range of lattice dynamical systems. 


\subsection{Flat Plateaus}\label{subsec:Flat} 

Our main result states that for sufficiently small $0 < d \ll 1$ equation (\ref{LDSMap}) has a $1$-loop of heteroclinic connections between $(0,0)$ and $(U_+(\mu),U_+(\mu))$.

\begin{prop}\label{prop:LDSGamma} 
	There exists $d_* > 0$ such that for all $0 < d \leq d_*$ equation (\ref{LDSMap}) has a $1$-loop of heteroclinic connections.
\end{prop}  

Proposition~\ref{prop:LDSGamma} demonstrates the existence of a 1-loop of heteroclinic connections, and hence $\Gamma_\mathrm{lift}$ is a single connected curve. Prior to proving Proposition~\ref{prop:LDSGamma} we state the following corollary which connects the results of Proposition~\ref{prop:LDSGamma} to Theorem~\ref{thm:Pulses} and in turn predicts snaking of the on- and off-site localized steady-states to (\ref{LDS}).

\begin{cor}\label{cor:1Loop} 
	The bifurcation curves of on- and off-site symmetric heteroclinic orbits of (\ref{LDSMap}) are single connected curves, and therefore the corresponding steady-state solutions of (\ref{LDS}) will snake.
\end{cor}

We now prove Proposition~\ref{prop:LDSGamma}. We begin by noting that it is easier to study the singular regime for the original map (\ref{LDS_Steady}). Setting $d= 0$ in (\ref{LDS_Steady}) gives the polynomial equation
\begin{equation}\label{ReactionTerm}
	0 = - \mu U_n + 2U_n^3 - U_n^5,
\end{equation}
which has roots $U_n = 0, U_\pm(\mu)$ for all $\mu \in [0,1]$. To construct heteroclinic connections of (\ref{LDSMap}) between the fixed points $(0,0)$ and $(U_+(\mu),U_+(\mu))$ for $0 < d \ll 1$, we define three singular heteroclinic connections for (\ref{ReactionTerm}) via $\bar{U}(\mu) = \{\bar{u}_n(\mu)\}_{n\in\mathbb{Z}}$ with
 \begin{equation} \label{HetSol1}
	\bar{u}_n(\mu) = \left\{
     		\begin{array}{cl} 0, & n \leq 0 \\ 
		U_+(\mu), & n > 0
		\end{array}
   	\right.
\end{equation}
and $\bar{V}^\pm(\mu) = \{\bar{v}_n(\mu)\}_{n\in\mathbb{Z}}$ with
 \begin{equation} \label{HetSol2}
	\bar{v}^\pm_n(\mu) = \left\{
     		\begin{array}{cl} 0, & n < 0 \\ 
		\pm U_-(\mu), & n= 0 \\
		U_+(\mu), & n > 0
		\end{array}
   	\right..
\end{equation}
We record for later use that $\|\bar{U}(\mu) - \bar{V}^\pm(\mu)\|_\infty\to 0$ as $\mu \to 0^+$ and $\|\bar{U}(\mu) - S^{-1}\bar{V}^+(\mu)\|_\infty\to 0$ as $\mu \to 1^-$, where $S$ is again the left shift operator (\ref{Shift}) on sequences indexed by $\mathbb{Z}$. We note that $\bar{V}^+(\mu)$ and $\bar{V}^-(\mu)$ bifurcate from $\bar{U}(\mu)$ in a pitchfork bifurcation at $\mu = 0$. This pitchfork bifurcation arises due to the symmetry action $\kappa:\ell^\infty \to \ell^\infty$ given by
	\[
		[\kappa U]_n = \left\{\begin{array}{cl} U_n, & n \neq 0 \\ 
		-U_0, & n= 0 
		\end{array}\right.,
	\]
	present when $d=0$. For $d \neq 0$ the system is no longer equivariant under $\kappa$ and therefore this bifurcation should become an imperfect pitchfork for $0 < d \ll 1$. As a result, one branch disconnects and continues through the bifurcation point, while the other two branches connect at a saddle-node, as is show in Figure~\ref{fig:AsymPitchfork}. One of the goals of this section is now to describe exactly how the pitchfork bifurcation at $d = 0$ breaks for $d > 0$ and small.
	
\begin{figure} 
	\centering
	\includegraphics[width=0.7\textwidth]{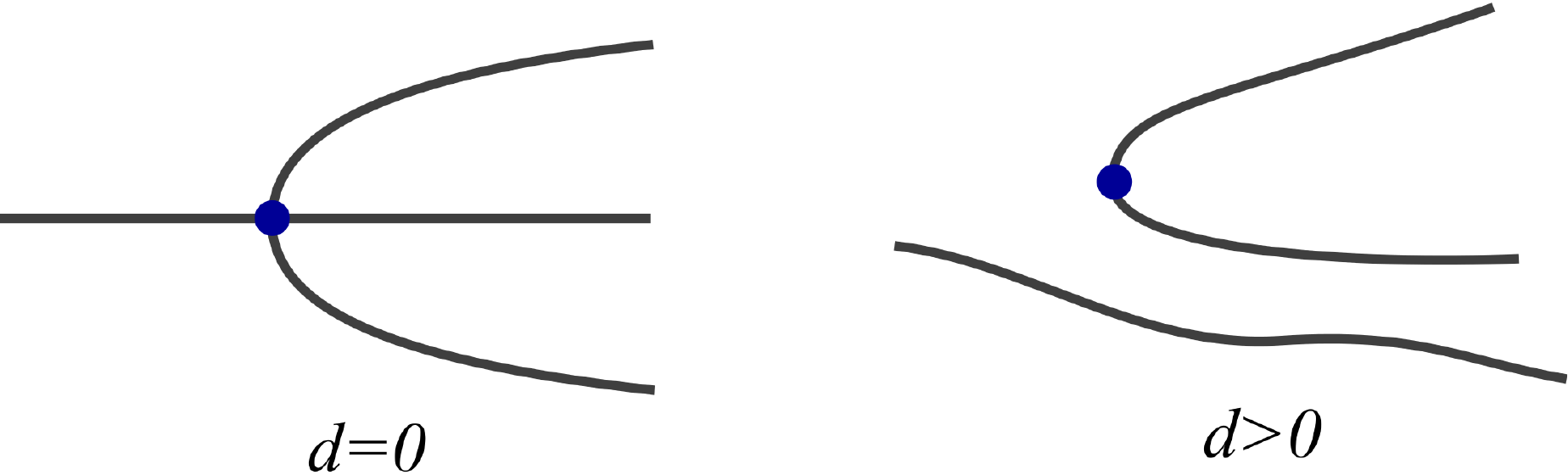}
	\caption{An illustration of the unfolding of the pitchfork bifurcation for $d \neq 0$.}
\label{fig:AsymPitchfork}
\end{figure} 

From Lemma~\ref{lem:GeneralLattice} we have that the solutions (\ref{HetSol1}) and (\ref{HetSol2}) continue regularly in $0 < d \ll 1$ for $\mu$ taken in any compact subinterval of the interval $(0,1)$. Therefore, we need only understand how the solutions (\ref{HetSol1}) and (\ref{HetSol2}) continue in $0 < d \ll 1$ near the bifurcation points at $\mu = 0,1$. Our first result focuses on the region near $\mu = 0$.    

\begin{lem}\label{lem:Pitchfork} 
	There exist $d_1,\mu_1 > 0$ and a function $\mu_l(d): [0,d_1] \to [0,\mu_1)$ such that for each fixed $d \in [0,d_1]$, at $\mu = \mu_l(d)$ a pair of steady-state solutions to (\ref{LDS}), $U_l(\mu,d)$ and $V_l^+(\mu,d)$, emanate in a saddle-node bifurcation and exist for all $\mu \in [\mu_l(d),\mu_1]$. These solutions are continuous in $\mu$ and $d$ and are such that $U_l(\mu,d) \to \bar{U}(\mu)$ and $V_l^+(\mu,d) \to \bar{V}^+(\mu)$ as $d \to 0^+$, for each fixed $\mu$. The function $\mu_l(d)$ is given by
	\[
		\mu_l(d) = \frac{3}{\sqrt[3]{2}}d^\frac{2}{3} + \mathcal{O}(d).	
	\]   
\end{lem}

\begin{proof}
	Here we will prove that the pitchfork bifurcation at $d = \mu = 0$ unfolds into a saddle-node bifurcation for $d > 0$ and sufficiently small. We will be concerned with solutions to the steady state equation:
\begin{equation} \label{ZeroEqn}
	d(U_{n+1} + U_{n-1} - 2U_n) - \mu U_n + 2U_n^3 - U_n^5 = 0,
\end{equation}
for which both $\bar{U}(\mu)$ and $\bar{V}(\mu)$ are solutions for all $\mu \in [0,1]$ when $d = 0$. We focus our analysis about $\mu = 0$. 

Linearizing $(\ref{ZeroEqn})$ about the steady-state $\bar{U}(0) = \bar{V}(0)$ and taking $d = 0$ leads to a bounded linear operator, $L_0:\ell^2 \to \ell^2$, acting on sequences $x =\{x_n\}_{n \in \mathbb{Z}} \in \ell^2$ by
\[
	[L_0 x]_n = \left\{
     		\begin{array}{cl} 0, & n \leq 0 \\ \\
		-8x_n, & n > 0.
		\end{array}
   	\right. 
\]
This allows us to apply the implicit function theorem to conclude that (\ref{ZeroEqn}) restricted to $n > 0$ has a unique solution $\{\bar{U}_n(\mu,d,U_0)\}_{n > 0}$ for each $\mu, d$, and $U_0$ near zero and that this solution satisfies 
\begin{equation} \label{Rescaling}
	\bar{U}_n(\mu,d,U_0) = \sqrt{2} + \mathcal{O}(\mu+dU_0). 	
\end{equation} 
It remains to solve (\ref{ZeroEqn}) for the indices $n \leq 0$.

For each $n \leq 0$, we introduce the variables
	\[
		U_n := \mu^\frac{1-n}{2}\tilde{U}_n,	
	\]
	along with the re-parametrization $d = \mu^\frac{3}{2}\tilde{d}$. Then, for each $n < 0$ the steady-state equation (\ref{ZeroEqn}) becomes
	\[
		0 = \mu^\frac{3-n}{2}\tilde{d}\tilde{U}_{n+1} + \mu^\frac{5-n}{2}\tilde{d}\tilde{U}_{n-1} + 2\mu^\frac{4-n}{2}\tilde{d}\tilde{U}_{n} -  \mu^\frac{3-n}{2}\tilde{U}_{n} + 2\mu^\frac{3-3n}{2}\tilde{U}_{n}^3 - \mu^\frac{5-5n}{2}\tilde{U}_{n}^5.
	\]
	Similarly, at the index $n = 0$ we obtain
	\[
		0 = \mu^\frac{3}{2}\tilde{d}\bar{U}_n(\mu,\mu^\frac{3}{2}\tilde{d},\mu^\frac{1}{2}\tilde{U}_0) + \mu^\frac{5}{2}\tilde{d}\tilde{U}_{-1} + 2\mu^\frac{4}{2}\tilde{d}\tilde{U}_{0} -  \mu^\frac{3}{2}\tilde{U}_{0} + 2\mu^\frac{3}{2}\tilde{U}_{0}^3 - \mu^\frac{5}{2}\tilde{U}_{0}^5.	    
	\]
	Hence, separating out orders in $\mu$ we obtain
	\begin{equation}\label{ZeroEqn2}
		\begin{split}
			&\underline{n = 0}: \quad 0=\mu^\frac{3}{2}(\tilde{d} - \tilde{U}_0 + 2\tilde{U}_0^3) + \mathcal{O}(\mu^2) \implies 0 = \tilde{d} - \tilde{U}_0 + 2\tilde{U}_0^3 + \mathcal{O}(\mu^\frac{1}{2}), \\
			&\underline{n < 0}: \quad 0=\mu^\frac{3-n}{2}(\tilde{d}\tilde{U}_{n+1} - \tilde{U}_n) + \mathcal{O}(\mu^\frac{4-n}{2}) \implies 0 = \tilde{d}\tilde{U}_{n+1} - \tilde{U}_n + \mathcal{O}(\mu^\frac{1}{2}),
		\end{split}
	\end{equation}
 	where we have applied (\ref{Rescaling}) to expand $\bar{U}_n(\mu,\mu^\frac{3}{2}\tilde{d},\mu^\frac{1}{2}\tilde{U}_0)$ about $\mu = 0$. 
	
	Let us consider the Banach space $\ell^\infty_{n\leq 0}$ containing sequences indexed by $n \leq 0$, analogous to the Banach space $\ell^\infty$. Then, we consider the function $h:\mathbb{R} \times \ell^\infty_{n\leq 0} \times [0,\infty) \to \ell^\infty_{n\leq 0}$ by
	\[
		[h(\tilde{d},\tilde{U},\mu)]_n = \left\{
     		\begin{array}{cl} \tilde{d} - \tilde{U}_0 + 2\tilde{U}_0^3 + \mathcal{O}(\mu^\frac{1}{2}), & n = 0 \\ \\
		\tilde{d}\tilde{U}_{n+1} - \tilde{U}_n + \mathcal{O}(\mu^\frac{1}{2}), & n < 0.
		\end{array}
   	\right. 
	\]
	so that the roots of $h$ are exactly the solutions of (\ref{ZeroEqn2}). We note that $h$ is smooth in $\tilde{d},\tilde{U}$, and $\mu^\frac{1}{2}$, and the derivative with respect to $\tilde{U}$, denoted $D_{\tilde{U}}h:\mathbb{R} \times \ell^\infty_{n\leq 0} \times [0,\infty) \to \ell^\infty_{n\leq 0}$ is the linear operator acting on the sequences $x = \{x_n\}_{n\leq 0}$ by
	\[
		[D_{\tilde{U}}h(\tilde{d},\tilde{U},\mu)x]_n = \left\{
     		\begin{array}{cl} (-1 + 6\tilde{U}_0^2)x_0 + \mathcal{O}(\mu^\frac{1}{2}), & n = 0 \\ \\
		\tilde{d}x_{n+1} - x_n + \mathcal{O}(\mu^\frac{1}{2}), & n < 0.
		\end{array}
   	\right. 
	\]  
	for all $(\tilde{d},\tilde{U},\mu)$. 
	
	At $\mu = 0$ there is a saddle-node bifurcation taking place for (\ref{ZeroEqn2}) at  $\tilde{U}_0 = \frac{1}{\sqrt{6}}$. In turn, the vector $e_0 \in \ell^\infty_{n\leq 0}$ given by 
	\[
		[e_0]_n = \left\{
     		\begin{array}{cl} 1, & n = 0 \\ \\
		0, & n < 0.
		\end{array}
   	\right. 	
	\] 
	belongs to the kernel of $D_{\tilde{U}}h(\tilde{d},\frac{1}{\sqrt{6}},0)$ for all $\tilde{d}\in\mathbb{R}$. Solving $[D_{\tilde{U}}h(\tilde{d},\frac{1}{\sqrt{6}},0)x]_n = 0$ for $n < 0$ requires solving 
	\[
		\tilde{d}x_{n+1} - x_n = 0,	
	\]  
	for every $n < 0$, which can be solved inductively to yield
	\[
		x_n= \tilde{d}^{-n}x_0.
	\]  
	The vector $\{\tilde{d}^{-n}\}_{n\leq 0}$ belongs to $\ell^\infty_{n\leq 0}$ if, and only if, $|\tilde{d}| \leq 1$. Therefore, the kernel of $D_{\tilde{U}}h(\tilde{d},\tilde{U},0)$ is spanned by $e_0$ for all $|\tilde{d}| \leq 1$. A similar argument shows that $e_0$ spans the cokernel of $D_{\tilde{U}}h(\tilde{d},\tilde{U},0)$ when $|\tilde{d}| \leq 1$, and hence $D_{\tilde{U}}h(\tilde{d},\tilde{U},0)$ is a Fredholm operator with index 0 for all $|\tilde{d}| \leq 1$.     
	
	We may therefore apply a Lyapunov-Schmidt reduction to $h$ in a neighbourhood of $(\tilde{U}_0,\mu) = (\frac{1}{\sqrt{6}},0)$, uniformly in $|\tilde{d}|\leq 1$ to result in the reduced real-valued bifurcation function, $h_l$, given by
	\[
		h_l(\tilde{d},\tilde{U}_0,\mu) = \tilde{d} - \tilde{U}_0 + 2\tilde{U}_0^3 + \mathcal{O}(\mu^\frac{1}{2}). 
	\]
	We may employ the implicit function theorem to obtain the function 
	\[
		 \tilde{d}(\tilde{U}_0,\mu) = \tilde{U}_0 - 2\tilde{U}_0^3 + \mathcal{O}(\mu^\frac{1}{2}) 
	\]
	so that $h_l(\tilde{d}(\tilde{U}_0,\mu),\tilde{U}_0,\mu) = 0$ for all $\tilde{U}_0$ and $\mu\geq 0$ sufficiently small. 
	
	Finally, the location of the saddle-node bifurcation in $h_l$ as a function of $\mu$ can be determined by solving
	\[
		0 = \partial_{\tilde{U}_0}h_l(\tilde{d}(\tilde{U}_0,\mu),\tilde{U}_0,\mu) = -1 + 6\tilde{U}_0^2 + \mathcal{O}(\mu^\frac{1}{2}) 
	\]
	for $(\tilde{U}_0,\mu)$ in a neighbourhood of $(\frac{1}{\sqrt{6}},0)$. Such a curve $\tilde{U}_0 = \tilde{U}_0(\mu)$ is guaranteed to exist by the implicit function theorem and satisfies
	\[
		\tilde{U}_0(\mu) = \frac{1}{\sqrt{6}} + \mathcal{O}(\mu^\frac{1}{2}), 	
	\]
	which in turn gives the location of the saddle-node bifurcations 
	\[
		\tilde{d}_\mathrm{sn} = \tilde{d}(\tilde{U}_0(\mu),\mu) = \frac{2}{3\sqrt{6}} + \mathcal{O}(\mu^\frac{1}{2}),
	\]
	valid $\mu$ sufficiently small. Recalling that $d = \mu^\frac{3}{2}\tilde{d}$ gives that the saddle-node bifurcations take place at 
	\[
		d_\mathrm{sn} = \mu^\frac{3}{2}\bigg(\frac{2}{3\sqrt{6}} + \mathcal{O}(\mu^\frac{1}{2})\bigg). 
	\]
	Using the inverse function theorem allows one to write $\mu$ as a function of $d_\mathrm{sn}$ to obtain the function $\mu_l(d)$ given in the lemma. We again have that $\mu_l(d)$ gives the location of a saddle-node bifurcations unfolded in $\mu$ since 
	\[
		d = \mu^\frac{3}{2}\tilde{d}(\tilde{U}_0,\mu) \implies d^\frac{2}{3} = \mu\tilde{d}(\tilde{U}_0,0) + \mathcal{O}(\mu^\frac{4}{3}),	
	\]   
	which upon applying the inverse function theorem in a neighbourhood of $\tilde{U}_0 = \frac{1}{\sqrt{6}}$ gives 
	\[
		\mu(d,\tilde{U}_0) = \frac{d^\frac{2}{3}}{\tilde{d}(\tilde{U}_0,0)} + \mathcal{O}(d).  
	\]
	Expanding about the point $\tilde{U}_0 = \frac{1}{\sqrt{6}}$ gives 
	\[
		\mu(d,\tilde{U}_0) = \frac{3d^\frac{2}{3}}{\sqrt[3]{2}}\bigg(1 - \frac{2}{3\sqrt{6}}\bigg(\tilde{U}_0 - \frac{1}{\sqrt{6}}\bigg)^2\bigg) + \mathcal{O}\bigg(d + d^\frac{2}{3}\bigg|\tilde{U}_0 - \frac{1}{\sqrt{6}}\bigg|^3\bigg) 	
	\]
	which shows that varying $\tilde{U}_0$ in a neighbourhood of $\frac{1}{\sqrt{6}}$ unfolds a saddle-node bifurcation in $\mu$ for fixed small $d > 0$. This completes the proof.
\end{proof} 

\begin{lem}\label{lem:SaddleNode} 
	There exist $d_2,\mu_2 > 0$ and a function $\mu_r(d):[0,d_2]\to(\mu_2,1]$ such that for each fixed $d \in [0,d_2]$, at $\mu = \mu_r(d)$ a pair of steady-state solutions to (\ref{LDS}), $U_r(\mu,d)$ and $V_r^+(\mu,d)$, emanate in a saddle-node bifurcation and exist for all $\mu \in [\mu_2,\mu_r(d)]$. These solutions are continuous in $\mu$ and $d$ and are such that $U_r(\mu,d) \to \bar{U}(\mu)$ and $V_r^+(\mu,d) \to S^{-1}\bar{V}^+(\mu)$ as $d \to 0^+$, for each fixed $\mu$. The function $\mu_r(d)$ is given by
   	\[
		\mu_r(d) = 1 - d + \mathcal{O}(d^\frac{3}{2}).	
	\]  
\end{lem}

\begin{proof}
	The proof is similar to the proof of Lemma~\ref{lem:Pitchfork} except that we can now apply the implicit function theorem to solve for $n \leq 0$ for each given $\tilde{U}_1$ and $\mu$ near 1 and $d$ small. These solutions satisfy 
	\begin{equation} \label{Rescaling2}
		\bar{U}_n(\mu,d,U_0) = \mathcal{O}(|\mu-1|+d+|U_0-1|)
	\end{equation}
	for each $n \leq 0$.
	
	Now, let $\mu = 1 - \tilde{\mu}$, and introduce
	\[
		U_n :=1 + \tilde{\mu}^\frac{1}{2}\tilde{U}_n,
	\]    
	for each $n \geq 1$, and $d = \tilde{\mu}\tilde{d}$. In a similar manner to the proof of Lemma~\ref{lem:Pitchfork}, we may expand in powers of $\tilde{\mu}$ to arrive at the system of equations 
	\begin{equation}\label{ZeroEqn3}
		\begin{split}
			&\underline{n = 1}: \quad 0 = -\tilde{d} + 1 - 4\tilde{U}_1^2 + \mathcal{O}(\tilde{\mu}^\frac{1}{2}), \\
			&\underline{n > 1}: \quad 0 = 1 - 4\tilde{U}_n^2 + \mathcal{O}(\tilde{\mu}^\frac{1}{2}),
		\end{split}
	\end{equation}
	where the equation at index $n = 1$ is simplified using $\bar{U}_n(1 - \tilde{\mu},\tilde{\mu}\tilde{d},1 + \tilde{\mu}^\frac{1}{2}\tilde{U}_0) = \mathcal{O}(\tilde{\mu}^\frac{1}{2})$ from  (\ref{Rescaling2}).
	
	A saddle-node bifurcation takes place in (\ref{ZeroEqn3}) at $\tilde{\mu} = 0$, $\tilde{d} = 1$, $\tilde{U}_0 = 0$, and $\tilde{U}_n = \frac{1}{2}$ for all $n > 1$. Following the methods of Lemma~\ref{lem:Pitchfork} we may apply a Lyapunov-Schmidt reduction in a neighbourhood of this saddle-node bifurcation at $\tilde{\mu} = 0$ to reduce to the real-valued bifurcation function, $h_r(\tilde{d},\tilde{U}_0,\mu)$, given by
	\[
		h_r(\tilde{d},\tilde{U}_0,\mu) = -\tilde{d} + 1 - 4\tilde{U}_1^2 + \mathcal{O}(\tilde{\mu}^\frac{1}{2}). 	
	\] 
	Then, persistence of the saddle-node bifurcation in $h_r$ at $(\tilde{d},\tilde{U}_0,\mu) = (1,0,0)$ can again be obtained by an application of the implicit function theorem, thus allowing one to obtain the expression $\mu_r(d)$ given in the statement of the lemma.

	Finally, to see that our branches continue from $\bar{U}(\mu)$ and $S^{-1}\bar{V}^+(\mu)$, we simply expand $U_+(\mu)$ about $\mu = 1$ to find that
	\[
		U_+(\mu) = \sqrt{1+\sqrt{1 -\mu}} = 1 + \frac{1}{2}\sqrt{1 - \mu} + \mathcal{O}(|1 - \mu|).
	\] 	
	Recalling that $\tilde{\mu} = 1 - \mu$, our obtained solutions $U_r(\mu,d)$ and $V_r^+(\mu,d)$ satisfy
	\[
		[U_r(\mu,d)]_n = [V_r^+(\mu,d)]_n = \left\{
     		\begin{array}{cl} 0, & n < 0  \\ \\
		 1,  & n = 0  \\ \\
		1 + \frac{1}{2}\sqrt{1 - \mu},  & n < 0.
		\end{array}
		\right\} + \mathcal{O}(|1 - \mu| + \sqrt{1 - \mu}\tilde{U}_0 + d|1 - \mu|^{-\frac{1}{2}})
	\]
	where varying $\tilde{U}_0$ in a neighbourhood of 0 unfolds the saddle-node bifurcation and gives the distinct branches $U_r(\mu,d)$ and $V_r^+(\mu,d)$. Local uniqueness guarantees that for each fixed $\mu$ sufficiently close to 1 we have that $U_r(\mu,d) \to \bar{U}(\mu)$ and $V_r^+(\mu,d) \to S^{-1}\bar{V}^+(\mu)$ as $d \to 0^+$, completing the proof.   
\end{proof}


\subsection{Oscillatory Plateaus}\label{subsec:Oscillatory} 

The results of this manuscript can also be applied to understanding the bifurcation structure of homoclinic orbits which spend long time near a periodic orbit of (\ref{Map}). Indeed, as opposed to inspecting the iteration scheme $u_{n+1} = F(u_n,\mu)$, we may apply the above results to the iteration scheme $u_{n+1} = F^k(u_n,\mu)$, for any $k \geq 1$, provided we have verified the necessary hypotheses for this to work. Let us now illustrate this concept with the concrete example of (\ref{LDS}).

The mapping (\ref{LDSMap}) has a 2-cycle given by 
\begin{equation}\label{2cycle}
	\{(U_+(\mu-4d),-U_+(\mu-4d)),(-U_+(\mu-4d),U_+(\mu-4d))\}
\end{equation}
for each $(d,\mu)$ satisfying $\mu - 4d \leq 1$. Since the elements of this 2-cycle become a pair of fixed points in the second iterate mapping of (\ref{LDSMap}), we wish to explore the bifurcation curves of heteroclinic connections between the fixed point $(0,0)$ and either of these fixed points in the second iterate mapping to apply the results of this manuscript. To achieve this we may work to understand how the bifurcation curves of heteroclinic connections between $(0,0)$ and $(U_+(\mu),U_+(\mu))$ behave for $d < 0$, and then exploit the staggering symmetry (\ref{StaggeringSym}) to extend these results to heteroclinic connections between $(0,0)$ and the 2-cycle (\ref{2cycle}) for $d > 0$. The reason for this is that if we assume $\{U_n\}_{n\in\mathbb{Z}}$ is a steady-state solution of (\ref{LDS}) for fixed $(d,\mu)$ satisfying $U_n \to 0$ as $n \to -\infty$ and $U_n \to U_+(\mu)$ as $n \to +\infty$, then the staggering symmetry (\ref{StaggeringSym}) implies that $\{(-1)^nU_n\}_{n\in\mathbb{Z}}$ is a steady-state solution of (\ref{LDS}) for $(-d,\mu-4d)$. The steady-state solution $\{(-1)^nU_n\}_{n\in\mathbb{Z}}$ now represents a heteroclinic connection of (\ref{LDSMap}) between the fixed point $(0,0)$ and the 2-cycle (\ref{2cycle}). This leads to the following proposition.

\begin{prop}\label{prop:LDSGamma2} 
	There exists $d_{**} > 0$ such that for all $-d_{**} \leq d < 0$ equation (\ref{LDSMap}) has a $0$-loop of heteroclinic connections. 
\end{prop}  

Note that it follows from Theorem~\ref{thm:Pulses} and Proposition~\ref{prop:LDSGamma2} that for sufficiently small $d < 0$ the map (\ref{LDSMap}) exhibits on- and off-site homoclinic orbits whose bifurcation structures are isolas. The staggering symmetry (\ref{StaggeringSym}) implies that for sufficiently small $d > 0$ there exists steady-state solutions to (\ref{LDS}) with oscillatory plateaus that decay to $0$ at $\pm \infty$ whose bifurcation structure are isolas as well. We summarize these findings with the following corollary.

\begin{cor}\label{cor:0Loop}
	The bifurcation curves of steady-state solutions of (\ref{LDS}) with oscillatory plateaus are isolas and therefore snaking is precluded in this situation.	
\end{cor}	

We now proceed with the proof of Proposition~\ref{prop:LDSGamma2}. It is again easier to study the singular regime for the original lattice differential equation (\ref{LDS_Steady}), and therefore to prove Proposition~\ref{prop:LDSGamma2} we follow in a similar manner to the previous subsection. We again consider $\bar{U}(\mu)$ and $\bar{V}^\pm(\mu)$ defined in (\ref{HetSol1}) and (\ref{HetSol2}), respectively, along with the following singular heteroclinic orbit $\bar{W}(\mu) = \{\bar{w}_n(\mu)\}_{n\in\mathbb{Z}}$ with
 \begin{equation} \label{HetSol3}
	\bar{w}_n(\mu) = \left\{
     		\begin{array}{cl} 0, & n < 0 \\ 
		-U_-(\mu), & n = 0 \\ 
		U_-(\mu), & n = 1 \\
		U_+(\mu), & n > 1 \\
		\end{array}
   	\right.
\end{equation}

From Lemma~\ref{lem:GeneralLattice} we again have that the solutions (\ref{HetSol1}), (\ref{HetSol2}), and (\ref{HetSol3}) continue regularly in $-1 \ll d < 0$ for $\mu$ taken in any compact subinterval of the interval $(0,1)$. Therefore, we need only understand how these singular heteroclinic orbits continue in $-1 \ll d < 0$ near the bifurcation points at $\mu = 0,1$. The proof of Proposition~\ref{prop:LDSGamma2} is broken down over the four lemmas, and the results are summarized visually in Figure~\ref{fig:0Loop}. Our first result extends Lemma~\ref{lem:Pitchfork} into the region $d < 0$.  

\begin{figure} 
	\centering
	\includegraphics[width=0.4\textwidth]{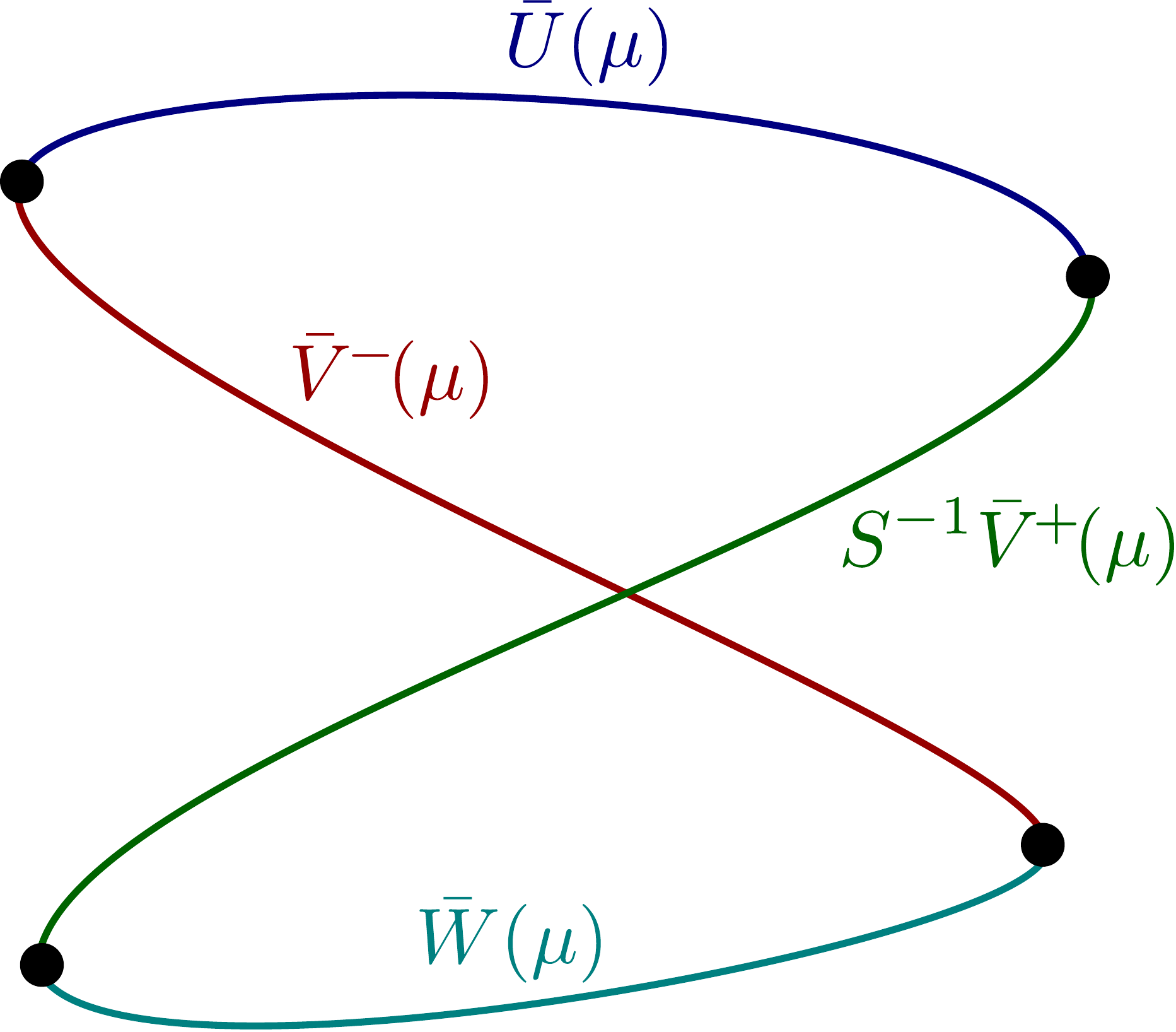}
	\caption{A visual description of the results of Lemmas~\ref{lem:Pitchfork2}-\ref{lem:SaddleNode3}. Black dots indicate the saddle-node bifurcations and labels for each curve indicates which singular heteroclinic orbit they are continued from using Lemma~\ref{lem:GeneralLattice}.}
\label{fig:0Loop}
\end{figure}

\begin{lem}\label{lem:Pitchfork2} 
	There exist $d_3,\mu_3 > 0$ and a function $\mu_{l,1}(d): [-d_3,0] \to [0,\mu_3)$ such that for each fixed $d \in [-d_3,0]$, at $\mu = \mu_{l,1}(d)$ a pair of steady-state solutions to (\ref{LDS}), $U_{l,1}(\mu,d)$ and $V_{l,1}^-(\mu,d)$, emanate in a saddle-node bifurcation and exist for all $\mu \in [\mu_{l,1}(d),\mu_3]$. These solutions are continuous in $\mu$ and $d$ and are such that $U_{l,1}(\mu,d) \to \bar{U}(\mu)$ and $V_{l,1}^-(\mu,d) \to \bar{V}^-(\mu)$ as $d \to 0^-$, for each fixed $\mu$. The function $\mu_{l,1}(d)$ is given by
	\[
		\mu_{l,1}(d) = -\frac{3}{\sqrt[3]{2}}d^\frac{2}{3} + \mathcal{O}(d).	
	\]   
\end{lem}

\begin{proof}
	This proof is nearly identical to the proof of Lemma~\ref{lem:Pitchfork}, but we now focus on the saddle-node bifurcation at $\tilde{U}_0 = -\frac{1}{\sqrt{6}}$. 
\end{proof}

\begin{lem}\label{lem:SaddleNode2} 
	There exist $d_4,\mu_4 > 0$ and a function $\mu_{r,1}(d):[-d_4,0]\to(\mu_4,1]$ such that for each fixed $d \in [-d_4,0]$, at $\mu = \mu_{r,1}(d)$ a pair of steady-state solutions to (\ref{LDS}), $U_{r,1}(\mu,d)$ and $V_{r,1}^+(\mu,d)$, emanate in a saddle-node bifurcation and exist for all $\mu \in [\mu_4,\mu_{r,1}(d)]$. These solutions are continuous in $\mu$ and $d$ and are such that $U_{r,1}(\mu,d) \to \bar{U}(\mu)$ and $V_{r,1}^+(\mu,d) \to S^{-1}\bar{V}^+(\mu)$ as $d \to 0^+$, for each fixed $\mu$. The function $\mu_{r,2}(d)$ is given by
   	\[
		\mu_{r,1}(d) = 1 + d + \mathcal{O}(d^\frac{3}{2}).	
	\]  
\end{lem}

\begin{proof}
	This proof is nearly identical to the proof of Lemma~\ref{lem:SaddleNode}, with the minor adjustment that we introduce 
	\[
		U_n :=1 - \tilde{\mu}^\frac{1}{2}\tilde{U}_n,
	\]    
	for each $n \geq 1$, and $d = -\tilde{\mu}\tilde{d}$. From here everything follows as in the proof of Lemma~\ref{lem:SaddleNode}. 
\end{proof}

\begin{lem}\label{lem:Pitchfork3} 
	There exist $d_5,\mu_5 > 0$ and a function $\mu_{l,2}(d): [-d_5,0] \to [0,\mu_5)$ such that for each fixed $d \in [-d_5,0]$, at $\mu = \mu_{l,2}(d)$ a pair of steady-state solutions to (\ref{LDS}), $U_{l,2}(\mu,d)$ and $V_{l,2}^-(\mu,d)$, emanate in a saddle-node bifurcation and exist for all $\mu \in [\mu_{l,2}(d),\mu_5]$. These solutions are continuous in $\mu$ and $d$ and are such that $U_{l,2}(\mu,d) \to S^{-1}\bar{V}^+(\mu)$ and $V_{l,2}^-(\mu,d) \to \bar{W}(\mu)$ as $d \to 0^-$, for each fixed $\mu$. The function $\mu_{l,2}(d)$ is given by
	\[
		\mu_{l,2}(d) = -\bigg(\frac{729}{4}\bigg)^\frac{1}{5}d^\frac{4}{5} + \mathcal{O}(d).	
	\]   
\end{lem}

\begin{proof}
	The proof is similar to the proofs of Lemma~\ref{lem:Pitchfork} and Lemma~\ref{lem:Pitchfork2} in that we apply the implicit function theorem to find that (\ref{ZeroEqn}) restricted to $n > 1$ has a unique solution $\{\bar{U}_n(\mu,d,U_1)\}_{n > 1}$ for each $\mu, d$, and $U_1$ near zero and that this solution satisfies 
\begin{equation} \label{Rescaling3}
	\bar{U}_n(\mu,d,U_1) = 1 + \mathcal{O}(\mu+dU_1). 	
\end{equation} 
It remains to solve (\ref{ZeroEqn}) for the indices $n \leq 1$.

For each $n \leq 0$, we introduce the variables
	\[
		U_n := (-1)^{n}\mu^\frac{2-n}{4}\tilde{U}_n,	
	\]
	along with the re-parametrization $d = -\mu^\frac{5}{4}\tilde{d}$. In a similar manner to the previous proofs, we may expand in powers of $\mu$ to arrive at the system of equations 
	\begin{equation}\label{ZeroEqn4}
		\begin{split}
			&\underline{n = 1}: \quad 0 = -\tilde{d} + \tilde{U}_1 + \mathcal{O}(\mu^\frac{1}{4}), \\
			&\underline{n = 0}: \quad 0 = \tilde{d}\tilde{U}_1 - \tilde{U}_0 + 2\tilde{U}_0^3 + \mathcal{O}(\mu^\frac{1}{4}), \\
			&\underline{n < 0}: \quad 0 = \tilde{d}\tilde{U}_{n+1} - \tilde{U}_n + \mathcal{O}(\mu^\frac{1}{4}),
		\end{split}
	\end{equation}
	where the equation at index $n = 1$ is simplified using $\bar{U}_2(\mu,d,U_1) = \bar{U}_2(\mu,\mu^\frac{5}{4}\tilde{d},\mu^\frac{1}{4}\tilde{U}_1) = 1 + \mathcal{O}(\mu)$ from  (\ref{Rescaling3}).
	
	The equation at index $n = 1$ can be solved using the implicit function theorem $\tilde{U}_1(\tilde{d},\mu,\tilde{U}_0)$ as a function over all $\tilde{d},\tilde{U}_0$ and sufficiently small $\mu$. Furthermore, this function has the expansion
	\[
		\tilde{U}_1(\tilde{d},\mu,\tilde{U}_0) = -\tilde{d} + \mathcal{O}(\mu^\frac{1}{4}). 	
	\] 
	Upon putting this into the remaining equations for (\ref{ZeroEqn4}), we arrive at the infinite system of equations
  	\[
		\begin{split}
			&\underline{n = 0}: \quad 0 = -\tilde{d}^2 - \tilde{U}_0 + 2\tilde{U}_0^3 + \mathcal{O}(\mu^\frac{1}{4}), \\
			&\underline{n < 0}: \quad 0 = \tilde{d}\tilde{U}_{n+1} - \tilde{U}_n + \mathcal{O}(\mu^\frac{1}{4}).
		\end{split}
	\]
	From here we may follow as in the proofs of Lemmas~\ref{lem:Pitchfork} and \ref{lem:Pitchfork2} to obtain the desired result.
\end{proof}

\begin{lem}\label{lem:SaddleNode3} 
	There exist $d_6,\mu_6 > 0$ and a function $\mu_{r,2}(d):[-d_6,0]\to(\mu_6,1]$ such that for each fixed $d \in [-d_6,0]$, at $\mu = \mu_{r,2}(d)$ a pair of steady-state solutions to (\ref{LDS}), $U_{r,2}(\mu,d)$ and $V_{r,2}^+(\mu,d)$, emanate in a saddle-node bifurcation and exist for all $\mu \in [\mu_6,\mu_{r,2}(d)]$. These solutions are continuous in $\mu$ and $d$ and are such that $U_{r,1}(\mu,d) \to \bar{V}^-(\mu)$ and $V_{r,1}^+(\mu,d) \to \bar{W}^+(\mu)$ as $d \to 0^+$, for each fixed $\mu$. The function $\mu_{r,2}(d)$ is given by
   	\[
		\mu_{r,2}(d) = 1 + d + \mathcal{O}(d^\frac{3}{2}).	
	\]  
\end{lem}

\begin{proof}
	This proof is identical to those of Lemmas~\ref{lem:SaddleNode} and \ref{lem:SaddleNode2}.
\end{proof}


\subsection{Numerical Validation}\label{subsec:Numerics} 

\paragraph{Patterns with flat plateaus snake.}
We first present numerical computations that indicate that patterns with flat plateaus snake for all values of $d>0$, though the width of the snaking curve shrinks to zero in the continuum limit $d\to\infty$.

We saw in Figure~\ref{fig:Snaking_1D} numerical verification of Corollary~\ref{cor:1Loop} for $d = 0.5$. This figure shows that symmetric steady-state solutions of (\ref{LDS}) with flat plateaus snake when plotting $\mu$ against the square of the $\ell^2$-norm of the bifurcating solution. Furthermore, near the left and right saddle-node bifurcations of the symmetric steady-states one can see that a pair of asymmetric steady-states bifurcate in a pitchfork bifurcation, as predicted by Lemma~\ref{lem:AsymPitchfork}. These asymmetric bifurcation curves continue across the bifurcation diagram connecting the on- and off-site bifurcation curves, which gives the bifurcation diagram the familiar snakes and ladders appearance.  

In Figure~\ref{fig:IncreasingCoupling} we show the effect on the snaking bifurcation curves as the coupling parameter $d$ is increased. We can see that as $d$ increases, the snaking region narrows and appears to collapse onto the Maxwell point $\mu = \mu_\mathrm{mx} = 0.75$. Moreover, in the continuum limit $d \to \infty$ of (\ref{LDS}) given by the partial differential equation
\begin{equation}\label{PDE}
	\partial_tU = \partial_{x}^2U - \mu U + 2U^3 - U^5, \quad U = U(x,t), \quad x \in \mathbb{R}, 
\end{equation}
it is known that the analogous steady-state solutions with flat plateaus do not snake and only exist at the Maxwell point $\mu = 0.75$. Therefore, the results of Fiedler and Scheurle \cite{Fiedler2} indicate that for all sufficiently large $d \gg 1$ in (\ref{LDS}) the snaking region of localized steady-states is exponentially small in $d$ and centred about the Maxwell point in $\mu$.

\begin{figure} 
	\centering
	\includegraphics[width=0.6\textwidth]{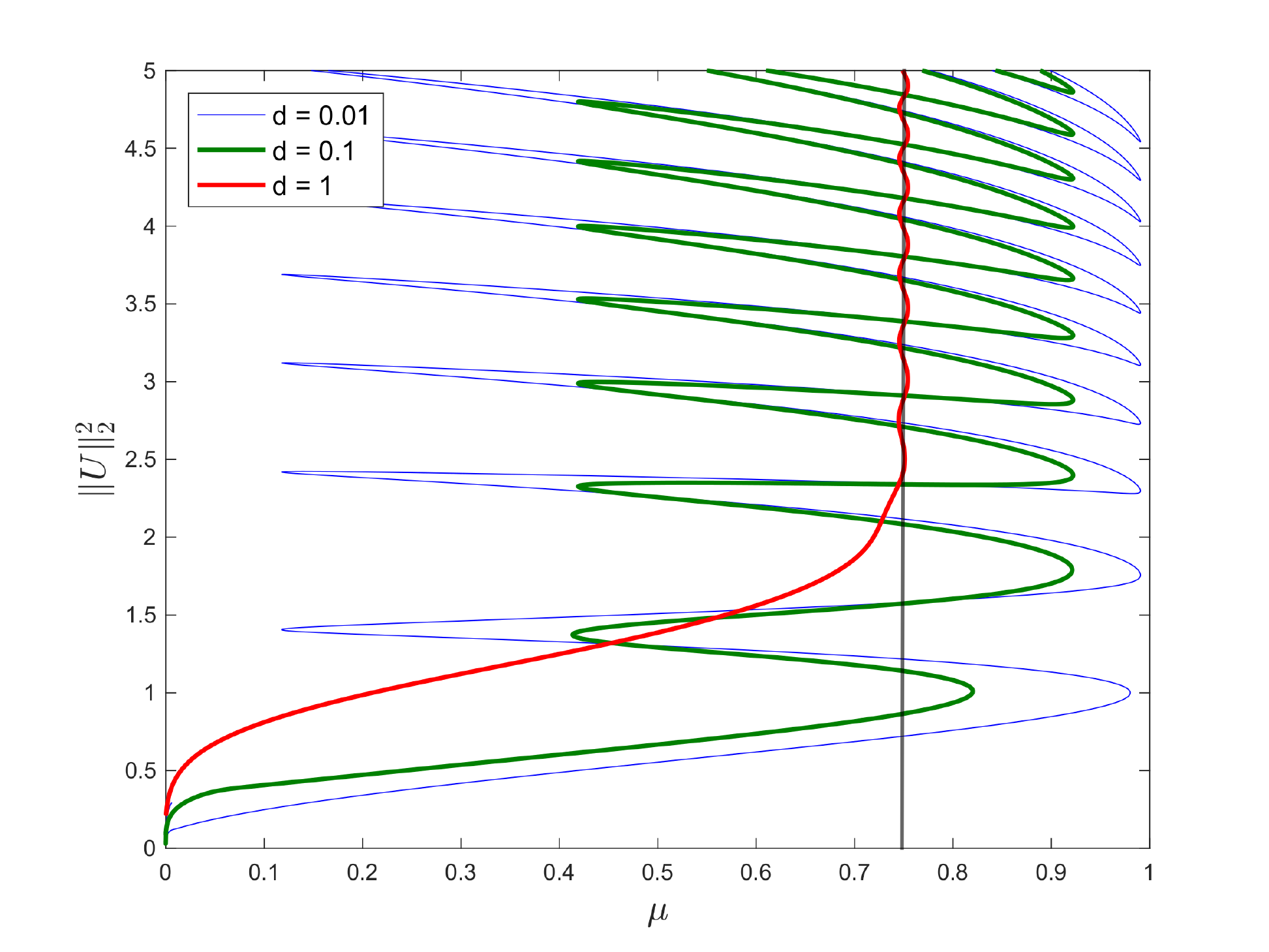}
	\caption{Shown are the bifurcation diagrams of on-site solutions of (\ref{LDS}) for $d = 0.001,0.1,$ and $1$. As $d$ increases, the snaking region narrows and appears to collapse onto the Maxwell point $\mu = \mu_\mathrm{mx} = 0.75$, represented by the vertical line. We note that off-site branches (not shown) have identical behaviour.}
\label{fig:IncreasingCoupling}
\end{figure} 

\paragraph{Patterns with oscillatory plateaus lie on isolas.}
We present numerical computations that indicate that patterns with oscillatory plateaus reside on isolas. These isolas appear to shrink as $d$ increases and disappear at a finite value of $d$. In particular, these patterns do not exist near the continuum limit $d\to\infty$.

Figure~\ref{fig:OscillatoryIsola} provides numerical confirmation of the results of Corollary~\ref{cor:0Loop} by plotting the resulting bifurcation diagram of a localized steady-state solution to (\ref{LDS}) with an oscillatory plateau. In particular, we have that the components of the plateau are very close to alternating between $\pm U_n(\mu-4d)$ and that the bifurcation diagram is composed of stacked isolas. Our numerical computations indicate that the width of these isolas shrinks very rapidly as $d$ increases: the isolas seem to disappear at around $d=0.11$ when they collapse onto themselves. The reason for this collapse is unknown and remains the subject of future work, but it should be noted that such states have no anologue in the continuum setting (\ref{PDE}) and therefore we would not expect that these bifurcation curves persist for all $d > 0$.    

\begin{figure} 
	\centering
	\includegraphics[width=\textwidth]{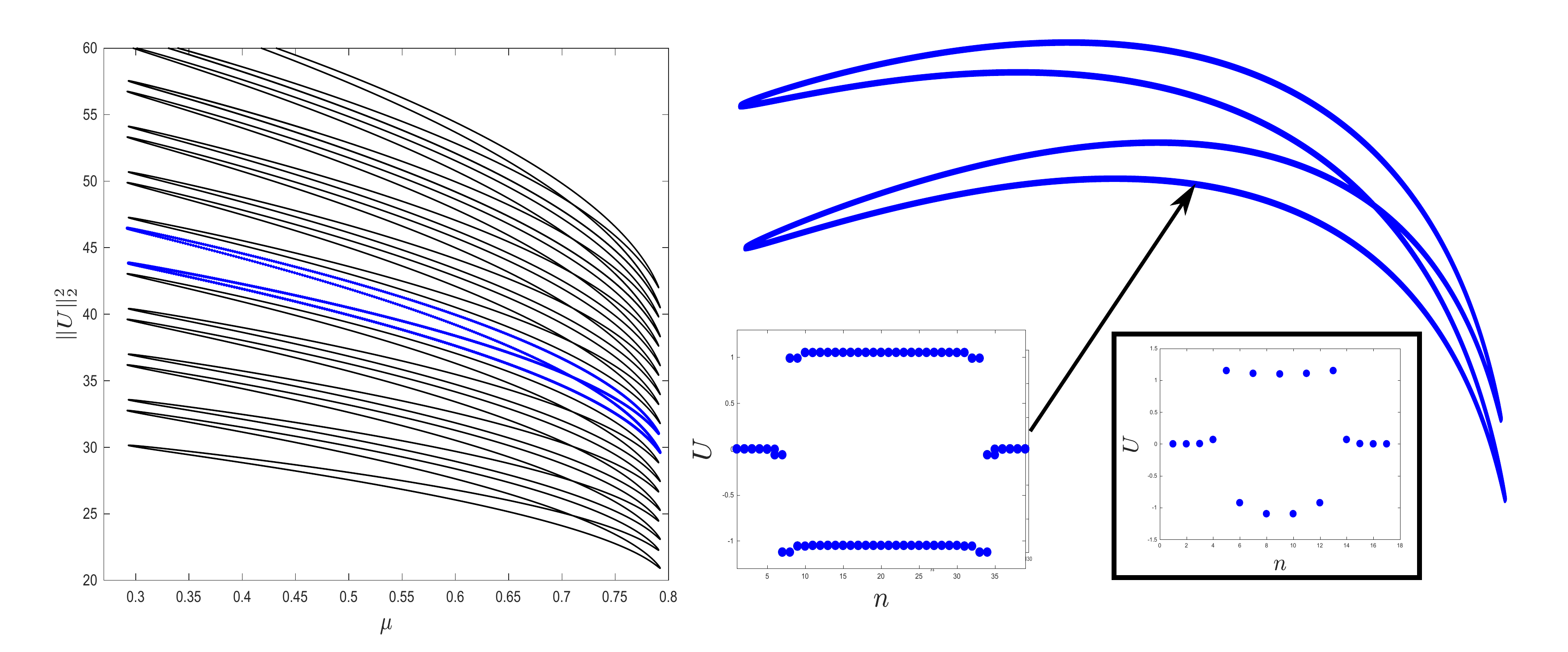} 
	\caption{Steady-states with oscillatory plateaus lead to isolas. The left panel contains a number of isolas from the bifurcation diagram at the parameter value $d = 0.05$ with one highlighted in blue for reference. The right panel contains the highlighted isola (stretched and rotated for visualization) along with a characteristic solution lying on the bifurcation curve. The bottom right inset provides a solution with a smaller oscillatory plateau from an isola lower down in the bifurcation diagram.}
\label{fig:OscillatoryIsola}
\end{figure}

As pointed out at the beginning of \S\ref{subsec:Oscillatory}, the results of this manuscript can be applied to understand the bifurcation structure of homoclinic orbits which spend a long time near a periodic orbit of any period. Although we have only focussed on periodic orbits with periods 1 and 2, numerical evidence leads one to believe that for small $d$ and appropriate $\mu$ the map (\ref{LDSMap}) exhibits periodic orbits of all periods. Unfortunately a staggering-type symmetry is not immediately apparent to be used to understand the bifurcation behaviour of heteroclinic orbits between the trivial fixed point and a periodic orbit of period $k \geq 3$, which therefore requires one to examine the $k$th iterate map of (\ref{LDSMap}) to understand the bifurcating homoclinic orbits which spend a long time near such a periodic orbit. Importantly thought, we do not expect any of these bifurcation diagrams to persist for all $d > 0$ since none of these states have an analogue in the continuum setting of (\ref{PDE}). 

Figure~\ref{fig:4Cycle} provides an example of a subset of the bifurcation diagram for localized steady-state solutions of (\ref{LDS}) for which their plateaus are approximately $4$-cycles. Here we see that as with the $2$-cycle case analyzed in \S\ref{subsec:Oscillatory}, these $4$-cycle localized states again lead to isolas. Interestingly, the bifurcation diagram appears to be populated by two different types of isolas, which are highlighted in Figure~\ref{fig:4Cycle} for reference. The approximate $4$-cycle along the plateau alternates from two positive states to two negative states. Numerical evidence indicates that these isolas only persist up to approximately $d = 0.21$.

\begin{figure} 
	\centering
	\includegraphics[width=\textwidth]{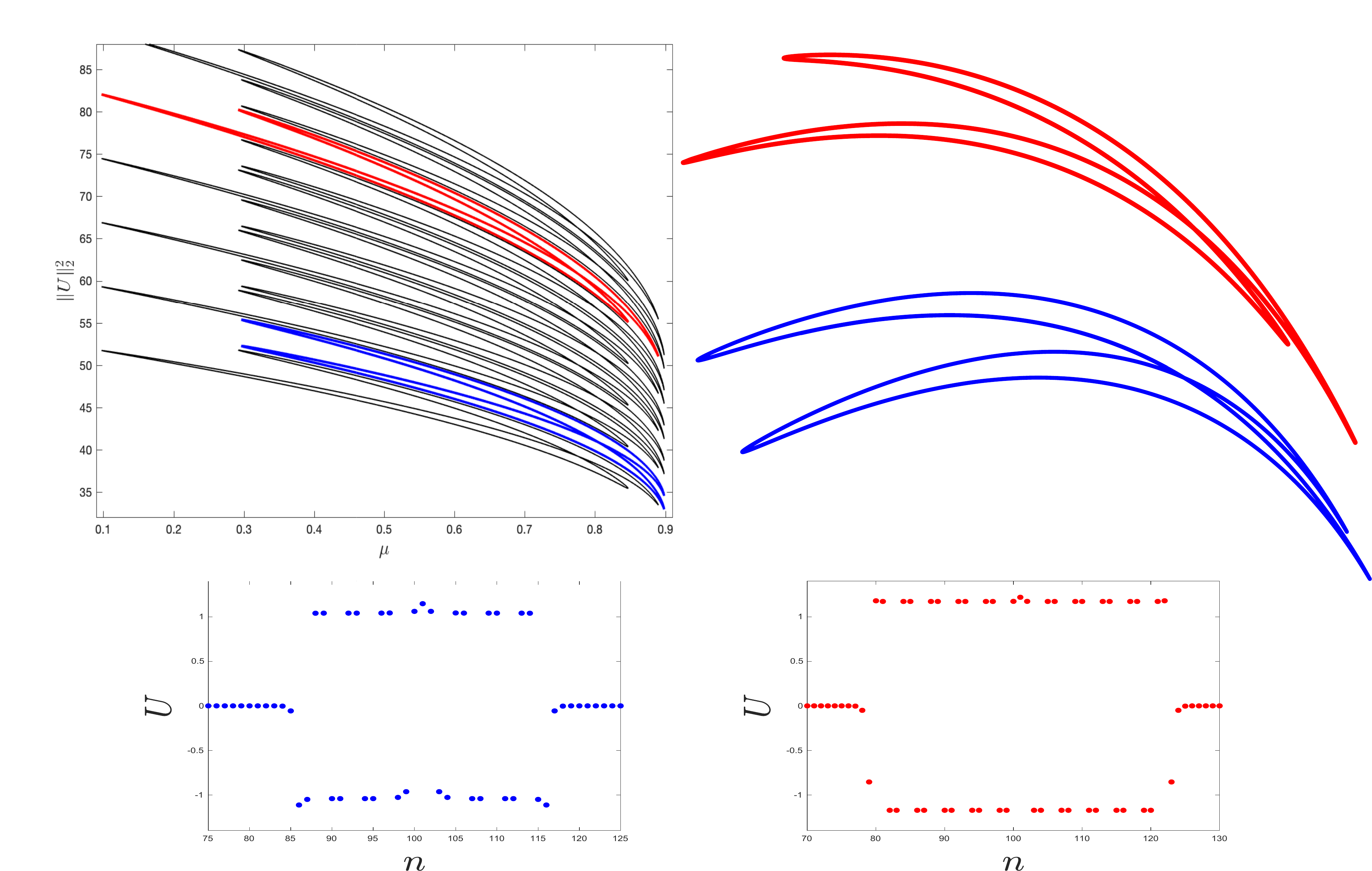} 
	\caption{$4$-cycle localized steady-state solutions of (\ref{LDS}). The left panel contains a number of isolas from the bifurcation diagram at the parameter value $d = 0.05$ with one highlighted in blue and another highlighted in red to demonstrate the two different types of isolas. These isolas are shown blown up and rotated on the right, where the figure eight structure is more apparent. At the bottom we provide sample profiles from each of the bifurcation curves, where we can see that the plateaus alternate from two positive to two negative values in an almost periodic manner.}
\label{fig:4Cycle}
\end{figure}

\paragraph{Asymmetric patterns.}
We present numerical computations of asymmetric patterns with flat and oscillatory plateaus. In \S\ref{sec:AsymHom} we extended the results of Theorem~\ref{thm:Pulses} to demonstrate the existence and bifurcation structure of asymmetric homoclinic orbits. In particular, we proved in Lemma~\ref{lem:AsymPitchfork} that near the saddle-node bifurcations on the curves of symmetric homoclinic orbits, a pitchfork bifurcation takes place which births a pair of asymmetric homoclinic orbits mapped into each other by the reverser. These bifurcating asymmetric solutions are show in Figure~\ref{fig:Snaking_1D} as green dotted curves which form the so-called ladder states. In Figure~\ref{fig:OscillatoryAsym} we demonstrate the existence of these branches of asymmetric solutions with oscillatory plateaus. We see that associated to each isola we have two distinct curves of asymmetric solutions whose curves originate and terminate near the saddle-node bifurcations on the curves of symmetric solutions with oscillatory plateaus.

\begin{figure} 
	\centering
	\includegraphics[width=0.7\textwidth]{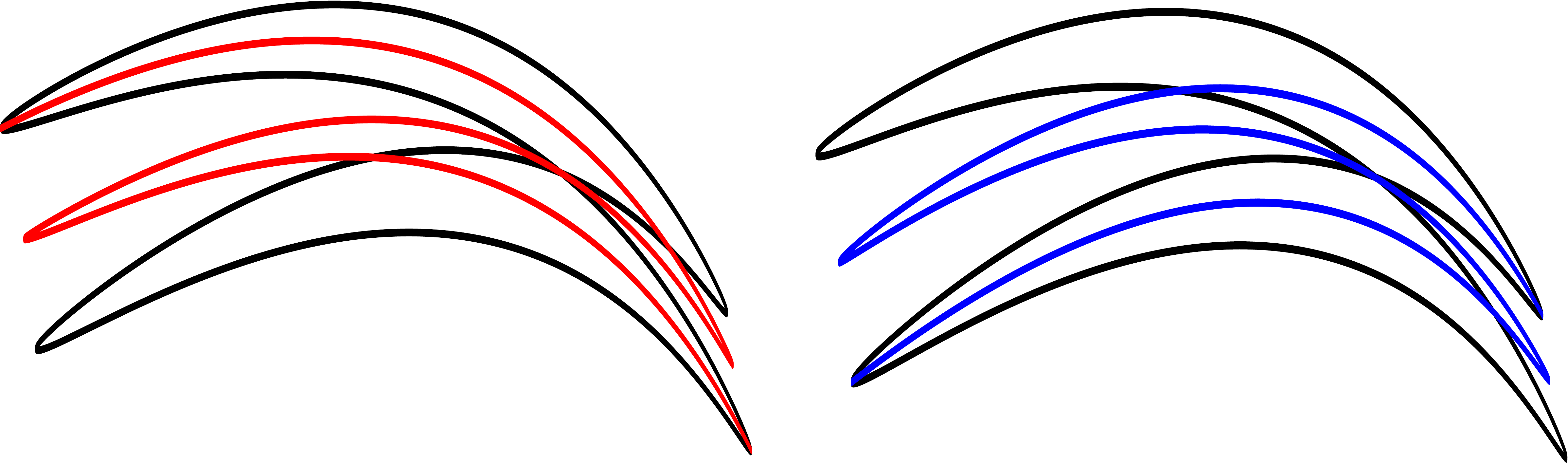} 
	\caption{Show are two copies of the same isola featured in Figure~\ref{fig:OscillatoryIsola} (stretched and rotated for visualization). In red and blue are two distinct curves of asymmetric solutions. Both asymmetric curves originate and terminate at pitchfork bifurcations with the symmetric solutions, exponentially close to the saddle-node bifurcations on the isolas.}
\label{fig:OscillatoryAsym}
\end{figure}


\subsection{Extension to Higher-Dimensional Maps}\label{subsec:Extensions} 

The analysis we presented in this manuscript was undertaken partially with the specific system (\ref{LDS}) in mind, but we note that it can be extended in a number of different ways to apply to more general lattice differential equations. This in turn could result in higher-dimensional mappings to analyze. For example, consider a fixed $N \geq 1$ and the equation of the form
\begin{equation} \label{LDS_Ext}
	\dot{U}_n = \sum_{i = 1}^N \bigg[d_i(U_{n+i} + U_{n-i} - 2U_n)\bigg] + f(U_n,\mu),\quad n\in\mathbb{Z},
\end{equation} 
with $U_n \in \mathbb{R}^k$ for each $n \in \mathbb{Z}$ and some integer $k \geq 1$, $d_i \in \mathbb{R}$ for all $i \in \{1,\dots,N\}$, and $f:\mathbb{R}^k\times\mathbb{R} \to \mathbb{R}^k$ is a smooth nonlinearity. Upon setting $\dot{U}_n = 0$ for all $n \in \mathbb{Z}$ and following the above procedure to obtain the analogous spatial mapping to (\ref{LDSMap}), we are left to consider a smooth diffeomorphism of the form $F:\mathbb{R}^{2Nk} \to \mathbb{R}^{2Nk}$. Furthermore, the symmetry of exchanging $U_{n+i}$ and $U_{n-i}$ in (\ref{LDS_Ext}) coming from the coupling terms $U_{n+i} + U_{n-i} - 2U_n$ will again endow the necessary reversible symmetry to formulate Hypothesis~\ref{hyp:Reverser}. 

It should be noted that in the context of system (\ref{LDS_Ext}) and its associated spatial mapping $F:\mathbb{R}^{2Nk} \to \mathbb{R}^{2Nk}$, the results of Lemmas~\ref{lem:Shilnikov} and \ref{lem:ShilSol} can be extended in a straightforward way. That is, we assume that $u^*$ is a fixed point of our mapping such that the linearization about $u^*$ is positive definite and has eigenvalues $\{\lambda(\mu),\lambda(\mu)^{-1}\}$ with $\lambda(\mu) > 1$ for all $\mu \in J$ and there exists $\rho > 0$ such that all other eigenvalues are either greater that $\lambda(\mu) + \rho$ or less than $1/(\lambda(\mu) + \rho)$ uniformly in $\mu \in J$. We can then choose local coordinates $(v^s,v^{ss},v^u,v^{uu})$ near $u^*$ that reflect the uniform spectral decomposition assumed above. In particular, $W^s(u^*,\mu)$ is given by $(v^s,v^{ss},0,0)$ in these coordinates and the reverser acts by $(v^s,v^{ss},v^u,v^{uu})\mapsto(v^u,v^{uu},v^s,v^{ss})$. Then we can find that the results of Lemma~\ref{lem:ShilSol} remain true in this situation with (\ref{ShilBnds}) replaced by 
\[
	|v^s_{n}| \leq M\eta^{n},\quad |v^{ss}_{n}| \leq M\eta^{n}, \quad |v^u_{n}| \leq M\eta^{N- n}, \quad |v^{uu}_{n}| \leq M\eta^{N- n}	
\] 
for some $M > 0$ and $\eta \in (0,1)$. We may then construct appropriate matching functions as in Lemma~\ref{lem:GammaLoc} and follow the results of this manuscript to satisfy the appropriate matching equations to obtain symmetric and asymmetric homoclinic orbits that correspond to localized steady-state solutions to system (\ref{LDS_Ext}). Therefore, the results of this manuscript remain valid for diffeomorphisms on $\mathbb{R}^{2Nk}$ as well.  


\section{Discussion}\label{sec:Discussion} 

In this paper, we analyzed symmetric and asymmetric localized patterns of lattice dynamical systems posed on the lattice $\mathbb{Z}$ --- our results extend previous analyses done on isolas and snaking diagrams for PDEs \cite{Aougab,Beck} to the spatially discrete setting. The main vehicle to obtain our results was an analysis of homoclinic orbits of two-dimensional reversible maps; we also discussed extensions to higher-dimensional maps. As in the continuous case, the key to understanding the global bifurcation structure of localized profiles on lattices lies in understanding the bifurcation structure of front solutions which manifest themselves as heteroclinic orbits to the associated spatial mapping. Finally, we applied our theoretical analysis to the discrete real Ginzburg--Landau equation posed on a one-dimensional lattice: our analysis focused on the regime near the anti-continuum limit where we rigorously predicted the full bifurcation structure of localized profiles with both flat and oscillatory plateaus.

As in the continuous case, our analysis relied on an underlying reversible structure that reflects a symmetric coupling between neighboring edges. More generally, we could consider steady-state solutions to a lattice dynamical system of the form 
\begin{equation} \label{LDS_Ext2}
	\dot{U}_n = d_1(U_{n+1} - U_n) + d_2(U_{n-i} - U_n) + f(U_n,\mu),\quad n\in\mathbb{Z}.
\end{equation}
When $d_1 \neq d_2$, the coupling is asymmetric, and it is easy to see that the resulting map that describes stationary states is not reversible in this case. In particular, the results of this manuscript do not apply directly to (\ref{LDS_Ext2}). We note, however, that Lemma~\ref{lem:GeneralLattice}, which describes persistence near the anti-continuum limit, can still be applied, and we therefore expect localized profiles of (\ref{LDS_Ext2}) with associated snaking and isola structures to exist for nonzero $|d_1|,|d_2|\ll1$.

\begin{figure} 
	\centering
	\includegraphics[width=\textwidth]{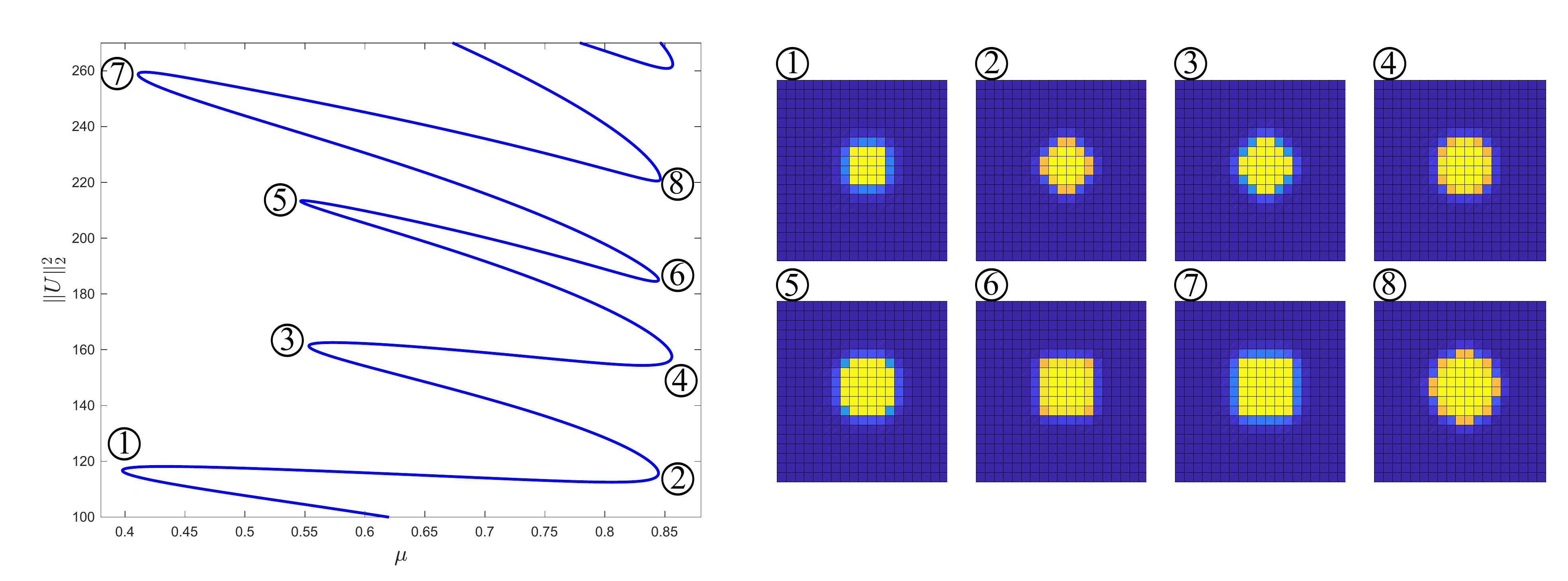} 
	\caption{Shown is part of the bifurcation diagram of Figure~\ref{fig:SquareSnake}[right] with saddle-node bifurcations labelled and sample profiles given. Moving from saddle node (1) to saddle node (7) shows how an additional ring emerges around the square pattern.}
\label{fig:Square_Saddles}
\end{figure} 

\begin{figure} 
	\centering
	\includegraphics[width=0.8\textwidth]{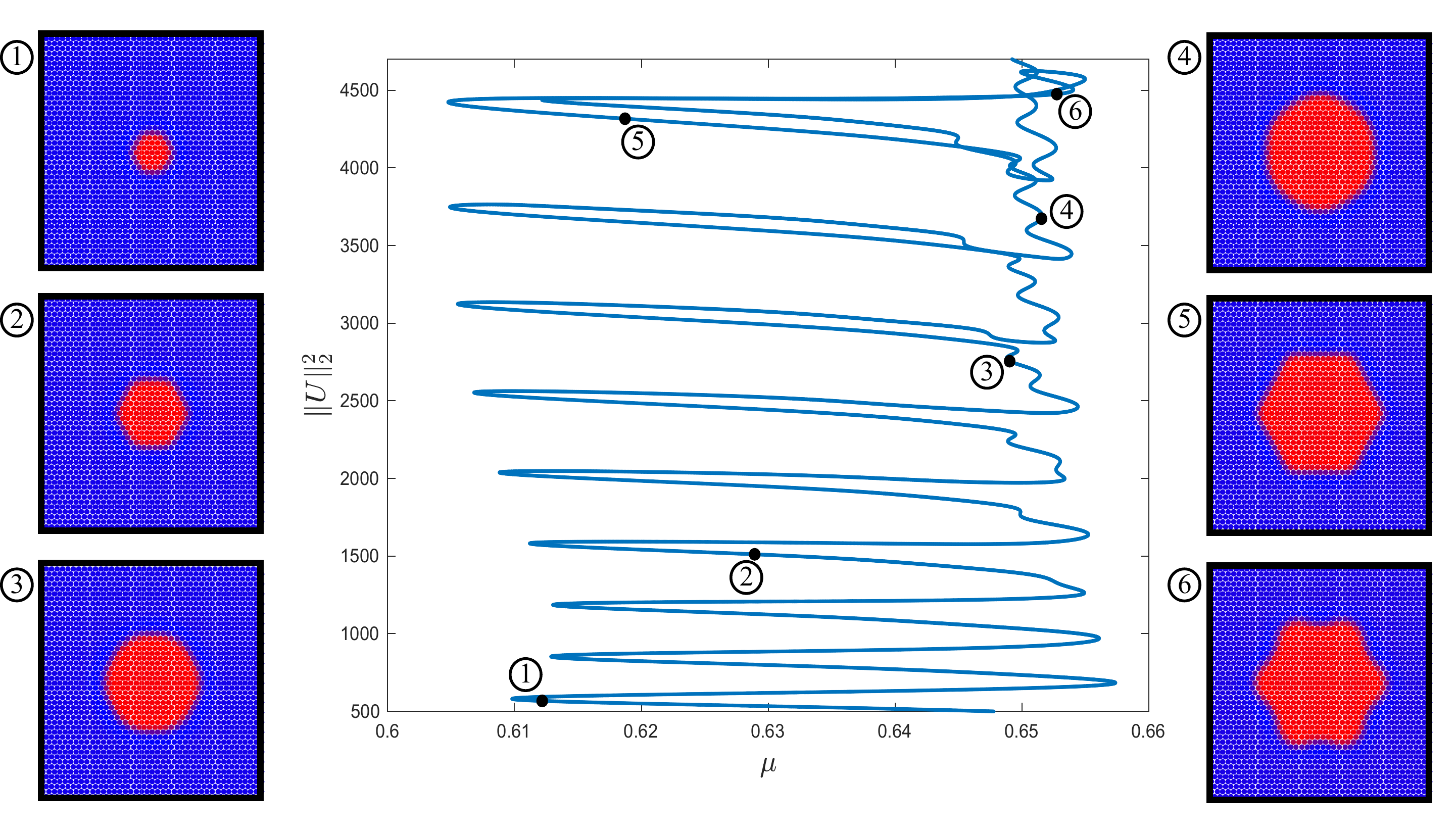} 
	\caption{Shown are the snaking branch and representative profiles of localized patterns in the discrete Swift--Hohenberg equation posed on an hexagonal lattice near the anti-continuum limit.}
\label{fig:Hex_Snaking}
\end{figure}

Finally, we briefly discuss the bifurcation structure of localized patterns on planar lattices. We saw in Figure~\ref{fig:SquareSnake} that the lattice system
\[
	\dot{U}_{n,m} = d(U_{n+1,m} + U_{n-1,m} + U_{n,m+1} + U_{n,m-1} - 4U_{n,m}) - \mu U_{n,m} + 2U_{n,m}^3 - U_{n,m}^5
\]
on the square $(n,m)\in\mathbb{Z}^2$ exhibits localized solutions that organize themselves in an intricate snaking structure. In particular, we find that some of the patterns resemble complete squares, while others take on more complicated shapes: see Figure~\ref{fig:Square_Saddles} for an additional illustration of this observation. The numerical computations summarized in Figure~\ref{fig:Hex_Snaking} show the discrete Swift--Hohenberg equation posed on a hexagonal lattice exhibits localized patterns that resemble large hexagonal patches which exist along snaking curves. As in the continuum case \cite{Lloyd}, the alignment of saddle nodes with multiple asymptotes that are visible in Figures~\ref{fig:SquareSnake} and~\ref{fig:Hex_Snaking} combined with the complicated shape of profiles along the bifurcation curve could reflect the selection of different fronts that become relevant as the patterns size grows, reflecting the importance of the interfacial energy between the bistable states that make up the patterns. Additional computations we carried out (which are not shown here) demonstrate that all saddle nodes approach $\mu=0,1$ in the anti-continuum limit as $d\to0$: this makes sense as there is no coupling, and therefore no interfacial energy, in this limit. We emphasize that Lemma~\ref{lem:GeneralLattice} still applies in this context: we therefore know that any localized stationary state at $d=0$ persists for $0<d\ll1$ for each fixed $0<\mu<1$. Thus, the key to understanding the complete bifurcation structure near this limit is to understand what happens near the bifurcations at $\mu=0,1$: as shown in \cite{Bramburger}, we can extend the techniques employed in \S\ref{sec:LDS} to determine the bifurcation structure of planar profiles of at least some of the saddle-node and pitchfork bifurcations near $\mu=0,1$ close to the anti-continuum limit. 

\paragraph{Acknowledgements.}
Bramburger was supported by an NSERC PDF. Sandstede was partially supported by the NSF through grant DMS-1714429.



\begin{thebibliography}{99}

\bibitem{Aougab}
T. Aougab, M. Beck, P. Carter, S. Desai, B. Sandstede, M. Stadt, and A. Wheeler.
Isolas versus snaking of localized rolls.
{\em J. Dyn. Differ. Eqns.} {\bf 31} (2019) 1199-1222.

\bibitem{Avitabile}
D. Avitabile, D.J.B. Lloyd, J. Burke, E. Knobloch, and B. Sandstede.
To snake or not to snake in the planar Swift--Hohenberg equation.
{\em SIAM J. Appl. Dynam. Syst.} {\bf 9} (2010) 704-733.

\bibitem{Beck}
M. Beck, J. Knobloch, D. Lloyd, B. Sandstede, and T. Wagenknecht.
Snakes, ladders, and isolas of localized patterns.
{\em SIAM J. Math. Anal.} {\bf 41} (2009) 936-972.

\bibitem{Beyn}
W.-J. Beyn and J.-M. Kleinkauf.
The numerical computation of homoclinic orbits for maps.
{\em SIAM J. Numer. Anal.} {\bf 34} (1997) 1207-1236.

\bibitem{Bramburger}
J.J. Bramburger and B. Sandstede.
Localized patterns in planar bistable lattice systems.
Nonlinearity, (20202) at press.

\bibitem{Burke}
J. Burke and E. Knobloch.
Localized states in the generalized Swift--Hohenberg equation.
{\em Phys. Rev. E} {\bf 73} (2006) 056211.

\bibitem{Burke2}
J. Burke and E. Knobloch.
Snakes and ladders: localized states in the Swift--Hohenberg equation.
{\em Phys. Rev. A} {\bf 360} (2007) 681-688.

\bibitem{Carretero}
R. Carretero-Gonz\'alez, J.D. Talley, C. Chong, and B.A. Malomed.
Multistable solitons in the cubic-quintic discrete nonlinear Schr\"odinger equation.
\textit{Physica~D} \textbf{216} (2006) 77-89.

\bibitem{kozyreff1} 
S.~J.~Chapman and G.~Kozyreff.
Exponential asymptotics of localised patterns and snaking bifurcation diagrams.
\textit{Physica~D} \textbf{238} (2009) 319--354.

\bibitem{Chong}
C. Chong, R. Carretero-Gonz\'alez, B.A. Malomed, and P.G. Kevrekidis.
Multistable solitons in higher-dimensional cubic-quintic nonlinear Schr\"odinger lattices.
{\em Physica~D} {\bf 238} (2009) 126-136.

\bibitem{Chong2}
C. Chong and D.E. Pelinovsky.
Variational approximations of bifurcations of asymmetric solitons in cubic-quintic nonlinear Schr\"odinger lattices.
\textit{Discrete Cont. Dyn. Syst. Ser.~S} \textbf{4} (2011) 1019-1031.

\bibitem{Schneider}
C. Chong, D.E. Pelinovsky, and G. Schneider.
On the validity of the variational approximation in discrete nonlinear Schr\"{o}dinger equations
\textit{Physica~D} \textbf{241} (2012) 115-124.

\bibitem{Coullet}
P.~Coullet, C.~Riera and C.~Tresser.
Stable static localized structures in one dimension.
\textit{Phys. Rev. Lett.} \textbf{84} (2000) 3069--3072.

\bibitem{Dawes}
J.H.P.~Dawes.
The emergence of a coherent structure for coherent structures: localized states in nonlinear systems.
\textit{Philos. Trans. R. Soc. Lond. Ser. A} \textbf{368} (2010) 3519--3534.

\bibitem{Ferrofluid}
M. Groves, D. Lloyd, and A. Stylianou.
Pattern formation on the free surface of a ferrofluid: spatial dynamics and homoclinic bifurcation.
{\em Physica~D} {\bf 350} (2017) 1-12.

\bibitem{Fiedler}
B. Fiedler.
Global pathfollowing of homoclinic orbits in two-parameter flows.
{\em Pitman Res.} {\bf 352} (1996) 79-146.

\bibitem{Fiedler2}
B. Fiedler and J. Scheurle.
Discretization of homoclinic orbits, rapid forcing and ``invisible'' chaos.
{\em Mem. Amer. Math. Soc.} {\bf 119} (1996).

\bibitem{Knobloch}
E.~Knobloch.
Spatial localization in dissipative systems.
\textit{Ann. Rev. Condens. Matter Phys.} \textbf{6} (2015) 325--359.

\bibitem{Wagenknecht}
J.~Knobloch, M.~Vielitz, and T.~Wagenknecht.
Non-reversible perturbations of homoclinic snaking scenarios.
{\em Nonlinearity} \textbf{25} (2012) 3469--3485.

\bibitem{kozyreff2}
G.~Kozyreff and S.~J.~Chapman.
Asymptotics of large bound states of localised structures.
\textit{Phys. Rev. Lett.} \textbf{97} (2006) 044502.

\bibitem{Kusdiantara}
R. Kusdiantara and H. Susanto.
Homoclinic snaking in the discrete Swift--Hohenberg equation.
{\em Phys. Rev. E} {\bf 96} (2017) 062214. 

\bibitem{Hotspot}
D. Lloyd and H. O'Farrell.
On localised hotspots of an urban crime model.
{\em Physica~D} {\bf 253} (2013) 23-39.

\bibitem{Lloyd}
D.J.B. Lloyd, B. Sandstede, D. Avitabile, A.R. Champneys.
Localized hexagon patters of the planar Swift--Hohenberg equation.
{\em SIAM J. Appl. Dynam. Syst.} {\bf 7} (2008) 1049-1100.

\bibitem{Makrides}
E. Makrides and B. Sandstede.
Predicting the bifurcation structure of localized snaking patterns.
{\em Physica~D} {\bf 253}, (2013) 23-39.

\bibitem{Makrides2}
E.~Makrides and B.~Sandstede.
Existence and stability of spatially localized patterns.
\textit{J. Differ. Eqns.} \textbf{266} (2019) 1073-1120.

\bibitem{McCullen}
N. McCullen and T. Wagenknecht.
Pattern formation on networks: From localized activity to Turing patterns.
\textit{Sci. Reports} \textbf{6} (2016) 27397.

\bibitem{Veg1}
E. Meron.
Pattern-formation approach to modelling spatially extended ecosystems.
{\em Ecol. Model.} {\bf 234} (2012) 70-82.

\bibitem{Palmer}
K.J. Palmer.
Existence of transversal homoclinic points in a degenerate case.
{\em Rocky Mt. J. Math.} {\bf 20} (1990) 1099-1118.

\bibitem{Papangelo}
A. Papangelo, A. Grolet, L. Salles, N. Hoffman, and M. Ciavarella.
Snaking bifurcations of self-excited oscillator chain with cyclic symmetry.
{\em Commun. Nonlin. Sci. Numer. Simul.} {\bf 44} (2006) 642-647.

\bibitem{Pelinovsky}
D.E. Pelinovsky.
\textit{Localization in periodic potentials}.
Cambridge University Press, Cambridge (2011).

\bibitem{Pomeau}
Y.~Pomeau.
Front motion, metastability, and subcritical bifurcations in hydrodynamics.
\textit{Physica~D} \textbf{130}, (1999) 73--104.

\bibitem{Xu}
B.~Sandstede and Y.~Xu.
Snakes and isolas in non-reversible conservative systems.
\textit{Dyn. Syst.} \textbf{27} (2012) 317--329.

\bibitem{Schecter}
S. Schecter.
Exchange lemmas 1: Deng's lemma.
{\em J. Differ. Eqns.} {\bf 245} (2008) 392-410.

\bibitem{Veg2}
E. Sheffer, H. Yizhaq, M. Shachak, and E. Meron.
Mechanisms of vegetation-ring formation in water-limited systems.
{\em J. Theor. Bio.} {\bf 273} (2011) 138-146.

\bibitem{Semiconductor}
V.B. Taranenko, I. Ganne, R.J. Kuszelewicz, and C.O. Weiss.
Patters and localized structures in bistable semiconductor resonators.
{\em Phys. Rev. A} {\bf 61} (2000) 063818.

\bibitem{Taylor}
C. Taylor and J.H.P. Dawes.
Snaking and isolas of localised states in bistable discrete lattices.
{\em Phys. Lett. A} {\bf 375} (2010) 14-22.

\bibitem{Hotspot2}
W.~H.~Tse and M.~J.~Ward.
Hotspot formation and dynamics for a continuum model of urban crime.
\textit{Eur. J. Appl. Math.} {\bf 27} (2015) 583--624.

\bibitem{Wiggins}
S. Wiggins.
{\em Introduction to applied nonlinear dynamical systems and chaos}.
Springer-Verlag, New York (2003).

\bibitem{Chemical}
V.~K.~Vanag, A.~M.~Zhabotinksky, and I.~R.~Epstein.
Pattern formation in the Belousov-Zhabotinksky reaction with photochemical global feedback.
\textit{J. Phys. Chem.~A} {\bf 104} (2000) 11566--11577.

\bibitem{Woods}
P.~D.~Woods and A.~R.~Champneys.
Heteroclinic tangles and homoclinic snaking in the unfolding of a degenerate reversible Hamiltonian Hopf bifurcation.
\textit{Physica~D} \textbf{129} (1999) 147--170.

\bibitem{Yulin}
A.V. Yulin and A.R. Champneys.
Discrete snaking: Multiple cavity solitons in saturable media.
\textit{SIAM J. Appl. Dynam. Syst.} \textbf{9} (2010) 391-431.

\end{thebibliography}
\end{document}